\documentclass[reqno,fleqn]{amsart}
\usepackage{amsfonts,amsthm,amssymb}
\usepackage{lipsum}
\usepackage[utf8]{inputenc}
\usepackage[english]{babel}

\usepackage{lmodern,textcomp}
\usepackage{amsmath}
\usepackage{mathtools}
\usepackage{latexsym}
\usepackage{tikz}
\usepackage{fancyvrb}
\usepackage{epsfig}
\usepackage{pstricks,slashbox,multirow}
\usepackage{graphicx}
\usepackage{rotating}
\usetikzlibrary{positioning}
\usepackage{subcaption}
\usepackage{datetime}
\newtheorem{theorem}{Theorem}[section]

\newtheorem{lemma}[theorem]{Lemma}

\newtheorem{cor}[theorem]{Corollary}
\newtheorem{definition}[theorem]{Definition}
\newtheorem{con}{Conjecture}
\newtheorem{exm}{Example}
\newtheorem{prm}[theorem]{Problem}
\newtheorem{oprm}{Open Problem}

\newtheorem{rem}[theorem]{Remark}

\newtheorem{note}[theorem]{Note}

\title[A study on Type-2 isomorphic circulant graphs and related groups]{A study on Type-2 isomorphic circulant graphs and related Abelian groups}
\author{\sc Vilfred Kamalappan} 
\address{Department of Mathematics, Central University of Kerala, Periye, Kasaragod, Kerala, India - 671 316.}
\email{vilfredkamal@gmail.com}

\subjclass[2010]{05C60, 05C25, 05C75.}
\keywords{Circulant graph, Cayley Isomorphism (CI) property, Type-1 isomorphism, Type-2 isomorphism, Type-1 group of $C_{n}(R)$, Type-2 group of $C_{n}(R)$ w.r.t. $r$, $(V_{n,r}(C_n(R)), ~\circ)$.}
\date{}

\begin{document}

\begin{abstract} Circulant graphs $C_n(R)$ and $C_n(S)$ are said to be \emph{Adam's isomorphic} if there exist some $a\in \mathbb{Z}_n^*$ such that $S = a R$ under arithmetic reflexive modulo $n$. In 1970, Elspas and Turner \cite{eltu} raised a question on the isomorphism of $C_{16}(1, 3, 7)$ and $C_{16}(2, 3, 5)$ and Vilfred \cite{v96} gave its answer by defining Type-2 isomorphism, different from Adam's isomorphism or Type-1 isomorphism, of $C_n(R)$ w.r.t. $m$ where $m > 1$ is a divisor of $\gcd(n, r)$ and $r\in R$. This paper is an extensive study on Type-2 isomorphic circulant graphs. Vilfred and Wilson \cite{vw0A} obtain isomorphic circulant graphs $C_{np^3}(R)$ of Type-2 w.r.t. $m$ = $p$, and related Abelian groups where $p$ is a prime number and $n\in\mathbb{N}$. Using Theorem \ref{c13}, a list of $T2_{np^3,p}(C_{np^3}(R^{np^3,x+yp}_i))$ =  $\{C_{np^3}(R^{np^3,x+yp}_{j}) : j = 1,2,...,p\}$ for $p$ = 3,5,7,11 and $n$ = 1 to 5 and also for $p$ = 13 and $n$ = 1 to 3 are given in the Annexure where $(T2_{np^3,p}(C_{np^3}(R^{np^3,x+yp}_i)), \circ)$ is an abelian group on the $p$ isomorphic circulant graphs $C_{np^3}(R^{np^3,x+yp}_i)$ of Type-2 w.r.t. $m$ = $p$, $1 \leq i,j \leq p$, $1 \leq x \leq p-1$, $y\in\mathbb{N}_0$, $0 \leq y \leq np - 1$, $1 \leq x+yp \leq np^2-1$, $p,np^3-p\in R^{np^3,x+yp}_i$ and $i,j,n,x\in\mathbb{N}$. We also show existence of isomorphic circulant graphs $C_n(R)$ and $C_n(S)$ which are neither Type-1 nor Type-2 w.r.t. any particular $m$. We use VB program to develope this theory
and for illustration of examples.
\end{abstract}

\maketitle

	
\section{Introduction}

\subsection{Historical Note}

In 1846 Catalan (cf. \cite{da79}) introduced circulant matrices and thereafter many authors studied different properties of circulant graphs \cite{ad67}-\cite{frs}, \cite{hz14}-\cite{vw3}. Adam \cite{ad67} considered conditions for isomorphism of circulant graphs;  Boesch and Tindell \cite{bt} studied connectivity of circulant graphs; Sachs \cite{sa62} studied self-complementary circulant graphs and a conjecture stated by him on the existence of self-complementary circulant graphs was settled in \cite{amv, frs} (Also, see Theorem 3.7.5, pages 55, 56 in \cite{v96}.); Vilfred \cite{v13} developed a theory of Cartesian product and factorization of circulant graphs similar to that of natural numbers; a question on the isomorphism of $C_{16}(1,3,7)$ and $C_{16}(2,3,5)$ was raised by Elspas and Turner in \cite{eltu} and Vilfred \cite{v20} gave its answer by defining Type-2 isomorphism; Li, Morris, Muzychuk, Palfy and Toida studied automorphism of circulant graphs \cite{li02}-\cite{mu97}, \cite{to77}. For further studies on circulant graphs, one can refer the book on circulant matrices by Davis \cite{da79} and the survey article by Kra and Simanca \cite{krsi}. 

Vilfred  \cite{v96, v17, v20} defined and studied circulant graph isomorphism of Type-2, a new type of isomorphism different from Adam's  isomorphism or Type-1 isomorphism, of $C_n(R)$ w.r.t. $r$ $\ni$ $\gcd(n, r)$ = $ m >  1$, $r\in R$ and $r,n\in\mathbb{N}$. Type-2 isomorphic circulant graphs don't have  CI-property. Such graphs of order $n$ w.r.t.  $r$ $\ni$ $\gcd(n, r)$ = 2,3,5,7 are obtained by the authors in \cite{v20}, \cite{vw1} - \cite{vw3}, $n\in\mathbb{N}$. Vilfred \cite{v20} extends the definition of Type-2 isomorphism of circulant graphs $C_n(R)$ w.r.t. $m$ by considering $m > 1$ is a divisor of $\gcd(n, r)$ and $r\in R$. This paper is an extensive study on Type-2 isomorphic circulant graphs. We obtain isomorphic circulant graphs $C_{np^3}(R)$ of Type-2 w.r.t. $m$ = $p$, and abelian groups related to these isomorphic circulant graphs where $p$ is a prime number, $r\in R$, $\gcd(n, r)$ = $m$ and $n\in\mathbb{N}$.

Investigation of symmetries as well as asymmetries of structures yields powerful results in Mathematics. Circulant graphs form a class of highly symmetric mathematical (graphical) structures and is one of the most studied graph classes. Circulant graphs have wide range of applications in coding theory, Ramsey theory, VLSI design, interconnection networks in parallel and distributed computing \cite{hz14}, the modeling of data connection networks and the theory of designs and error-correcting codes \cite{h03,st02}. 

\subsection{Preliminaries}

Let $n$ be a positive integer, and $\mathbb{Z}_n$ the additive group of residues modulo $n$. For a subset $S$ of $\mathbb{Z}_n$ one can define the circulant graph $C_n(S)$ as the graph with vertex set $\mathbb{Z}_n$ and edge set $\{\{x, y\} | x,y\in \mathbb{Z}_n,$ $x-y\in S\}$. For an integer $k$ coprime to $n$, there is a very natural isomorphism between circulants $C_n(S)$ and $C_n(kS)$  namely the map $\varphi_k : x \mapsto kx$ (observe that $\varphi_k$ is an automorphism of the group $\mathbb{Z}_n$) where $kS$ = $\{ks : s\in S\}$. See Lemma \ref{a4}. In this case, we say that $C_n(S)$ and $C_n(kS)$ are \emph{Adam's isomorphic} and we call the isomorphism as \emph{Type-1 isomorphism} \cite{v17}.  

The adjacency matrix $A(G)$ of a circulant graph $G$ is circulant. Thus, if $[a_1,a_2,\ldots,$ $a_n]$ is the first row of the adjacency matrix $A(G)$ of a circulant graph $G$, then $a_1$ = $0$ and for $2 \leq i \leq n$, $a_i$ = $a_{n-i+2}$ \cite{da79}. Through-out this paper, $C_n(R)$ denotes circulant graph $C_n(r_1, r_2, . . . , r_k)$ for a set $R$ = $\{ r_1, r_2, ~. ~.~ .~ , r_k \}$ where $1 \leq$ $r_1 < r_2 < . . . <$ $r_k \leq [\frac{n}{2}]$.  And we consider only connected circulant graphs of finite order, $V(C_n(R))$ = $\{v_0, v_1, v_2, . . . , v_{n-1}\}$ with $v_{i+r}$ adjacent to $v_{i}$ for each $r \in R$ and $0 \leq$ $i \leq n-1$, subscript addition taken modulo $n$, $K_n = C_n(1,2,\ldots,n-1)$ and all cycles have length at least $3,$ unless otherwise specified. However when $\frac{n}{2} \in R,$ edge $v_iv_{i+\frac{n}{2}}$ is taken as a single edge for considering the degree of the vertex $v_i$ or $v_{i+\frac{n}{2}}$ and as a double edge while counting the number of edges or cycles in $C_n(R),$ $0 \leq i \leq n-1.$ 

We generally write $C_n$ for $C_n(1).$ We will often assume that the vertices of circulant graph $C_n(R)$, with-out further comment, are the corners of a regular $n$-polygon, labeled clockwise. In $C_n(R)$ when $r\in R$, it is understood that $r\in\mathbb{Z}_{\frac{n}{2}}$. Vertex $v_i$ in each figure is considered with label $i$, $v_i\in V(C_n(R))$ and $i\in\mathbb{Z}_n$. Isomorphic circulant graphs $C_{54}(2,3,16,20)$, $C_{54}(3,4,14,22)$ and $C_{54}(3,8,10,26)$ which are of Type-2 isomorphic w.r.t.  $m$ = 3 and are shown in Figures 1, 2 and 3. See Example \ref{e1} in Section 7 for their Type-2 isomorphism. 
\begin{figure}[ht]
	\centerline{\includegraphics[width=6.3in]{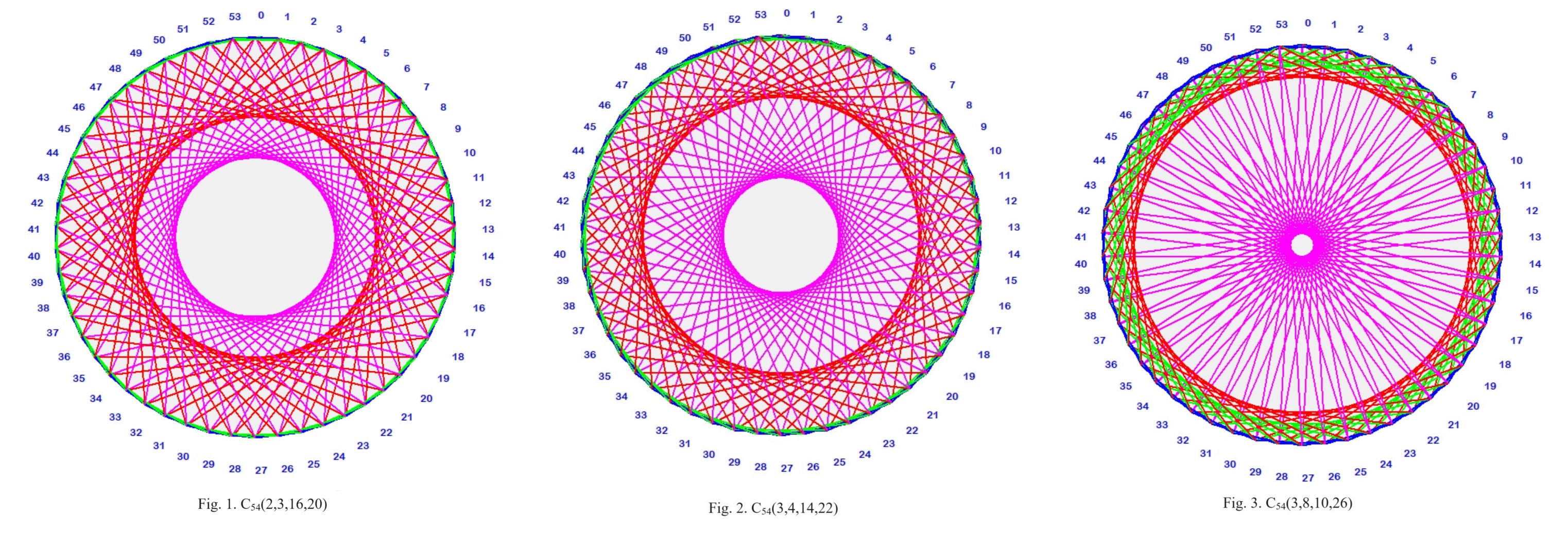}}
\end{figure}
Many symbols are used in this paper and we present important symbols in the following Table 1 to make it easier to the readers. 

\begin{table}\label{t1}
	\caption{List of important symbols used in this paper}
	\begin{center}
		\scalebox{.9}{
	\begin{tabular}{||c||*{4}{c||}}\hline \hline 
				S. No. &  Symbol & Meaning 
				\\ \hline \hline 
				& & \\
				
	1. & $C_n(1)$ & = $C_n$ \\
	2. & $C_n(1,2,...,[n/2])$ & = $K_n$ \\
				
	3. & $C_n(R)$ & Circulant graph of order $n$ with jump sizes \\
		    & & $r_1,r_2,...,r_k$, $R = \{r_1,r_2,...,r_k\} \subseteq \mathbb{Z}_{\frac{n}{2}}$ \\
			& & and $1 \leq r_1 < r_2 < ... < r_k \leq [n/2]$. \\
				
	4. & $\mathbb{Z}_n$ & $= \{0,1,2,...,n-1\}$ \\
				
	5. & $\mathbb{Z}^*_n$ & $ = \{1,2,...,n-1\}$ = $\mathbb{Z}_n \setminus \{0\}$\\
				& & \\
				
	6. & $\varphi_n$ & = $\{ x \in \mathbb{Z}_n : \gcd(n, x) = 1 \}$.\\
				
	7. & $\varphi_{n, x}:$ $S$ $\rightarrow$ $\mathbb{Z}_n :$ &   $\varphi_{n,x}(s)$ = $xs$, $\forall$ $s \in S$, $S \subseteq \mathbb{Z}_n$,  $S \neq \emptyset$ and $x \in \varphi_{n}$.\\ 
				& & \\
				
	8. & $\varphi_{n,x}(C_n(R))$  = $C_n(\varphi_{n,x}(R))$ &  = Adam's isomorphism of $C_n(R)$ w.r.t. $x$ \\ 
				& & =  $C_n(xR)$ where $x\in\varphi_n$ and $\varphi_{n,x}(R)$ in  \\
				& & $C_n(\varphi_{n,x}(R))$ is calculated  under \\
				& &  reflexive modulo $n$. \\
				
	9. & $Ad_n$ & $= \{\varphi_{n,x}/~ x\in\varphi_n\}$. \\
				
	10. & $\varphi_{n,x} \circ \varphi_{n,y}$ & = $\varphi_{n,xy}$, $x,y\in\varphi_n$. \\	
			& & \\
				
	11. & $Ad_n(C_n(R))$ & $= \{\varphi_{n,x}(C_n(R)) = C_n(xR)/ x\in\varphi_n\}$. \\
							
	12. & $(\varphi_{n,x} \circ \varphi_{n,y})(C_n(R))$ & = $\varphi_{n,xy}(C_n(R))$ = $C_n((xy)R)$ \\	
			& & \hfill = $\varphi_{n,x}(C_n(R)) \circ \varphi_{n,y}(C_n(R))$, $x,y\in\varphi_n$. \\	
				
	13. & $C_n(xR) \circ C_n(yR)$ & = $C_n((xy)R)$ = $\varphi_{n,xy}(C_n(R))$, $x,y\in\varphi_n$. \\	
			& & \\
				
    14. & $(Ad_n(C_n(R)),~\circ)$ & $= (T1_n(C_n(R)),~\circ)$ = Adam's group or \\ 
          & &   Type-1 group of $C_n(R)$. \\

	15. & $\theta_{n,m,t}:$ $\mathbb{Z}_n$ $\rightarrow$ $\mathbb{Z}_n$ $\ni$ &   $\theta_{n,m,t}(x) =  x+jtm$ where $x = qm+j$,   \\
			& & $m > 1$ is a divisor of $n$, $0 \leq j \leq m-1$,  \\
		& &  $m,x\in\mathbb{Z}_n$ and  $0 \leq q,t \leq \frac{n}{m}-1$. \\
				
	16. & $\theta_{n,m,t} \circ \theta_{n,m,t'}$ & $= \theta_{n,m,t+t'}$ where $0 \leq t,t' \leq \frac{n}{m}-1$, $m > 1$ is a divisor \\
			& &  of $n$ and $t+t'$ is calculated under arithmetic $mod~\frac{n}{m}$. \\
				
	17. & $\theta_{n,m,t}:$ $V(C_n(R))$ $\rightarrow$ $V(K_n)$ $\ni$ & $\theta_{n,m,t}(v_x) = u_{x+jtm}$, $\forall$ $v_x\in V(C_n(R))$, \\
			& & $\theta_{n,m,t}((v_x, v_y))$ = $ (\theta_{n,m,t}(v_x), \theta_{n,r,t}(v_y))$ and \\ 
			&  & $\theta_{n,m,t}(C_n(R)) = C_n(\theta_{n,m,t}(R))$ where $\theta_{n,m,t}(R)$ in \\
			& &   $C_n(\theta_{n,m,t}(R))$  is calculated under reflexive modulo $n$, \\
			& &  $x = qm+j$, $m > 1$ is a divisor of  $\gcd(n,r)$, $r\in R$,  \\
			& &  $V(C_n(R))$ = $\{v_0, v_1, v_2, . . . , v_{n-1}\}$, \\ 
			& & $V(K_n)$ =  $\{u_0, u_1, u_2, . . . , u_{n-1}\}$, $(v_x, v_y)\in E(C_n(R))$, \\
			& & \hfill $0 \leq q,t \leq \frac{n}{m}-1$ and $0 \leq j \leq m-1$. \\
			& & \\
				
	18. & $V_{n,r}$ & $= \{\theta_{n,m,t}/ t = 0,1,2,...,\frac{n}{m}-1\}$, $m > 1$ is a divisor of $n$. \\
				
	19. & $V_{n,r}(s)$ & $= \{\theta_{n,m,t}(s)/ t = 0,1,2,...,\frac{n}{m}-1\}$, $r,s\in\mathbb{Z}_n$ and \\
			& & \hfill $m > 1$ is a divisor of  $n$.  \\
				
	20. & $V_{n,r}(C_n(R))$ & $= \{\theta_{n,m,t}(C_n(R))/ t = 0,1,2,...,\frac{n}{m}-1\}$, \\
			& & \hfill $m > 1$ is a divisor of  $\gcd(n,r)$ and $r\in R$. \\
				
	21. &  $C_n(R)$ and $C_n(S)$ are Type-2 &  if $\theta_{n,m,t}(C_n(R)) = C_n(S)$ for some $t$, \\
			& isomorphic w.r.t. $m$ &   $m > 1$ is a divisor of  $\gcd(n,r)$, $r\in R$ and \\
				
			& &   $C_n(S) \neq C_n(xR)$, $\forall$ $x\in\varphi_n$, $1 \leq t \leq \frac{n}{m}-1$. \\
			& & \\
				
	22. & $T2_{n,r}(C_n(R))$ & $= \{C_n(R)\} \cup \{C_n(S): C_n(S)$ and $C_n(R)$ are isomorphic of \\
			& &  Type-2  w.r.t. $r$, $m > 1$ is a divisor of  $\gcd(n,r)$ and $r\in R$.  \\ 
				
	23. & $(T2_{n,r}(C_n(R)),~\circ)$ & = Type-2 group of $C_n(R)$ w.r.t. $r$, \\
			& & \hfill $m > 1$ is a divisor of  $\gcd(n,r)$ and $r\in R$.\\ 
			& & \\ 	\hline \hline
	\end{tabular}}
	\end{center}
\end{table}

We present here a few definitions and results that are required in the subsequent sections.

\begin{theorem}{\rm\cite{v96}  \quad \label{a1}  In $C_n(R)$, the length of a cycle of period $r$ is $\frac{n}{\gcd(n,r)}$ and the number of disjoint periodic cycles of period $r$ is $\gcd(n,r)$, $r \in R$. \hfill $\Box$}
\end{theorem}

\begin{cor}{\rm \cite{v96} \quad \label{a2}  In $C_n(R)$, length of a cycle of period $r$ is $n$ if and only if $\gcd(n,r)$ = $1,$ $r \in R$. \hfill $\Box$}
\end{cor}

\begin{theorem}{\rm \cite{v17} \quad \label{a3} If $C_n(R)$ and $C_n(S)$ are isomorphic, then there exists a bijection $f$ from $R \to S$ such that $\gcd(n, r)$ = $\gcd(n, f(r))$ for all $r\in R$.  }
\end{theorem}
\begin{proof} The proof is by induction on the order of $R$. 
\end{proof}

\subsection{ Structure of the Paper}

This paper contains 11 sections and an annexure. In Section 2, we present Adam's isomorphism or Type-1 isomorphism of a circulant graph $C_n(R)$ and we define Type-1 group of $C_n(R)$. Section 3 contains our study on Type-2 isomorphism and Type-2 isomorphic circulant graphs and show that the pairs $C_{16}(1, 2, 7)$ and $C_{16}(2, 3, 5)$; $C_{48}(1, 2, 23)$ and $C_{48}(2, 11, 13)$; $C_{48}(1, 6, 23)$ and $C_{48}(6, 11, 13)$; $C_{96}(1, 2, 47)$ and $C_{96}(2, 23, 25)$; $C_{96}(1, 6, 47)$ and $C_{96}(6, 23, 25)$ are Type-2 isomorphic w.r.t. $m$ = 2. In Section 4, all the 8 pairs of circulant graphs of order 16 and all the 12 triples of circulant graphs of order 27 are presented. Section 5 contains the definition of the symmetric equidistance condition w.r.t. a vertex and  results on Type-2 isomorphic circulant graphs w.r.t. $r$ = 2. In Section 6, we define $V_{n,r}(C_n(R))$ and $T2_{n,r}(C_n(R))$ and also Type-1 and Type-2 groups of $C_n(R)$ and prove that $(V_{n,r}(C_n(R)), \circ)$ and $(T2_{n,r}(C_n(R)), \circ)$ are Abelian groups and $T2_{n,r}(C_n(R)) \subseteq V_{n,r}(C_n(R))$. Section 7 contains examples of Type-1 and Type-2 groups. Section 8 shows $C_{54}(1, 3, 17, 19) \cong C_{54}(5, 13, 21, 23)$ but are neither of Type-1 isomorphic nor of Type-2 isomorphic w.r.t. 3 or 21. Section 9 presents  results on Type-2 isomorphic circulant graphs w.r.t. $r$ = 3, 5, 7. In Section 10, we obtain  families of Type-2 isomorphic circulant graphs $C_{np^3}(R)$ w.r.t. $r$ = $p$ and related Abelian groups. Section 11 contains conclusion. Theorems \ref{c10} and \ref{c13} are the main results of the paper. In the Annexure, using Theorem \ref{c13}, a list of $T2_{np^3,p}(C_{np^3}(R^{np^3,x+yp}_i))$ = $\{\Theta_{np^3,p,jn}(C_{np^3}(R^{np^3,x+yp}_i)): j = 0,1,...,p-1\}$ = $\{C_{np^3}(R^{np^3,x+yp}_{i+j}) : j = 0,1,...,p-1$ and $i+j$ is calculated under addition modulo $p\}$ for $p$ = 3,5,7,11 and $n$ = 1 to 5 and also for $p$ = 13 and $n$ = 1 to 3 and are given where $(T2_{np^3,p}(C_{np^3}(R^{np^3,x+yp}_i)), \circ)$ is an abelian group on the $p$ isomorphic circulant graphs $C_{np^3}(R^{np^3,x+yp}_j)$ of Type-2 w.r.t. $r$ $\ni$ $\gcd(n, r)$  = $p$, $1 \leq i,j \leq p$, $1 \leq x \leq p-1$, $y\in\mathbb{N}_0$, $0 \leq y \leq np - 1$, $1 \leq x+yp \leq np^2-1$, $p,np^3-p\in R^{np^3,x+yp}_i$ and $i,j,n,x\in\mathbb{N}$.
 We follow Remark \ref{r11} to establish Type-2 isomorphism w.r.t. $m$ between circulant graphs $C_n(R)$ and $C_n(S)$.   

Effort to understand the isomorphism, pointed out by Elspas and Turner \cite{eltu}, that exists between $C_{16}(1,2,7)$ and $C_{16}(2,3,5)$ and to develop its general theory are the motivation to do this work. For all basic ideas in graph theory, we follow \cite{ha69}.

\section{Adam's Isomorphism or Type-1 Isomorphism and Type-1 group of $C_n(R)$}

In this section, we present our study on Adam's isomorphism or Type-1 isomorphism of circulant graph $C_n(R)$ and define Type-1 group of $C_n(R)$. 

\begin{lemma}{\rm \cite{v17} \quad \label{a4} Let $S \subseteq \mathbb{Z}_n$, $S \neq \emptyset$ and $x \in \mathbb{Z}_n.$ Define a mapping $\varphi_{n, x}:$ $S$ $\rightarrow$ $\mathbb{Z}_n$ $\ni$ $\varphi_{n, x}(s)$ = $xs$, $\forall$ $s \in S,$ under multiplication modulo $n$. Then $\varphi_{n, x}$ is bijective if and only if $S = \mathbb{Z}_n$ and $\gcd(n, x) = 1$. }
\end{lemma}
\begin{proof}\quad Here, we give a proof using a property of periodic cycles of a circulant graph $C_n(R)$ with jump size $x$. Let $S = \mathbb{Z}_n$. Then the numbers $0,$ $x,$ $2x,$ $3x,$ $\ldots,$ $(n-1)x$, under multiplication modulo $n$, are all distinct if and only if $\gcd(n, x)$ = 1 since the cycle of period $x$ in $C_n(R),$ $x \in R$, is of length $n$ if and only if $\gcd(n, x)$ = 1, using Corollary \ref{a2}.
	
Conversely, if $S \neq \mathbb{Z}_n$, then $S$ is a proper subset of $\mathbb{Z}_n$ and so $\varphi_{n,x}$ is not a bijective mapping. Hence the result follows. 
\end{proof}

Hereafter, unless otherwise it is mentioned in other way, we consider $\varphi_{n,x}$ with $\gcd(n,x)$ = 1 only. Let $\varphi_n$ = $\{ x \in \mathbb{Z}_n : \gcd(n, x) = 1 \}$. Clearly,  $( \varphi_n,~\circ)$ is an abelian group under the binary operation $\lq\circ\rq$,  multiplication modulo $n$.

\begin{definition}{\rm\cite{ad67}} \quad \label{a5} For $R =$ $\{r_1$, $r_2$, $\ldots$, $r_k\}$ and $S$ = $\{s_1$, $s_2$, $\ldots$, $s_k\}$, circulant graphs $C_n(R)$ and $C_n(S)$ are {\it Adam's isomorphic} if there exists a positive integer $x$ $\ni$ $\gcd(n, x)$ = 1 and $S$ = $\{xr_1$, $xr_2$, $\ldots$, $xr_k\}_n^*$ where $<r_i>_n^*$, the {\it reflexive modular reduction} of a sequence $< r_i >$, is the sequence obtained by reducing each $r_i$ under modulo $n$ to yield $r_i'$ and then replacing all resulting terms $r_i'$ which are larger than $\frac{n}{2}$ by $n-r_i'.$  
\end{definition}

\begin{definition}{\rm \cite{v17}} \label{a6} Let $Ad_n = \{\varphi_{n,x}: x\in \varphi_n\}$, $Ad_n(S) = \{\varphi_{n,x}(S): x\in \varphi_n\}$ = $\{xS: x\in \varphi_n\}$ and $Ad_n(C_n(R)) = T1_n(C_n(R)) = \{\varphi_{n,x}(C_n(R)) = C_n(\varphi_{n,x}(R)) = C_n(xR): x\in \varphi_n\}$ for sets $R,S \subseteq \mathbb{Z}_n$ where $\varphi_{n,x}(R)$ in $C_n(\varphi_{n,x}(R))$ is calculated under the reflexive modulo $n$. Define $'\circ'$ in $Ad_n(C_n(R))$ such that $\varphi_{n,x} \circ \varphi_{n,y}$ = $\varphi_{n,xy}$, $C_n(xR) \circ  C_n(yR)$ = $C_n((xy)R)$ and $(\varphi_{n,x} \circ \varphi_{n,y})(C_n(R))$ = $\varphi_{n,x}(C_n(R)) \circ \varphi_{n,y}(C_n(R))$ for every $x,y \in \varphi_n$, under arithmetic modulo $n$. Clearly, $Ad_n(C_n(R))$ is the set of all circulant graphs which are Adam's isomorphic to $C_n(R)$ and $(Ad_n(C_n(R)), \circ )$ = $(T1_n(C_n(R)), \circ )$ is an abelian group and we call it as the {\em Adam's group} or {\em Type-1 group on} $C_n(R)$ under $'\circ'$.
\end{definition}

{\rm In the above definition, $(\varphi_{n,x} \circ \varphi_{n,y})(C_n(R))$ = $\varphi_{n,x}( \varphi_{n,y}(C_n(R)))$ = $(\varphi_{n,x}(C_n(yR)))$ = $C_n(x(yR))$ = $C_n((xy)R)$ = $C_n(xR)$ $\circ$  $C_n(yR)$ = $\varphi_{n,x}(C_n(R))$ $\circ$ $\varphi_{n,y}(C_n(R))$, $\forall$ $x,y\in \varphi_n$, under arithmetic modulo $n$. Moreover, $C_n(S)\in Ad_n(C_n(R))$ implies, $\exists$ $x\in\varphi_n$ $\ni$ $C_n(S)$ = $\varphi_{n,x}(C_n(R))$ = $C_n(xR)$. Corresponding to $x\in \varphi_n$, $\exists$ $x^*\in \varphi_n$ $\ni$ $xx^*$ = $1\in\varphi_n$ and $\varphi_{n,x^*}(C_n(S))$ = $\varphi_{n,x^*}(C_n(xR))$ = $C_n(x^*(xR))$ = $C_n((x^*x)R)$ = $C_n(R)$. Thus, $C_n(R)$ = $\varphi_{n,x^*}(C_n(S))$ which implies, $C_n(R)\in Ad_n(C_n(S))$, $x^*\in \varphi_n$. This also implies, $Ad_n(C_n(R))$ = $Ad_n(C_n(S))$ whenever $C_n(S)\in Ad_n(C_n(R))$ or $C_n(R)\in Ad_n(C_n(S))$. Thus, we get the following result corresponding to $Ad_n(C_n(R))$. }
 
\begin{theorem} \quad \label{a7b} Let $Ad_n(C_n(R))$ = $\{\varphi_{n,x}(C_n(R)) = C_n(xR): x\in\varphi_n \}$. Then, $C_n(S)\in Ad_n(C_n(R))$ if and only if $Ad_n(C_n(R))$ = $Ad_n(C_n(S))$ if and only if $C_n(R)\in Ad_n(C_n(S))$. \hfill $\Box$
\end{theorem}

Circulant graphs $C_{54}(1,17,18,19),$ $C_{54}(5,13,18,23)$ and $C_{54}(7,11,18,25)$ are Adam's isomorphic since $C_{54}(5(1,17,~18,19))$ = $C_{54}(5,13,~18,23)$ and  $C_{54}(7(1,17,~18,19))$ = $C_{54}(7,11,~18,25)$. Also, $Ad_{54}(C_{54}(1,17,18,19))$ = $\{\varphi_{54,x}(C_{54}(1,17,18,19)): x\in \varphi_{54}\}$ = $\{\varphi_{54,x}(C_{54}(1,17,18,19)): x = 1,5,7,$ $11,13,17,19,23,25,29,31,35,37,41,43,47,49,53\}$ = $\{C_{54}(1,17, 18,19),$ $C_{54}(5,13,18,23),$ $C_{54}(7,11,18,$ $25)\}$ = $Ad_{54}(C_{54}(5,13,18,23))$ = $Ad_{54}(C_{54}(7,11,18,25))$. The three Adam's isomorphic circulant graphs are shown in Figures 4, 5, 6.  

\begin{figure}[ht]
	\centerline{\includegraphics[width=6.3in]{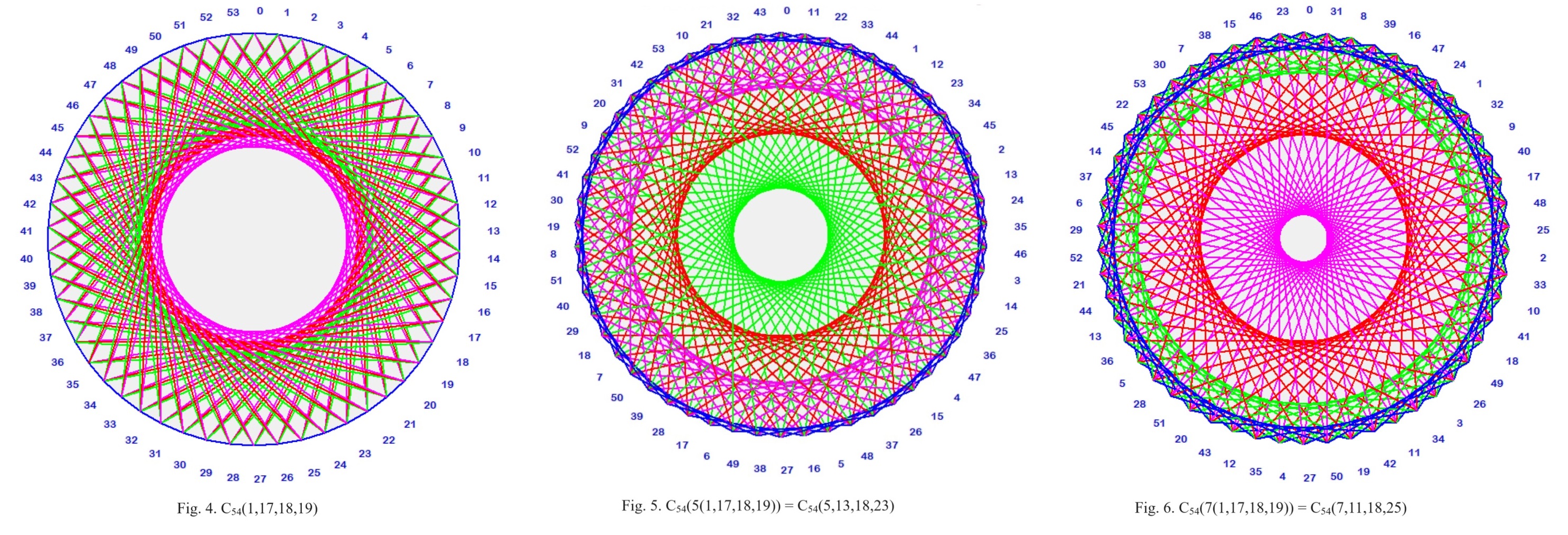}}
\end{figure}
A circulant graph $C_n(R)$ is said to have {\it Cayley Isomorphism} (CI) property if whenever $C_n(S)$ is isomorphic to $C_n(R),$ they are Adam's isomorphic. CI-problem determines which graphs (or which groups) have the $CI$-property. Muzychuk \cite{mu04, mu97} completed the complete classification of cyclic $CI$-groups, and moreover, derived an efficient algorithm for recognizing isomorphisms between circulant graphs using certain set of permutations, and proved that two circulants are isomorphic if and only if they are isomorphic via some function contained in the set. But investigation of circulant graphs without $CI$-property is not much done.  

\section{Type-2 isomorphism and Type-2 isomorphic circulant graphs}

Adam conjectured that any two isomorphic circulant graphs are
Adam’s isomorphic \cite{ad67}. Elspas and Turner \cite{eltu} rised a question on the type of isomorphism that exists on the two circulant graphs $C_{16}(1, 2, 7)$ and $C_{16}(2, 3, 5)$ which are  isomorphic but not of Adam’s. Circulant graphs $C_{16}(1, 2, 7)$ and $C_{16}(2, 3, 5)$ are given in Figures 7 and 8. In \cite{v96}, Vilfred defined a new type of circulant graph isomorphism, different from Adam’s isomorphism and studied under the heading ‘generalized circulant graph isomorphism’. Hereafter, Adam’s isomorphism on circulant graphs is also called as {\em Type-1 isomorphism} and the new type of isomorphism that is defined below is called {\em Type-2 isomorphism}.  In this section, we present our study on Type-2 isomorphism and Type-2 isomorphic circulant graphs. We also prove that $C_{16}(1, 2, 7)$ and $C_{16}(2, 3, 5)$ are Type-2 isomorphic and list all pairs of isomorphic circulant graphs of order 16 as well as all triples of isomorphic circulant graphs of order 27. This study is an extension, corrected and improved results obtained in \cite{v20,vw0A}. We start with a slightly modified lemma but with the same proof as given in \cite{v20} that is essential to define Type-2 isomorphism.
\begin{figure}[ht]
	\centerline{\includegraphics[width=3.2in]{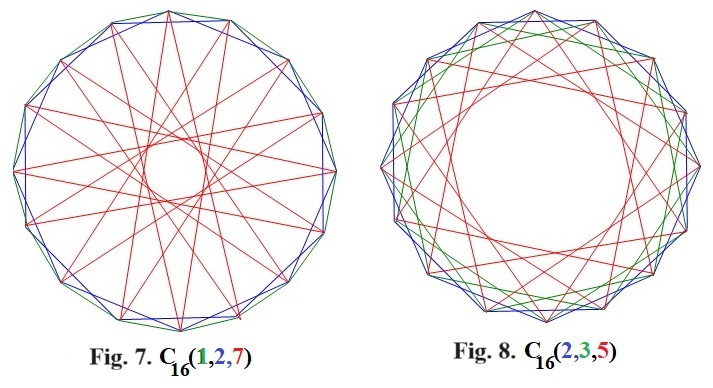}}
\end{figure}

\begin{lemma}{\rm \quad \label{a7}  Let $m >1$ be a divisor of $n$. Then, for each value of $t$, the mapping $\Theta_{n,m,t}: \mathbb{Z}_n \rightarrow \mathbb{Z}_n$ defined by $\Theta_{n,m,t}(x)$ = $x+jtm$, $x \in \mathbb{Z}_n$, is bijective under arithmetic modulo $n$ where $x = qm+j,$ $0 \leq j \leq m-1,$ $0 \leq q,t \leq \frac{n}{m}-1$ and $j,m,q,t \in \mathbb{Z}_n$. In particular, the result is true when $m$ = $\gcd(n, r) > 1$ and $r \in \mathbb{Z}_n$. } 
\end{lemma}
\begin{proof}\quad From the definition of $\Theta_{n,m,t}$, we get the following properties. 
	\begin{enumerate}
		\item[{\rm(i)}] $\Theta_{n,m,t}(km) = km$ for every $k \in \mathbb{Z}_n$, $km \in \mathbb{Z}_n$ and $0 \leq t \leq \frac{n}{m}-1$.
		\item[{\rm(ii)}] For $0 \leq i,j \leq m-1$ and $0 \leq t \leq \frac{n}{m}-1$, $\Theta_{n,m,t}(i)$ = $\Theta_{n,m,t}(j)$ if and only if $i$ = $j$ if and only if $\Theta_{n,m,t}(qm+i)$ = $\Theta_{n,m,t}(qm+j),$ $\gcd(n,r)$ = $m > 1$, and $0 \leq qm \leq n-1$.
		\item[{\rm(iii)}]	For $0 \leq i \leq m-1$, $0 \leq km,qm \leq n-1$ and $0 \leq q,t \leq \frac{n}{m}-1$,  $\Theta_{n,m,t}(km+i)$ = $\Theta_{n,m,t}(qm+i)$ if and only if $k$ = $q.$ 
		
		From the above three properties, we get,
		
		\item[{\rm(iv)}]	For $0 \leq i,j \leq m-1$, $0 \leq km,qm \leq n-1$ and $0 \leq t \leq \frac{n}{m}-1$, $\Theta_{n,m,t}(i+km)$ = $\Theta_{n,m,t}(j+qm)$ if and only if $i$ = $j$ and $k$ = $q$.  
	\end{enumerate}
	This implies that for each value of $t$, the mapping $\Theta_{n,m,t}$ is bijective when $m > 1$ is a divisor of $n$ and $0 \leq t \leq \frac{n}{m}-1$.  
	
	In particular, when $m$ = $\gcd(n, r) > 1$ and $r \in \mathbb{Z}_n$ implies, $m > 1$ is a divisor of $n$ and hence the result is also true in this case. Hence, we get the result.   
\end{proof}

\begin{rem}\quad \label{r8} Related to Lemma \ref{a7}, when $n > r > m$ and $\gcd(n, r)$ = $m > 1$, the mapping $\Theta_{n,r,t}: \mathbb{Z}_n \rightarrow \mathbb{Z}_n$ defined by $\Theta_{n,r,t}(x)$ = $x+jrt$ need not be bijective for every $t$ where  $x$ = $qr+j$, $0 \leq j \leq r-1$, $0 \leq q,t \leq \frac{n}{r}$, $0 \leq qrt \leq n-1$ and $j,m,q,t\in\mathbb{Z}_n$. Tables 2, 3 and 4 illustrate the above for $r$ = 6, 9 and 12, respectively when $R$ = $\{1,8,10,r, 27-r, 17,19,26\}$ $($of $C_{27}(R \cup (27-R)))$ so that $\Theta_{27,6,t}(x)$ = $x+6jt$, $\Theta_{27,9,t'}(x)$ = $x+9j't'$ and $\Theta_{27,12,t}(x)$ = $x+12j''t$, $0 \leq j \leq 5$, $0 \leq j' \leq 8$, $0 \leq j'' \leq 11$, $0 \leq t \leq \frac{n}{\gcd(n,r)}-1 = \frac{27}{\gcd(27, 6)}-1 = \frac{27}{\gcd(27, 12)}-1 = 8$ and $0 \leq t' \leq \frac{n}{\gcd(n,r)}-1 = 2$. In these tables non-bijective indicates the corresponding mapping $\Theta_{n,r,t}: R \cup (27-R) \rightarrow R \cup (27-R)$ is not bijective and thereby the mapping  $\Theta_{27,r,t}: \mathbb{Z}_{27} \rightarrow \mathbb{Z}_{27}$ defined by $\Theta_{n,r,t}(x)$ = $x+jrt$ need not be bijective for every $t$, $r$ = 6 and 12, and $R \cup (27-R) \subseteq \mathbb{Z}_{27}$. 
\end{rem}

Correspondingly, we get the following open problem.

\begin{oprm}\quad \label{op1} {\rm Let $n > n-r > r > m$, $\gcd(n, r)$ = $m > 1$ and $\Theta_{n,r,t}:$ $\mathbb{Z}_n \rightarrow \mathbb{Z}_n$ such that $\Theta_{n,r,t}(x)$ = $x+jrt$ where  $x$ = $qr+j$, $0 \leq j \leq r-1$, $0 \leq q,t \leq \frac{n}{\gcd(n,r)}-1$, $0 \leq qrt \leq n-1$ and $j,m,q,t \in \mathbb{Z}_n$. Then, 
		\begin{enumerate}
			\item Find values of $r$ and $t$ for which $\Theta_{n,r,t}$ is bijective.
			\item Find values of $r$ and $t$ for which $\Theta_{n,r,t}$ is not  bijective. \hfill $\Box$	 
	\end{enumerate} } 
\end{oprm}

\begin{table}\label{2} 
	\caption{\footnotesize{Calculation of $\Theta_{27,6,t}(R \cup (27-R))$, $R$ = $\{1,6,8,10\}$, $\Theta_{27,6,t}(x)$ = $x+6jt$, $x = 6q+j$, $0 \leq j \leq 5$, $x\in R \cup (27-R) \subseteq \mathbb{Z}_{27}$, $0 \leq t \leq \frac{27}{\gcd(27,6)}-1 = 8$.}}
	\begin{center}
		\scalebox{0.75}{
			\begin{tabular}{||c||c||c|c|c|c|c|c|c|c|c||} \hline \hline
				~ \hspace{.05cm} $t$ \hspace{.1cm} & \backslashbox{$\Theta_{27,6,t}(x)$}{Jump size \\ $x$}
				& \hspace{.1cm} 1 \hspace{.1cm} & \hspace{.1cm} 6 \hspace{.1cm} & \hspace{.1cm} 8 \hspace{.1cm} & \hspace{.1cm} 10 \hspace{.1cm} & \hspace{.1cm} 17 \hspace{.1cm} & \hspace{.1cm} 19 \hspace{.1cm} & \hspace{.1cm} 21 \hspace{.1cm} & \hspace{.1cm} 26 \hspace{.1cm} & -- \\\hline \hline
				& & &  &   &  &  & & & &  \\
				t & $\Theta_{27,6,t}(x)$ & 1+6t & 6 & 8+12t & 10+24t & 17+3t & 19+6t & 21+18t & 26+12t & {Identity or  Non-bijective}  \\ \hline \hline
				& & &  &   &  &  & & & &  \\
				0 & $\Theta_{27,6,0}(x)$ & 1 & 6 & 8 & 10 & 17 & 19 & 21 & 26 & Identity  \\\hline
				& & &  &   &  &  & & & & \\
				1 & $\Theta_{27,6,1}(x)$ & 7 & 6 & 20 & 7 & 20 & 25 & 12 & 11 & Non-bijective \\\hline 
				& & &  &   &  &  & & & & \\
				2 & $\Theta_{27,6,2}(x)$ & 13 & 6 & 5 & 4 & 23 & 4 & 3 & 23  & Non-bijective \\\hline
				& & &  &   &  &  & & &  & \\
				3 & $\Theta_{27,6,3}(x)$ & 19 & 6 & 17 & 1 & 26 & 10 & 21 & 8 & Identity \\\hline
				& & &  &   &  &  & & & & \\
				4 & $\Theta_{27,6,4}(x)$ & 25 & 6 & 2 & 25 & 2 & 16 & 12 & 20 & Non-bijective \\\hline
				& & &  &   &  &  & & & & \\
				5 & $\Theta_{27,6,5}(x)$ & 4 & 6 & 14 & 22 & 5 & 22 & 3 & 5 & Non-bijective \\\hline
				& & &  &   &  &  & & & & \\
				6 & $\Theta_{27,6,6}(x)$ & 10 & 6 & 26 & 19 & 8 & 1 & 21 & 17 & Identity \\\hline
				& & &  &   &  &  & & & & \\
				7 & $\Theta_{27,6,7}(x)$ & 16 & 6 & 11 & 16 & 11 & 7 & 12 & 2 & Non-bijective \\\hline
				& & &  &   &  &  & & & & \\
				8 & $\Theta_{27,6,8}(x)$ & 22 & 6  & 23 & 13 & 14 & 13 & 3 & 14 & Non-bijective  \\\hline\hline
		\end{tabular}}
	\end{center}
\end{table} 

\begin{table}\label{3} 
	\caption{\footnotesize{Calculation of $\Theta_{27,9,t}(R \cup (27-R))$, $R$ = $\{1,8,9,10\}$, $\Theta_{27,9,t}(x)$ = $x+9jt$, $x = 9q+j$, $0 \leq j \leq 8$, $x\in R \cup (27-R) \subseteq \mathbb{Z}_{27}$, $0 \leq t \leq \frac{27}{\gcd(27,9)}-1 = 2$.}}
	\begin{center}
		\scalebox{0.75}{
			\begin{tabular}{||c||c||c|c|c|c|c|c|c|c|c||} \hline \hline
				~ \hspace{.05cm} $t$ \hspace{.1cm} & \backslashbox{$\Theta_{27,9,t}(x)$}{Jump size \\ $x$}
				& \hspace{.1cm} 1 \hspace{.1cm} & \hspace{.1cm} 8 \hspace{.1cm} & \hspace{.1cm} 9 \hspace{.1cm} & \hspace{.1cm} 10 \hspace{.1cm} & \hspace{.1cm} 17 \hspace{.1cm} & \hspace{.1cm} 18 \hspace{.1cm} & \hspace{.1cm} 19 \hspace{.1cm} & \hspace{.1cm} 26 \hspace{.1cm} & -- \\\hline \hline
				& & &  &   &  &  & & & &  \\
				t & $\Theta_{27,9,t}(x)$ & 1+9t & 8+18t & 9 & 10+9t & 17+18t & 18 & 19+9t & 26+18t & {Identity or  non-bijective}   \\ \hline \hline
				& & &  &   &  &  & & & &  \\
				0 & $\Theta_{27,9,0}(x)$ & 1 & 8 & 9 & 10 & 17 & 18 & 19 & 26 & Identity  \\\hline
				& & &  &   &  &  & & & & \\
				1 & $\Theta_{27,9,1}(x)$ & 10 & 26 & 9 & 19 & 8 & 18 & 1 & 17 & Identity \\\hline 
				& & &  &   &  &  & & & & \\
				2 & $\Theta_{27,9,2}(x)$ & 19 & 17 & 9 & 1 & 26 & 18 & 10 & 8  & Identity \\\hline\hline
		\end{tabular}}
	\end{center}
\end{table} 

\begin{table}\label{4} 
	\caption{\footnotesize{Calculation of $\Theta_{27,12,t}(R \cup (27-R))$, $R$ = $\{1,8,10,12\}$, $\Theta_{27,12,t}(x)$ = $x+12jt$, $x = 12q+j$, $0 \leq j \leq 11$, $x\in R \cup (27-R)$, $0 \leq t \leq \frac{27}{\gcd(27,12)}-1 = 8$.}}
	\begin{center}
		\scalebox{0.75}{
			\begin{tabular}{||c||c||c|c|c|c|c|c|c|c|c||} \hline \hline
				~ \hspace{.05cm} $t$ \hspace{.1cm} & \backslashbox{$\Theta_{27,12,t}(x)$}{Jump size \\ $x$}
				& \hspace{.1cm} 1 \hspace{.1cm} & \hspace{.1cm} 8 \hspace{.1cm} & \hspace{.1cm} 10 \hspace{.1cm} & \hspace{.1cm} 12 \hspace{.1cm} & \hspace{.1cm} 15 \hspace{.1cm} & \hspace{.1cm} 17 \hspace{.1cm} & \hspace{.1cm} 19 \hspace{.1cm} & \hspace{.1cm} 26 \hspace{.1cm} & -- \\\hline \hline
				& & &  &   &  &  & & & &  \\
				t & $\Theta_{27,12,t}(x)$ & 1+12t & 8+15t & 10+12t & 12 & 15+9t & 17+6t & 19+3t & 26+24t & {Identity or  Non-bijective}   \\ \hline \hline
				& & &  &   &  &  & & & &  \\
				0 & $\Theta_{27,12,0}(x)$ & 1 & 8 & 10 & 12 & 15 & 17 & 19 & 26 & Identity  \\\hline
				& & &  &   &  &  & & & & \\
				1 & $\Theta_{27,12,1}(x)$ & 13 & 23 & 22 & 12 & 24 & 23 & 22 & 23 & Non-bijective \\\hline 
				& & &  &   &  &  & & & & \\
				2 & $\Theta_{27,12,2}(x)$ & 25 & 11 & 7 & 12 & 6 & 2 & 25 & 20  & Non-bijective \\\hline
				& & &  &   &  &  & & &  & \\
				3 & $\Theta_{27,12,3}(x)$ & 10 & 26 & 19 & 12 & 15 & 8 & 1 & 17 & Identity \\\hline
				& & &  &   &  &  & & & & \\
				4 & $\Theta_{27,12,4}(x)$ & 22 & 14 & 4 & 12 & 24 & 14 & 4 & 14 & Non-bijective \\\hline
				& & &  &   &  &  & & & & \\
				5 & $\Theta_{27,12,5}(x)$ & 7 & 2 & 16 & 12 & 6 & 20 & 7 & 11 & Non-bijective \\\hline
				& & &  &   &  &  & & & & \\
				6 & $\Theta_{27,12,6}(x)$ & 19 & 17 & 1 & 12 & 15 & 26 & 10 & 8 & Identity \\\hline
				& & &  &   &  &  & & & & \\
				7 & $\Theta_{27,12,7}(x)$ & 4 & 5 & 13 & 12 & 24 & 5 & 13 & 5 & Non-bijective \\\hline
				& & &  &   &  &  & & & & \\
				8 & $\Theta_{27,12,8}(x)$ & 16 & 20  & 25 & 12 & 6 & 11 & 16 & 2 & Non-bijective \\\hline\hline
		\end{tabular}}
	\end{center}
\end{table} 

Based on Lemma \ref{a7}, we define Type-2 isomorphism of circulant graphs in a more general sense than as given in \cite{v20} as follows.

\begin{definition} \quad  \label{a10} Let $V(K_n) = \{u_0,u_1,u_2,...,u_{n-1}\}$, $V(C_n(R)) = \{v_0,v_1,v_2,...,$ $v_{n-1}\},$ $r\in R$, $m > 1$ be a divisor of $\gcd(n, r)$ and $|R| \geq 3$.  Define 1-1 mapping (see Lemma \ref{a7}) $\Theta_{n,m,t} :$ $V(C_n(R)) \rightarrow V(K_n)$ $\ni$ $\Theta_{n,m,t}(v_x) = u_{x+jtm}$,  $\Theta_{n,m,t}((v_x, v_{x+s}))$ = $(\Theta_{n,m,t}(v_x),$ $\Theta_{n,m,t}(v_{x+s}))$ under subscript arithmetic modulo $n$ and $\Theta_{n,m,t}(C_n(R))$ = $C_n(\Theta_{n,m,t}(R))$ where $\Theta_{n,m,t}(R)$ in $C_n(\Theta_{n,m,t}(R))$ is calculated under the reflexive modulo $n$, $\forall$ $x \in \mathbb{Z}_n$, $x = qm+j,$ $0 \leq j \leq m-1$, $s\in R$ and $0 \leq q,t \leq \frac{n}{m} -1$. And for a particular value of $t,$ if  $\Theta_{n,m,t}(C_n(R))$ = $C_n(S)$ for some $S$  and  $S \neq yR$ for all $y\in \varphi_n$ under reflexive modulo $n,$ then $C_n(R)$ and $C_n(S)$ are called \it {isomorphic circulant graphs of Type-2 w.r.t. $m$.} 
\end{definition}

Using the above definition, we show in the following problem that circulant graphs $C_{16}(1,2,7)$ and $C_{16}(2,3,5)$ are Type-2 isomorphic w.r.t. $m$ = 2.

\begin{prm}\quad \label{p3} {\rm Show that circulant graphs $C_{16}(1,2,7)$ and $C_{16}(2,3,5)$ are Type-2 isomorphic w.r.t. $m$ = 2. Also, find $\Theta_{16,2,t}(C_{16}(1,2,7))$ for all possible values of $t$. }
\end{prm}
\noindent
{\bf Solution.}\quad Let $R$ = $\{1,2,7\}$, $S$ = $R \cup (16-R)$ and $T$ = $\{2,3,5\}$. $2\in R,T$ and $\gcd(n, r)$ = $\gcd(16, 2)$ = 2. Let $m$ = 2. We calculate $\Theta_{16,2,t}(C_{16}(1,2,7))$ for $t$ = 1 to $\frac{16}{\gcd(16, 2)}-1$ = 7 and present it in Table 5. From Table 5, we get,  

$\Theta_{16,2,0}(S)$ = $\Theta_{16,2,0}(\{1,2,7, 9,14,15\})$ = $\{1,2,7, 9,14,15\}$. 

$\Rightarrow$ $\Theta_{16,2,0}(C_{16}(R))$ = $C_{16}(R)$;

$\Theta_{16,2,1}(S)$ = $\Theta_{16,2,1}(\{1,2,7, 9,14,15\})$ = $\{1,2,3, 9,11,14\}$. 

$\Rightarrow$   $\Theta_{16,2,1}(C_{16}(R))$ $\neq$ $C_{16}(T)$ for any $T$ using Theorem \ref{ab14}; 

$\Theta_{16,2,2}(S)$ = $\Theta_{16,2,2}(\{1,2,7, 9,14,15\})$ = $\{2,3,5, 11,13,14\}$. 

$\Rightarrow$  $\Theta_{16,2,2}(C_{16}(R))$ = $C_{16}(2,3,5)$; 

$\Theta_{16,2,3}(S)$ = $\Theta_{16,2,3}(\{1,2,7, 9,14,15\})$ = $\{2,5,7, 13,14,15\}$. 

$\Rightarrow$  $\Theta_{16,2,3}(C_{16}(R))$ $\neq$ $C_{16}(T)$ for any $T$ using Theorem \ref{ab14};  

$\Theta_{16,2,4}(S)$ = $\Theta_{16,2,4}(\{1,2,7, 9,14,15\})$ = $\{1,2,7, 9,14,15\}$. 

$\Rightarrow$  $\Theta_{16,2,4}(C_{16}(R))$ = $C_{16}(R)$; 

$\Theta_{16,2,5}(S)$ = $\Theta_{16,2,5}(\{1,2,7, 9,14,15\})$ = $\{1,2,3, 9,11,14\}$. 

$\Rightarrow$   $\Theta_{16,2,5}(C_{16}(R))$ $\neq$ $C_{16}(T)$ for any $T$ using Theorem \ref{ab14}; 

$\Theta_{16,2,6}(S)$ = $\Theta_{16,2,6}(\{1,2,7, 9,14,15\})$ = $\{2,3,5, 11,13,14\}$. 

$\Rightarrow$  $\Theta_{16,2,6}(C_{16}(R))$ = $C_{16}(2,3,5)$;

$\Theta_{16,2,7}(S)$ = $\Theta_{16,2,7}(\{1,2,7, 9,14,15\})$ = $\{2,5,7, 13,14,15\}$. 

$\Rightarrow$ $\Theta_{16,2,7}(C_{16}(R))$ $\neq$ $C_{16}(T)$ for any $T$ using Theorem \ref{ab14}.

It is clear from Table 5 that there are 4 distinct isomorphic graphs $\Theta_{16,2,t}(C_{16}(1,2,7))$, without vertex label, for $t$ = 0, 1, 2,  3 and these graphs are shown in Figures 9 to 12.

Also, $\Theta_{16,2,2}(C_{16}(1,2,7))$ = $C_{16}(2,3,5)$ implies, $C_{16}(1,2,7)$ $\cong$ $C_{16}(2,3,5)$. 

$Ad_{16}(C_{16}(1,2,7))$ = $\{\varphi_{16,x}(C_{16}(1,2,7)): x = 1,3,5,7,9,11,13,15\}$ 

= $\{C_{16}(x(1,2,7)): x = 1,3,5,7,9,11,13,15\}$ = $\{C_{16}(1,2,7), C_{16}(3,5,6)\}$ 

= $\{C_{16}(x(1,2,7)): x = 1,3\}$. This implies, $C_{16}(2,3,5)$ $\notin Ad_{16}(C_{16}(1,2,7))$. Hence, $C_{16}(1,2,7)$ and $C_{16}(2,3,5)$ are Type-2 isomorphic w.r.t. $r$ = 2.  \hfill $\Box$
\begin{table}
	\caption{ Calculation of $\Theta_{16,2,t}(\{1,2,7, 9,14,15\})$ for $t$ = 0 to 7.}
	\begin{center}
		\scalebox{.95}{
			\begin{tabular}{||c||*{6}{c|}|c||c||}\hline \hline 
				Jump size $x$ &  1 & 2 & 7 & 9 & 14 & 15 & Pairwise Equidistant  \\ \cline{1-7} 
				\backslashbox{ {\hspace{.5cm} $t$}}{$\Theta_{16,2,t}(x)$}
				& $1+2t$ & 2 &
				$7+2t$ & $9+2t$ & 14 & $15+2t$ & from $v_0$ or not
				\\\hline \hline
				0 & 1 & 2 & 7 & 9 & 14 & 15 & Yes \\
				1 & 3 & 2 & 9 & 11 & 14 & 1 & Not\\
				2 & 5 & 2 & 11 & 13 & 14 & 3 & Yes \\
				3 & 7 & 2 & 13 & 15 & 14 & 5 & Not\\
				\hline\hline 
				4 & 9 & 2 & 15 & 1 & 14 & 7 & Yes \\
				5 & 11 & 2 & 1 & 3 & 14 & 9 & Yes \\
				6 & 13 & 2 & 3 & 5 & 14 & 11 & Yes \\
				7 & 15 & 2 & 5 & 7 & 14 & 13 & Not \\
				\hline \hline
		\end{tabular}}
	\end{center}
\end{table} 

\begin{figure}[ht]
	\centerline{\includegraphics[width=4.5in]{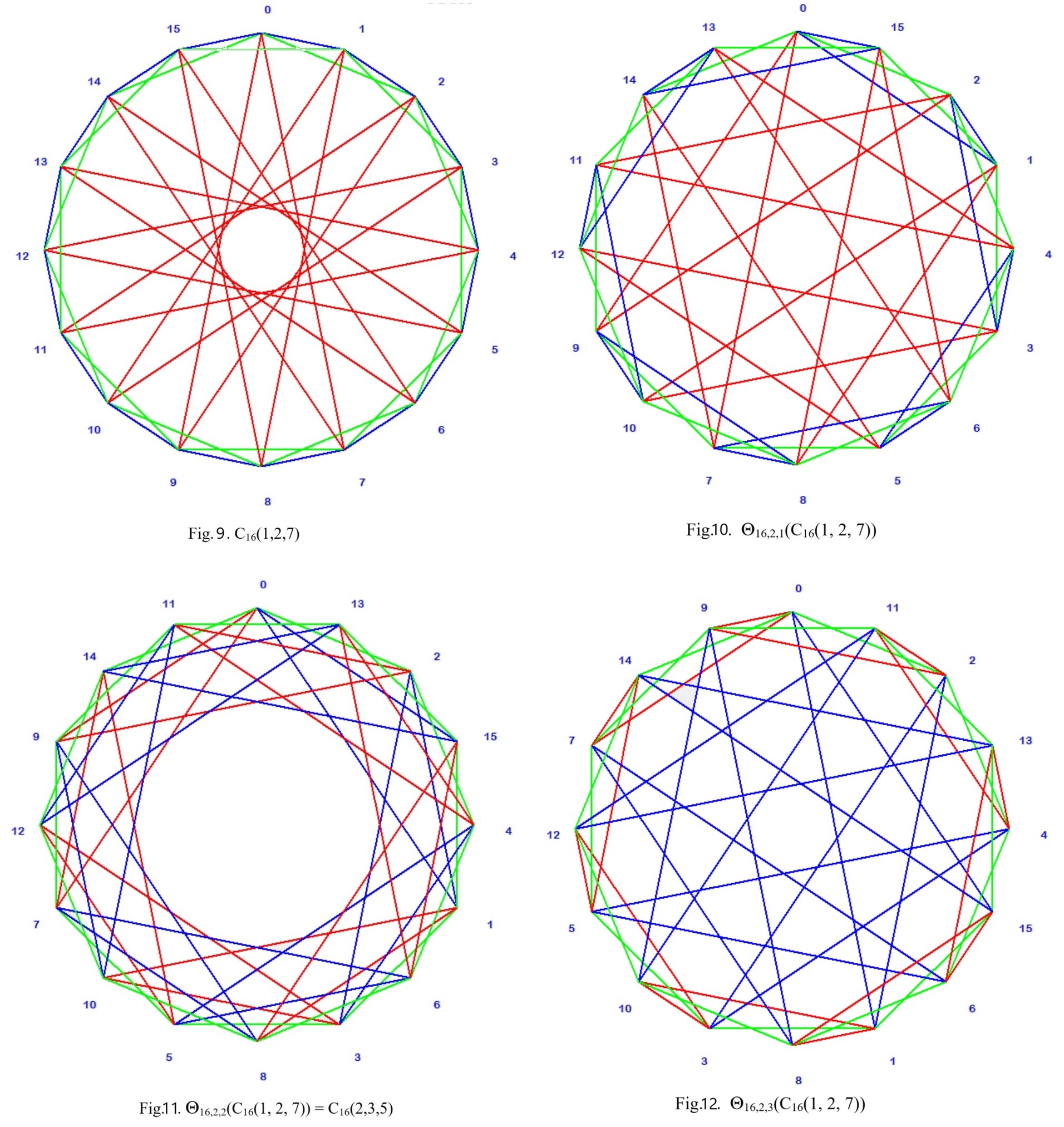}}
\end{figure}

\begin{prm}\quad \label{p4} {\rm Show that circulant graphs 
\begin{enumerate}
\item [\rm (a1)] $C_{48}(1,2,23)$ and $C_{48}(2,11,13)$ are Type-2 isomorphic w.r.t. $m$ = 2.

\item [\rm (a2)] $C_{48}(1,2,23)$ and $C_{48}(2,11,13)$ are not Type-2 isomorphic w.r.t. $m$ = 6.

\item [\rm (b1)] $C_{48}(1,6,23)$ and $C_{48}(6,11,13)$ are Type-2 isomorphic w.r.t. $m$ = 2.

\item [\rm (b2)] $C_{48}(1,6,23)$ and $C_{48}(6,11,13)$ are Type-2 isomorphic w.r.t. $m$ = 6.
\end{enumerate}
 }
\end{prm}
\noindent
{\bf Solution.}\quad 
\begin{enumerate}
\item [\rm (a1)] Let $R$ = $\{1,2,23\}$, $S$ = $R \cup (48-R)$ and $T$ = $\{2,11,13\}$. $2\in R,T$ and $\gcd(n, r)$ = $\gcd(48, 2)$ = 2. Let $m$ = 2. We calculate $\Theta_{48,2,t}(C_{48}(1,2,23))$ for $t$ = 1 to $\frac{48}{\gcd(48, 2)}-1$ = 23. Consider,

$\Theta_{48,2,1}(S)$ = $\Theta_{48,2,1}(\{1,2,23, 25,46,47\})$ = $\{3,2,25, 27,46,1\}$ = $\{1,2,3, 25,27,46\}$; 

$\Theta_{48,2,2}(S)$ = $\Theta_{48,2,2}(\{1,2,23, 25,46,47\})$ =  $\{5,2,27, 29,46,3\}$ = $\{2,3,5, 27,29,46\}$; 

$\Theta_{48,2,3}(S)$ = $\Theta_{48,2,3}(\{1,2,23, 25,46,47\})$ = $\{7,2,29, 31,46,5\}$ = $\{2,5,7, 29,31,46\}$; 

$\Theta_{48,2,4}(S)$ = $\Theta_{48,2,4}(\{1,2,23, 25,46,47\})$ = $\{9,2,31, 33,46,7\}$ = $\{2,7,9, 31,33,46\}$; 

$\Theta_{48,2,5}(S)$ = $\Theta_{48,2,5}(\{1,2,23, 25,46,47\})$ = $\{11,2,33, 35,46,9\}$ = $\{2,9,11, 33,35,46\}$. 

$\Rightarrow$  $\Theta_{48,2,i}(C_{48}(R))$ $\neq$ $C_{48}(X)$ for any $X$ for $i$ = 1 to 5 using Theorem \ref{ab14}. Also, 

$\Theta_{48,2,6}(S)$ = $\Theta_{48,2,6}(\{1,2,23, 25,46,47\})$ =  $\{13,2,35, 37,46,11\}$ =  $\{2,11,13, 35,37,46\}$. 

$\Rightarrow$  $\Theta_{48,2,6}(C_{48}(R))$ = $C_{48}(T)$. $\Rightarrow$  $C_{48}(R)$ $\cong$ $C_{48}(T)$.
 
By similar calculations, we get, 
$\Theta_{48,2,i+6t}(S)$ = $\Theta_{48,2,i}(S)$  for  $0 \leq i \leq 5$ and $0 \leq 2(i+6t) \leq 23$; 

$\Theta_{48,2,i+6t}(S)$ $\neq$ $C_{48}(X)$ for any $X$ for $1 \leq i \leq 5$ and $0 \leq 2(i+6t) \leq 23$ and    

  $\Theta_{48,2,6}(C_{48}(1,2,23))$ = $C_{48}(2,11,13)$. $\Rightarrow$ $C_{48}(1,2,23)$ $\cong$ $C_{48}(2,11,13)$. Also,

$Ad_{48}(C_{48}(1,2,23))$ = $Ad_{48}(C_{48}(1,2,23, 25,46,47))$ 

= $\{\varphi_{48,x}(C_{48}(1,2,23, 25,46,47)): x\in\varphi_{48}\}$ 
 =  $\{C_{48}(x(1,2,23, 25,46,47)): x\in\varphi_{48}\}$ 

= $\{C_{48}(1,2,23, 25,46,47), C_{48}(5,10,19, 29,38,43),$ 

\hfill $C_{48}(7,14,17, 31,34,41), C_{48}(11,13,22, 36,35,37)\}$ 

= $\{C_{48}(1,2,23), C_{48}(5,10,19), C_{48}(7,14,17), C_{48}(11,13,22)\}$.

$\Rightarrow$  $C_{48}(2,11,13) \notin Ad_{48}(C_{48}(1,2,23))$.

 This implies, $C_{48}(1,2,23)$ and $C_{48}(2,11,13)$ are Type-2 isomorphic w.r.t. $r$ = 2. 
 
 \item [\rm (a2)] Let $R$ = $\{1,2,23\}$, $S$ = $R \cup (48-R)$ and $T$ = $\{2,11,13\}$. $2\in R,T$ and $\gcd(48, 6)$ = 6. 
 \\
 Let $m$ = 6. Calculating $\Theta_{48,6,t}(C_{48}(1,2,23))$ for $t$ = 1 to $\frac{48}{\gcd(48, 6)}-1$ = 7, we get,
 
 $\Theta_{48,6,1}(S)$ = $\Theta_{48,6,1}(\{1,2,23, 25,46,47\})$ = $\{7,14,5, 31,22,29\}$ = $\{5,7,14, 22,29,31\}$; 
   
 $\Theta_{48,6,2}(S)$ = $\Theta_{48,6,2}(\{1,2,23, 25,46,47\})$ = $\{13,26,35, 37,46,11\}$ = $\{11,13,26, 35,37,46\}$; 

$\Theta_{48,6,3}(S)$ = $\Theta_{48,6,3}(\{1,2,23, 25,46,47\})$ = $\{19,38,17, 43,22,41\}$ = $\{17,19,22, 38,41,43\}$; 
 
$\Theta_{48,6,4}(S)$ = $\Theta_{48,6,4}(\{1,2,23, 25,46,47\})$ = $\{25,2,47, 1,46,23\}$ = $\{1,2,23, 25,46,47\}$; 

$\Theta_{48,6,5}(S)$ = $\Theta_{48,6,5}(\{1,2,23, 25,46,47\})$ = $\{31,14,29, 7,22,5\}$ = $\{5,7,14, 22,29,31\}$; 

$\Theta_{48,6,6}(S)$ = $\Theta_{48,6,6}(\{1,2,23, 25,46,47\})$ = $\{37,26,11, 13,46,35\}$ = $\{11,13,26, 35,37,46\}$; 

$\Theta_{48,6,7}(S)$ = $\Theta_{48,6,7}(\{1,2,23, 25,46,47\})$ = $\{43,38,41, 19,22,17\}$ = $\{17,19,22, 38,41,43\}$; 

$\Rightarrow$  $\Theta_{48,6,i}(C_{48}(R))$ $\neq$ $C_{48}(X)$ for any $X$ for $i$ = 1 to 7 using Theorem \ref{ab14}. 
\\
$\Rightarrow$ $C_{48}(1,2,23)$ and $C_{48}(2,11,13)$ are not Type-2 isomorphic w.r.t. $m$ = 6 since $\gcd(48, 6)$ = 6.
 
\item [\rm (b1)] Let $R$ = $\{1,6,23\}$, $S$ = $R \cup (48-R)$ and $T$ = $\{6,11,13\}$. $6\in R,T$ and $\gcd(48, 2)$ = 2. Let $m$ = 2. We calculate $\Theta_{48,2,t}(C_{48}(1,6,23))$ for $t$ = 1 to $\frac{48}{\gcd(48, 2)}-1$ = 23. We have $\Theta_{48,2,t}(6)$ = 6 and using the calculations of $(a1)$, we get,

$\Theta_{48,2,1}(S)$ = $\{1,6,3, 25,27,42\}$; 

 $\Theta_{48,2,2}(S)$ = $\{3,5,6, 27,29,42\}$; 

$\Theta_{48,2,3}(S)$ = $\{5,6,7, 29,31,42\}$; 

$\Theta_{48,2,4}(S)$ = $\{6,7,9, 31,33,42\}$; 

$\Theta_{48,2,5}(S)$ = $\{6,9,11, 33,35,42\}$; 

$\Theta_{48,2,6}(S)$ = $\{6,11,13, 35,37,42\}$. 

This implies,  $\Theta_{48,2,i}(C_{48}(R))$ $\neq$ $C_{48}(X)$ for any $X$ for $i$ = 1 to 5 using Theorem \ref{ab14} and $\Theta_{48,2,6}(C_{48}(R))$ = $C_{48}(T)$ which implies, $C_{48}(R)$ $\cong$ $C_{48}(T)$.

Similarly, we get, 
$\Theta_{48,2,i+6t}(S)$ = $\Theta_{48,2,i}(S)$  for  $0 \leq i \leq 5$ and $0 \leq 2(i+6t) \leq 23$; 

$\Theta_{48,2,i+6t}(S)$ $\neq$ $C_{48}(X)$ for any $X$ for $1 \leq i \leq 5$ and $0 \leq 2(i+6t) \leq 23$ and    

$\Theta_{48,2,6}(C_{48}(1,6,23))$ = $C_{48}(6,11,13)$. $\Rightarrow$ $C_{48}(1,6,23)$ $\cong$ $C_{48}(6,11,13)$.

Also, $Ad_{48}(C_{48}(1,6,23))$ 
= $\{C_{48}(1,6,23), C_{48}(5,18,19), C_{48}(6,7,17), C_{48}(11,13,18)\}$.

$\Rightarrow$  $C_{48}(6,11,13) \notin Ad_{48}(C_{48}(1,6,23))$.

This implies, $C_{48}(1,6,23)$ and $C_{48}(6,11,13)$ are Type-2 isomorphic w.r.t. $m$ = 2. 
 
 Type-2 isomorphic circulant graphs $C_{48}(1,2,23)$ and $\Theta_{48,2,6}(C_{48}(1,2,23))$ = $C_{48}(2,11,13)$ as well as  $C_{48}(1,6,23)$ and $\Theta_{48,6,2}(C_{48}(1,6,23))$ = $C_{48}(6,11,13)$ are given in Figures 13 to 16.
 
 \item [\rm (b2)] Let $R$ = $\{1,6,23\}$, $S$ = $R \cup (48-R)$ and $T$ = $\{6,11,13\}$. $6\in R,T$ and $\gcd(48, 6)$ = 6. Let $m$ = 6. We calculate $\Theta_{48,6,t}(C_{48}(1,6,23))$ for $t$ = 1 to $\frac{48}{\gcd(48, 6)}-1$ = 7. We have $\Theta_{48,2,t}(6)$ = 6 and using the calculations of $(a2)$, we get,
 
 $\Theta_{48,6,1}(S)$ = $\{5,6,7, 29,31,42\}$;
   
 $\Theta_{48,6,2}(S)$ = $\{6,11,13, 35,37,42\}$; 
 
 $\Theta_{48,6,3}(S)$ = $\{6,17,19, 41,42,43\}$;
   
 $\Theta_{48,6,4}(S)$ = $\{1,6,23, 25,42,47\}$; 
 
 $\Theta_{48,6,5}(S)$ = $\{5,6,7, 29,31,42\}$;
   
 $\Theta_{48,6,6}(S)$ = $\{6,11,13, 35,37,42\}$; 
 
 $\Theta_{48,6,7}(S)$ = $\{6,17,19, 41,42,43\}$. 
 
 This implies,  $\Theta_{48,6,i}(C_{48}(R))$ $\neq$ $C_{48}(X)$ for any $X$ for $i$ = 1, 3, 5, 7 using Theorem \ref{ab14} and $C_{48}(1,6,23)$ $\cong$ $C_{48}(6,11,13)$ since  $\Theta_{48,6,6}(1,6,23, 25,42,47)$ = $\{6,11,13, 35,37,42\}$. 
 
 Also, $C_{48}(6,11,13) \notin Ad_{48}(C_{48}(1,6,23))$, see $(b1)$.
 
 This implies, $C_{48}(1,6,23)$ and $C_{48}(6,11,13)$ are Type-2 isomorphic w.r.t. $m$ = 6.  
\begin{figure}[ht]
	\centerline{\includegraphics[width=4.5in]{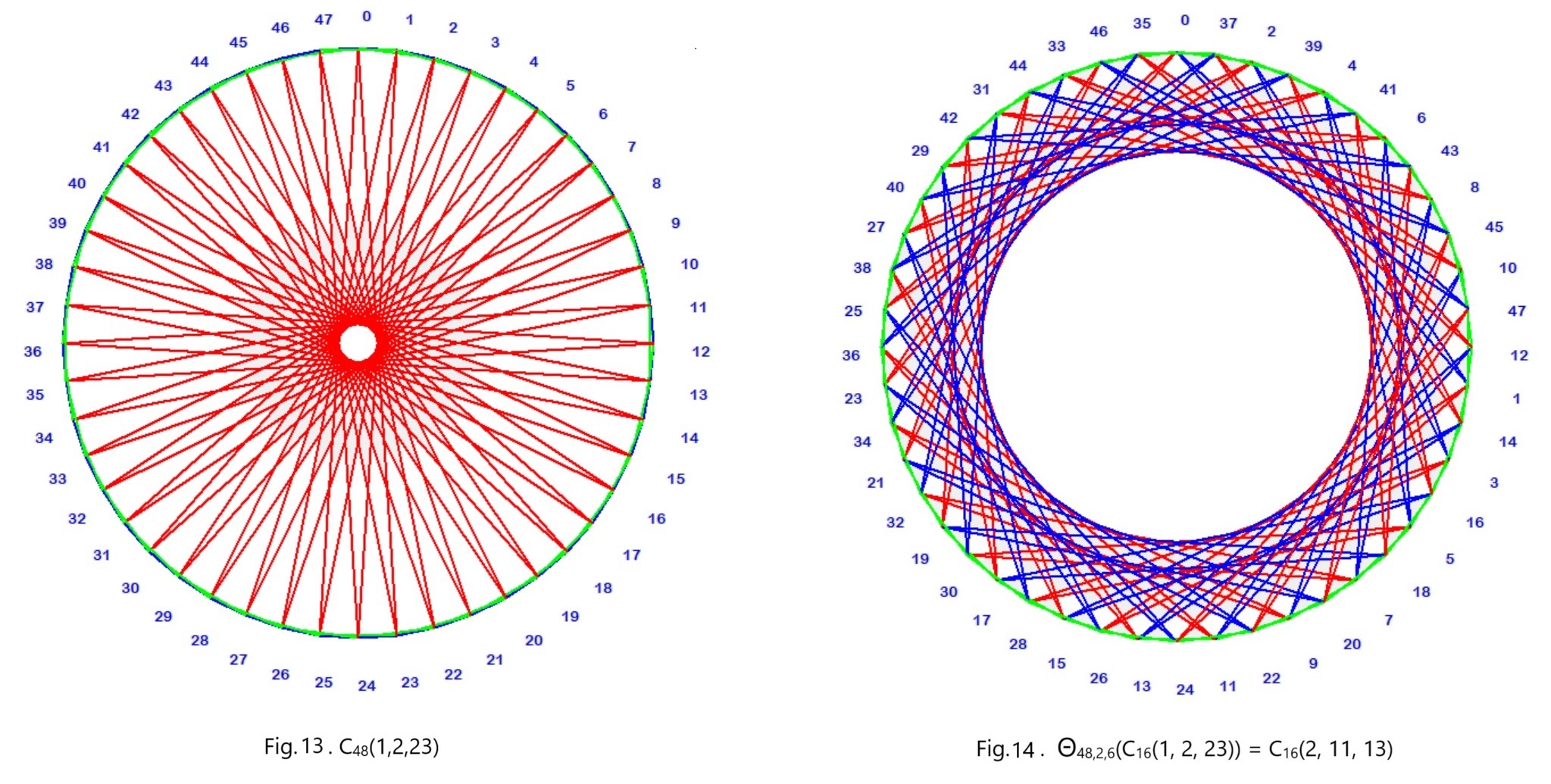}}
\end{figure}
\begin{figure}[ht]
	\centerline{\includegraphics[width=4.5in]{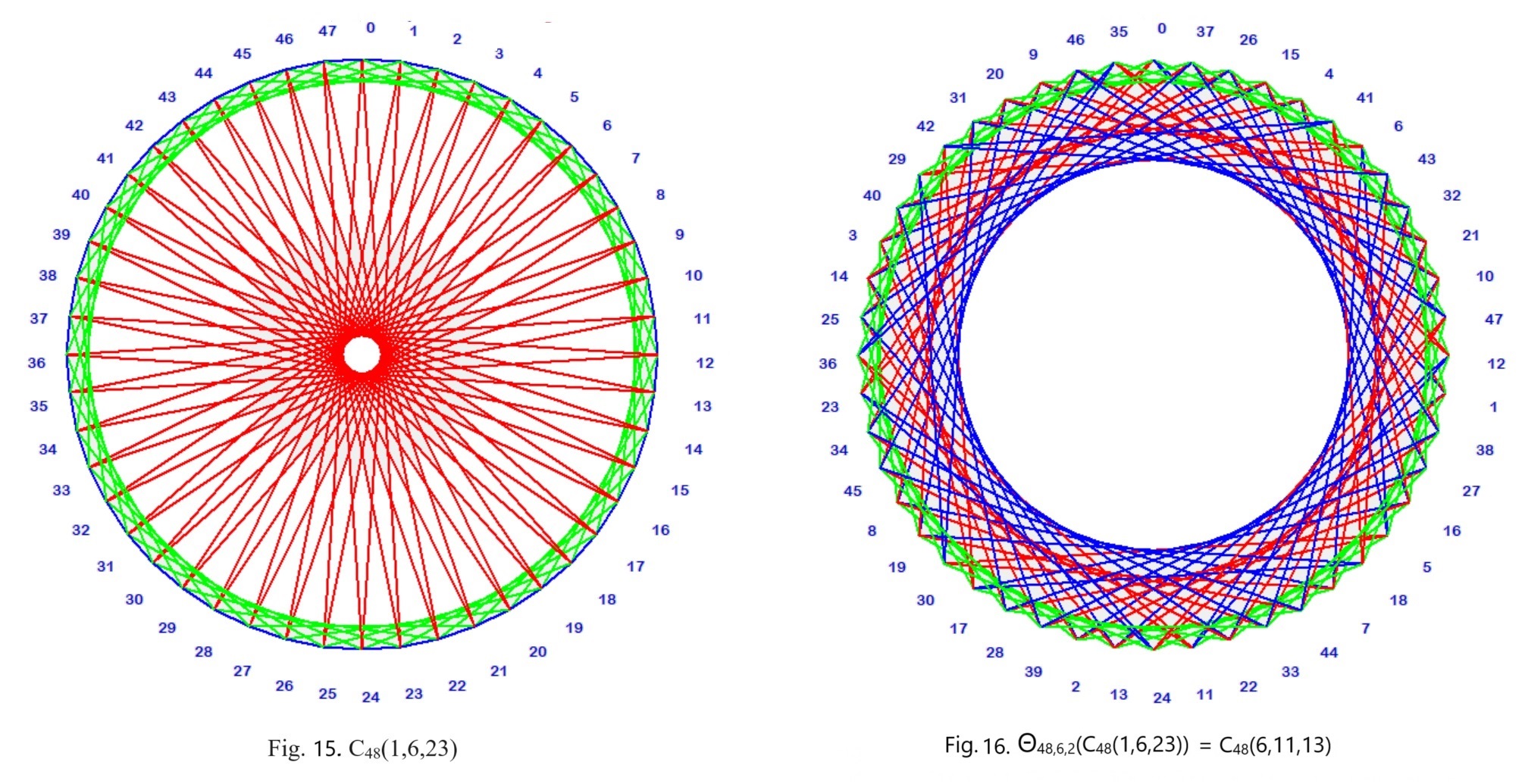}}
\end{figure}
 \hfill $\Box$
\end{enumerate}

\begin{prm}\quad \label{p5} {\rm Show that circulant graphs 
		\begin{enumerate}
		\item [\rm (a1)] $C_{96}(1,2,47)$ and $C_{96}(2,23,25)$ are Type-2 isomorphic w.r.t. $m$ = 2.
			
		\item [\rm (a2)] $C_{96}(1,2,47)$ and $C_{96}(2,23,25)$ are not Type-2 isomorphic w.r.t. $m$ = 6.
			
		\item [\rm (b1)] $C_{96}(1,6,47)$ and $C_{96}(6,23,25)$ are Type-2 isomorphic w.r.t. $m$ = 2.
			
		\item [\rm (b2)] $C_{96}(1,6,47)$ and $C_{96}(6,23,25)$ are Type-2 isomorphic w.r.t. $m$ = 6.
\end{enumerate}
	}
\end{prm}
\noindent
{\bf Solution.}\quad 
\begin{enumerate}
\item [\rm (a1)] Let $R$ = $\{1,2,47\}$, $S$ = $R \cup (96-R)$ and $T$ = $\{2,23,25\}$. $2\in R,T$ and $\gcd(96, 2)$ = 2. Let $m$ = 2. We calculate $\Theta_{96,2,t}(C_{96}(1,2,47))$ for $t$ = 1 to $\frac{96}{\gcd(96, 2)}-1$ = 47. It is easy to check that for $0 \leq i+12t \leq 47$ and $0 \leq i \leq 11$,
	
	$\Theta_{96,2,i+12t}(C_{96}(1,2,47))$ = $\Theta_{96,2,i}(C_{96}(1,2,47))$ and 
	
	$\Theta_{96,2,i}(C_{96}(1,2,47))$ $\neq$ $C_{96}(X)$ for any $X$;
	
	$\Theta_{96,2,12}(C_{96}(1,2,47))$ = $C_{96}(2,23,25)$;
	
	$\Theta_{96,2,2\times 12}(C_{96}(1,2,47))$ = $C_{96}(1,2,47)$; 
	
	$\Theta_{96,2,3\times 12}(C_{96}(1,2,47))$ = $C_{96}(2,23,25)$ and 
	
	$C_{96}(2,23,25) \notin Ad_{96}(C_{96}(1,2,47))$ since $Ad_{96}(C_{96}(1,2,47))$ = $\{C_{96}(1,2,47), C_{96}(5,10,43),$ 
	
	$C_{96}(7,14,41), C_{96}(11,22,37), C_{96}(13,26,35), C_{96}(17,31,34), C_{96}(19,29,38),C_{96}(23,25,46) \}$.
	
	This implies, $C_{96}(1,2,47)$ and $C_{96}(2,23,25)$ are Type-2 isomorphic w.r.t. $m$ = 2.   
	
	\item [\rm (a2)] Let $R$ = $\{1,2,47\}$, $S$ = $R \cup (96-R)$ and $T$ = $\{2,23,25\}$. $\gcd(96, 2)$ = 2 and $\gcd(96, 6)$ = 6. Let $m$ = 6. We calculate $\Theta_{96,6,t}(C_{96}(1,2,47))$ for $t$ = 1 to $\frac{96}{\gcd(96, 6)}-1$ = 15. It is easy to check  
	
	$\Theta_{96,6,8}(C_{96}(1,2,47))$ = $C_{96}(1,2,47)$ and 
	
	$\Theta_{96,6,i}(C_{96}(1,2,47))$ $\neq$ $C_{96}(X)$ for any $X$ for $i \neq 8$ and $1 \leq i \leq 15$.
	
	This implies, $C_{96}(1,2,47)$ and $C_{96}(2,23,25)$ are not Type-2 isomorphic w.r.t. $m$ = 6. 
	
	\item [\rm (b1)] Let $R$ = $\{1,6,47\}$, $S$ = $R \cup (96-R)$ and $T$ = $\{6,23,25\}$. $\gcd(96, 2)$ = 2 and $\gcd(96, 6)$ = 6. Let $m$ = 2. We calculate $\Theta_{96,2,t}(C_{96}(1,6,47))$ for $t$ = 1 to $\frac{96}{\gcd(96, 2)}-1$ = 47. As $\Theta_{96,2,t}(6)$ = 6 and using the calculations of $(a1)$, it is easy to verify that for $0 \leq i+12t \leq 47$ and $0 \leq i \leq 11$,
	
	$\Theta_{96,2,i+12t}(C_{96}(1,6,47))$ = $\Theta_{96,2,i}(C_{96}(1,6,47))$ and 
	
	$\Theta_{96,2,i}(C_{96}(1,6,47))$ $\neq$ $C_{96}(X)$ for any $X$;
	
	$\Theta_{96,2,12}(C_{96}(1,6,47))$ = $C_{96}(6,23,25)$;
	
	$\Theta_{96,2,2\times 12}(C_{96}(1,6,47))$ = $C_{96}(1,6,47)$; 
	
	$\Theta_{96,2,3\times 12}(C_{96}(1,6,47))$ = $C_{96}(6,23,25)$ and $C_{96}(6,23,25) \notin Ad_{96}(C_{96}(1,6,47))$.
	
	This implies, $C_{96}(1,2,47)$ and $C_{96}(2,23,25)$ are Type-2 isomorphic w.r.t. $m$ = 2.   
	
	\item [\rm (b2)] Let $R$ = $\{1,6,47\}$, $S$ = $R \cup (96-R)$ and $T$ = $\{6,23,25\}$. $\gcd(96, 6)$ = 6 and $6\in R,T$. Let $m$ = 6. We calculate $\Theta_{96,6,t}(C_{96}(1,6,47))$ for $t$ = 1 to $\frac{96}{\gcd(96, 6)}-1$ = 15. It is easy to check that 
	
	$\Theta_{96,6,i}(C_{96}(1,6,47))$ $\neq$ $C_{96}(X)$ for any $X$ for $i \neq 4, 8, 12$ and $0 \leq i \leq 15$,	
	
	$\Theta_{96,6,4}(C_{96}(1,6,47))$ = $C_{96}(6,23,25)$,
	
	$\Theta_{96,6,8}(C_{96}(1,6,47))$ = $C_{96}(1,6,47)$,
	
	$\Theta_{96,6,12}(C_{96}(1,6,47))$ = $C_{96}(6,23,25)$ and 	
	$C_{96}(6,23,25) \notin Ad_{96}(C_{96}(1,6,47))$.
	
	This implies, $C_{96}(1,6,47)$ and $C_{96}(6,23,25)$ are Type-2 isomorphic w.r.t. $m$ = 6.    \hfill $\Box$
\end{enumerate}

\begin{rem}\quad \label{r11} Following steps are used to establish isomorphism of Type-2 w.r.t. $r$ between circulant graphs $C_n(R)$ and $C_n(S)$. (i) $R$ $\neq$ $S$ and $|R| = |S| \geq 3$; (ii) $\exists$ $r\in R,S$ and $m > 1$ $\ni$ $m$ is a divisor of $\gcd(n, r)$ and for some $t$ $\ni$ $1 \leq t \leq \frac{n}{m} -1$, $\Theta_{n,m,t}(C_n(R))$ = $C_n(S)$ and (iii) $S$ $\neq$ $xR$ for all $x\in\varphi_n$ under arithmetic reflexive modulo $n$. 
\end{rem} 

\begin{rem} \label{r12} \quad While searching for possible value(s) of $t$ for which the transformed graph $\Theta_{n,m,t}(C_n(R))$ is circulant of the form $C_n(S)$ for some $S \subseteq [1, \frac{n}{2}],$ the calculation on $r_i$s which are integer multiples of $m$ need not be done  under the transformation $\Theta_{n,m,t}$ as there is no change in these $r_i$s where $m > 1$ is a divisor of $\gcd(n, r)$ and $r\in R$. Also, for a given circulant graph $C_n(R)$, w.r.t. different values of $m$, we may get different Type-2 isomorphic circulant graphs.
\end{rem}

\begin{note} \label{a13}  {\rm We consider $t$ = $0,1,\ldots,\frac{n}{m}-1$ in the transformation $\Theta_{n,m,t}$ on $C_n(R)$ since the length of a periodic cycle of period $r$ in $C_n(R)$ is $\frac{n}{m}$ (Theorem \ref{a1}) where $m > 1$ is a divisor of $\gcd(n,r)$ and $r\in R$.}
\end{note}

\section{ Type-2 isomorphic circulant graphs of orders 16 and 27}

Given a circulant graph $C_n(R)$ having isomorphic circulant graphs of Type-2 w.r.t. $m$, Remark \ref{r12} helps us to obtain more  isomorphic graphs which covers Type-2 w.r.t. $m$ as well as some Adam's isomorphic graphs of $C_n(R)$. In this section, we find all Type-2 isomorphic circulant graphs of orders 16 and 27. We show that there are 8 pairs of isomorphic circulant graphs of Type-2  on 16 points and 12 triples of isomorphic circulant graphs of Type-2 on 27 points. We start with circulant graphs of order 16.  

\begin{prm}\quad \label{p41} Show that there are 8 pairs of Type-2 isomorphic circulant graphs of order 16 and they are isomorphic of Type-2 w.r.t. $m$ = 2 only.
\end{prm}
\noindent
{\bf Solution.}\quad Here, $n$ = 16 and so the possible values of $m > 1$ $\ni$ $m$ is a divisor of $\gcd(n, r)$ = $\gcd(16, r)$ are $m$ = 2, 4, 8 and $\gcd(16, 2)$ = 2 = $\gcd(16, 6)$, $\gcd(16, 4)$ = 4, $\gcd(16, 8)$ = 8. We start with finding isomorphic circulant graphs of Type-2 w.r.t. $m$ = 2. In Problem \ref{p3}, we proved that circulant graphs 

(1) $C_{16}(1,2,7)$ and $C_{16}(2,3,5)$ are isomorphic of Type-2 w.r.t. $m$ = 2.

Using Remark \ref{r12} in this pair of isomorphic circulant graphs of Type-2 w.r.t. $m$ = 2, we get the following pairs of isomorphic circulant graphs of the form $C_{16}(R)$ which are either of Adam's or of Type-2 w.r.t. $m$ = 2. 		
\begin{enumerate}
	
	\item [\rm (2)]	$C_{16}(1,6,7),$ $C_{16}(3,5,6) = \Theta_{16,2,2}(C_{16}(1,6,7));$ 
	
	\item [\rm (3)]	$C_{16}(1,2,4,7),$ $C_{16}(2,3,4,5) = \Theta_{16,2,2}(C_{16}(1,2,4,7));$ 
	
	\item [\rm (4)]	$C_{16}(1,2,6,7),$ $C_{16}(2,3,5,6) = \Theta_{16,2,2}(C_{16}(1,2,6,7));$ 
	
	\item [\rm (5)]	$C_{16}(1,2,7,8),$ $C_{16}(2,3,5,8) = \Theta_{16,2,2}(C_{16}(1,2,7,8));$
	
	\item [\rm (6)]	$C_{16}(1,4,6,7),$ $C_{16}(3,4,5,6) = \Theta_{16,2,2}(C_{16}(1,4,6,7));$ 
	
	\item [\rm (7)]	$C_{16}(1,6,7,8),$ $C_{16}(3,5,6,8) = \Theta_{16,2,2}(C_{16}(1,6,7,8));$ 
	
	\item [\rm (8)]	$C_{16}(1,2,4,6,7),$ $C_{16}(2,3,4,5,6) = \Theta_{16,2,2}(C_{16}(1,2,4,6,7));$ 
	
	\item [\rm (9)]	$C_{16}(1,2,4,7,8),$ $C_{16}(2,3,4,5,8) = \Theta_{16,2,2}(C_{16}(1,2,4,7,8));$
	
	\item [\rm (10)]	$C_{16}(1,2,6,7,8),$ $C_{16}(2,3,5,6,8) = \Theta_{16,2,2}(C_{16}(1,2,6,7,8));$ 
	
	\item [\rm (11)]	$C_{16}(1,4,6,7,8),$ $C_{16}(3,4,5,6,8) = \Theta_{16,2,2}(C_{16}(1,4,6,7,8))$; and
	
	\item [\rm (12)]	$C_{16}(1,2,4,6,7,8)$, $C_{16}(2,3,4,5,6,8) = \Theta_{16,2,2}(C_{16}(1,2,4,6,7,8))$. 
\end{enumerate}
And among them the following 4 pairs are Adam's isomorphic.

\begin{enumerate}
	
	\item [\rm (4)]	$C_{16}(1,2,6,7),$ $C_{16}(3(1,2,6,7)) = C_{16}(5(1,2,6,7)) = C_{16}(2,3,5,6)$;
	
	\item [\rm (8)]	$C_{16}(1,2,4,6,7),$ $C_{16}(3(1,2,4,6,7)) = C_{16}(5(1,2,4,6,7)) = C_{16}(2,3,4,5,6)$; 
	
	\item [\rm (10)]	$C_{16}(1,2,6,7,8),$ $C_{16}(3(1,2,6,7,8)) = C_{16}(5(1,2,6,7,8)) = C_{16}(2,3,5,6,8)$; and
	
	\item [\rm (12)]	$C_{16}(1,2,4,6,7,8),$ $C_{16}(3(1,2,4,6,7,8)) = C_{16}(5(1,2,4,6,7,8)) = C_{16}(2,3,4,5,6,8)$.
\end{enumerate}
And all others are of Type-2 
isomorphic w.r.t. $m$ = 2 since 

		\begin{enumerate}
	\item [\rm (2)]	$\Theta_{16,2,2}(C_{16}(1,6,7))$ = $C_{16}(3,5,6)$. $\Rightarrow$ $C_{16}(1,6,7) \cong C_{16}(3,5,6)$. Also, 
	
	$Ad_{16}(C_{16}(1,6,7))$ = $\{C_{16}(1,6,7), C_{16}(2,3,5) \}$. $\Rightarrow$ $C_{16}(3,5,6) \notin Ad_{16}(C_{16}(1,6,7))$.
	
	$\Rightarrow$ $C_{16}(1,6,7)$ and $C_{16}(3,5,6)$ are Type-2 isomorphic w.r.t. $m$ = 2; 
	
	\item [\rm (3)]	$\Theta_{16,2,2}(C_{16}(1,2,4,7))$ = $C_{16}(2,3,4,5)$. $\Rightarrow$ $C_{16}(1,2,4,7) \cong C_{16}(2,3,4,5)$. Also, 
	
	$Ad_{16}(C_{16}(1,2,4,7))$ = $\{C_{16}(1,2,4,7), C_{16}(3,4,5,6) \}$. $\Rightarrow$ $C_{16}(2,3,4,5) \notin Ad_{16}(C_{16}(1,2,4,7))$.
	
	$\Rightarrow$ $C_{16}(1,2,4,7)$ and $C_{16}(2,3,4,5)$ are Type-2 isomorphic w.r.t. $m$ = 2;
	
	\item [\rm (5)]	$\Theta_{16,2,2}(C_{16}(1,2,7,8))$ = $C_{16}(2,3,5,8)$. $\Rightarrow$ $C_{16}(1,2,7,8) \cong C_{16}(2,3,5,8)$. Also, 
	
	$Ad_{16}(C_{16}(1,2,7,8))$ = $\{C_{16}(1,2,7,8), C_{16}(3,5,6,8) \}$. $\Rightarrow$ $C_{16}(2,3,5,8) \notin Ad_{16}(C_{16}(1,2,7,8))$.
	
	$\Rightarrow$ $C_{16}(1,2,7,8)$ and $C_{16}(2,3,5,8)$ are Type-2 isomorphic w.r.t. $m$ = 2;
	
	\item [\rm (6)]	$\Theta_{16,2,2}(C_{16}(1,4,6,7))$ = $C_{16}(3,4,5,6)$. $\Rightarrow$ $C_{16}(1,4,6,7) \cong C_{16}(3,4,5,6)$. Also, 
	
	$Ad_{16}(C_{16}(1,4,6,7))$ = $\{C_{16}(1,4,6,7), C_{16}(2,3,4,5) \}$. $\Rightarrow$ $C_{16}(3,4,5,6) \notin Ad_{16}(C_{16}(1,4,6,7))$.
	
	$\Rightarrow$ $C_{16}(1,4,6,7)$ and $C_{16}(3,4,5,6)$ are Type-2 isomorphic w.r.t. $m$ = 2;
	
	\item [\rm (7)]	$\Theta_{16,2,2}(C_{16}(1,6,7,8))$ = $C_{16}(3,5,6,8)$. $\Rightarrow$ $C_{16}(1,6,7,8) \cong C_{16}(3,5,6,8)$. Also, 
	
	$Ad_{16}(C_{16}(1,6,7,8))$ = $\{C_{16}(1,6,7,8), C_{16}(2,3,5,8) \}$. $\Rightarrow$ $C_{16}(3,5,6,8) \notin Ad_{16}(C_{16}(1,6,7,8))$.
	
	$\Rightarrow$ $C_{16}(1,6,7,8)$ and $C_{16}(3,5,6,8)$ are Type-2 isomorphic w.r.t. $m$ = 2;
	
	\item [\rm (9)]	$\Theta_{16,2,2}(C_{16}(1,2,4,7,8))$ = $C_{16}(2,3,4,5,8)$. $\Rightarrow$ $C_{16}(1,2,4,7,8) \cong C_{16}(2,3,4,5,8)$. Also, 	
	
	$Ad_{16}(C_{16}(1,2,4,7,8))$ = $\{C_{16}(1,2,4,7,8), C_{16}(3,4,5,6,8) \}$. 
	\\
	$\Rightarrow$ $C_{16}(2,3,4,5,8) \notin Ad_{16}(C_{16}(1,2,4,7,8))$.
	
	$\Rightarrow$ $C_{16}(1,2,4,7,8)$ and $C_{16}(2,3,4,5,8)$ are Type-2 isomorphic w.r.t. $m$ = 2 and
	
	\item [\rm (11)] $\Theta_{16,2,2}(C_{16}(1,4,6,7,8))$ = $C_{16}(3,4,5,6,8)$. $\Rightarrow$ $C_{16}(1,4,6,7,8) \cong C_{16}(3,4,5,6,8)$.  
	
	Also, $Ad_{16}(C_{16}(1,4,6,7,8))$ = $\{C_{16}(1,4,6,7,8), C_{16}(2,3,4,5,8) \}$. 
	
	$\Rightarrow$ $C_{16}(3,4,5,6,8) \notin Ad_{16}(C_{16}(1,4,6,7,8))$.
	
	$\Rightarrow$ $C_{16}(1,4,6,7,8)$ and $C_{16}(3,4,5,6,8)$ are Type-2 isomorphic w.r.t. $m$ = 2.  
\end{enumerate}

On the other hand, even though the following 3 pairs of circulant graphs are isomorphic, they are not Type-2 isomorphic because of the following reasons. 
\begin{enumerate}
	\item [\rm (13)]	$C_{16}(1,4,7)$ $\cong$ $C_{16}(3,4,5)$ since $ \Theta_{16,4,1}(C_{16}(1,4,7)) = C_{16}(3,4,5)$ and $C_{16}(3(1,4,7))$ = $C_{16}(3,4,5)$. Also, $m$ = 4 = $\gcd(16, 4)$. 
	
	\item [\rm (14)] $C_{16}(1,7,8)$ $\cong$ $C_{16}(3,5,8)$ and the two circulant graphs are Adam's isomorphic since  $C_{16}(3(1,7,8))$ = $C_{16}(3,5,8)$. Here, $\gcd(n, r)$ = $\gcd(16, 8)$ = 8.	
	
	\item [\rm (15)]	$C_{16}(1,4,7,8)$ $\cong$ $C_{16}(3,4,5,8)$ since $\Theta_{16,4,1}(C_{16}(1,4,7,8)) = C_{16}(3,4,5,8)$ and $C_{16}(3(1,4,7,8))$ = $C_{16}(3,4,5,8)$. Also, $\gcd(n, r)$ = $\gcd(16, 4)$ = 4.
\end{enumerate}

Thus there are 8 pairs of Type-2 isomorphic circulant graphs of order 16 and each pair is isomorphic of Type-2 w.r.t. $m$ = 2 only. \hfill $\Box$   

Next we find all possible Type-2 isomorphic circulant graphs of  order 27. We show that there are 12 triples of Type-2 isomorphic circulant graphs of order 27 and each triple is isomorphic of Type-2 w.r.t. $m$ = 3. We start with showing circulant graphs 

(1) $C_{27}(1,3,8,10)$, $C_{27}(2,3,7,11)$ and $C_{27}(3,4,5,13)$ are isomorphic of Type-2 w.r.t. $m$ = 3.  

\begin{prm}\quad \label{p6} {\rm Show that  $C_{27}(1,3,8,10)$, $C_{27}(2,3,7,11)$ and $C_{27}(3,4,5,13)$ are isomorphic of Type-2 w.r.t. $m$ = 3. }
\end{prm}
\noindent
{\bf Solution.}\quad Let $R_1$ = $\{1,3,8,10\}$, $R_2$ = $\{2,3,7,11\}$ and $R_3$ = $\{3,4,5,13\}$. Here, $n$ = 27 and $r$ = 3 so that $3\in R_1, R_2, R_3$ and $\gcd(n, r)$ = $\gcd(27, 3)$ = 3. Using the definition of $\Theta_{n,m,t}$, we get,

$\Theta_{27,3,1}(C_{27}(1,3,8,10))$ = $\Theta_{27,3,1}(C_{27}(1,3,8,10, 17,19,24,26))$ = $C_{27}(\Theta_{27,3,1}(1,3,8,10, 17,19,24,26))$ 

 \hspace{3.5cm} = $C_{27}(4,3,14,13, 23,22,24,5)$ = $C_{27}(3,4,5,13)$ and 

$\Theta_{27,3,2}(C_{27}(1,3,8,10))$ = $C_{27}(\Theta_{27,3,2}(1,3,8,10, 17,19,24,26))$ = $C_{27}(7,3,20,16, 2,25,24,11)$

 \hspace{3.5cm} = $C_{27}(2,3,7,11)$.
 
 $\Rightarrow$ $C_{27}(1,3,8,10)$ $\cong$ $C_{27}(3,4,5,13)$ and  $C_{27}(1,3,8,10)$ $\cong$ $C_{27}(2,3,7,11)$. 
 
 $\Rightarrow$ $C_{27}(1,3,8,10)$ $\cong$ $C_{27}(2,3,7,11)$  $\cong$ $C_{27}(3,4,5,13)$. 

Also, $Ad_{27}(C_{27}(1,3,8,10))$ = $Ad_{27}(C_{27}(1,3,8,10, 17,19,24,26))$ 

\hfill = $\{C_{27}(x(1,3,8,10, 17,19,24,26)): x = 1,2,4,5,7,8,10,11,13,14,16,17,19,20,22,23,25,26\}$ 

\hspace{.9cm} = $\{C_{27}(1,3,8,10, 17,19,24,26),$ $C_{27}(2,6,7,11, 16,20,25),$ $C_{27}(4,5,12,13, 14,15,22,23)\}$ 

\hspace{.9cm} = $\{C_{27}(1,3,8,10),$ $C_{27}(2,6,7,11),$ $C_{27}(4,5,$ $12,13)\}$.

 $\Rightarrow$ $C_{27}(2,3,7,11),C_{27}(3,4,5,13)\notin Ad_{27}($ $C_{27}(1,3,8,10))$.
 
 $\Rightarrow$ $C_{27}(1,3,8,10)$, $C_{27}(2,3,7,11)$ and $C_{27}(3,4,5,13)$ are isomorphic of Type-2 w.r.t. $m$ = 3. \hfill $\Box$
 
  In the next problem, using Remark \ref{r12} in the above Type-2 isomorphic circulant graphs, we find all Type-2 isomorphic circulant graphs of order 27.
  
 \begin{prm}\quad \label{p7} Show that there are 12 triples of Type-2 isomorphic circulant graphs of order 27 and they are all isomorphic of Type-2 w.r.t. $m$ = 3.
 \end{prm}
 \noindent
 {\bf Solution.}\quad Here, $n$ = 27. The possible values of $m > 1$ $\ni$ $m$ is a divisor of $\gcd(n, r)$ = $\gcd(27, r)$ are $m$ = 3, 9. $\gcd(27, 3)$ = 3 = $\gcd(27, 6)$ = $\gcd(27, 12)$ and $\gcd(27, 9)$ = 9. We start with finding isomorphic circulant graphs of Type-2 w.r.t. $m$ = 3. In Problem \ref{p6}, we proved that circulant graphs 
 
 (1) $C_{27}(1,3,8,10)$, $C_{27}(2,3,7,11)$ and $C_{27}(3,4,5,13)$ are isomorphic of Type-2 w.r.t. $m$ = 3. 
 
 Using Remark \ref{r12} in this triple of isomorphic circulant graphs of Type-2 w.r.t. $m$ = 3, we get the following triples of isomorphic circulant graphs of the form $C_{27}(R)$ and each of these triples are either Adam's isomorphic circulant graphs or Type-2 isomorphic circulant graphs w.r.t. $m$ = 3. 	
 
 \begin{enumerate}
 	
	\item [\rm (2)]	$C_{27}(1,6,8,10)$, $C_{27}(2,6,7,11)$, $C_{27}(4,5,6,13)$; 
			
	\item [\rm (3)]	$C_{27}(1,8,10,12)$, $C_{27}(2,7,11,12)$, $C_{27}(4,5,12,13)$;
		
	\item [\rm (4)]	$C_{27}(1,3,6,8,10)$, $C_{27}(2,3,6,7,11)$, $C_{27}(3,4,5,6,13)$; 
		
	\item [\rm (5)]	$C_{27}(1,3,8,9,10)$, $C_{27}(2,3,7,9,11)$, $C_{27}(3,4,5,9,13)$; 
		
	\item [\rm (6)]	$C_{27}(1,3,8,10,12)$, $C_{27}(2,3,7,11,12)$, $C_{27}(3,4,5,12,13)$;
		
	\item [\rm (7)]	$C_{27}(1,6,8,9,10)$, $C_{27}(2,6,7,9,11)$, $C_{27}(4,5,6,9,13)$;
		
	\item [\rm (8)]	$C_{27}(1,6,8,10,12)$, $C_{27}(2,6,7,11,12)$, $C_{27}(4,5,6,12,13)$; 
		
	\item [\rm (9)]	$C_{27}(1,8,9,10,12)$, $C_{27}(2,7,9,11,12)$, $C_{27}(4,5,9,12,13)$; 
		
	\item [\rm (10)]	$C_{27}(1,3,6,8,9,10)$, $C_{27}(2,3,6,7,9,11)$, $C_{27}(3,4,5,6,9,13)$;
		
	\item [\rm (11)]	$C_{27}(1,3,6,8,10,12)$, $C_{27}(2,3,6,7,11,12)$, $C_{27}(3,4,5,6,12,13)$;
		
	\item [\rm (12)]	$C_{27}(1,3,8,9,10,12)$, $C_{27}(2,3,7,9,11,12)$, $C_{27}(3,4,5,9,12,13)$; 
		
	\item [\rm (13)] $C_{27}(1,6,8,9,10,12)$, $C_{27}(2,6,7,9,11,12)$, $C_{27}(4,5,6,9,12,13)$. 
	\end{enumerate}	

	And among them the following triple is of Adam's isomorphic circulant graphs.

	(11) $C_{27}(1,3,6,8,10,12)$, $C_{27}(2,3,6,7,11,12) \cong C_{27}(2(1,3,6,8,10,12))$, 
	
	\hspace{3.8cm}  $C_{27}(3,4,5,6,12,13) \cong C_{27}(4(1,3,6,8,10,12))$.	
	
	\vspace{.1cm}	
	And all others are of Type-2 isomorphic w.r.t. $r = 3$ since 
	
	\begin{enumerate}
		\item [\rm (2)]	$\theta_{27,3,1}(C_{27}(1,6,8,10))$ = $C_{27}(4,5,6,13)$ and $\theta_{27,3,2}(C_{27}(1,6,8,10))$ = $C_{27}(2,6,7,11)$.
		
	 $\Rightarrow$ $C_{27}(1,6,8,10)$ $\cong$ $C_{27}(4,5,6,13)$ $\cong$ $C_{27}(2,6,7,11)$. 
		
	Also,	$Ad_{27}(C_{27}(1,6,8,10))$ = $\{C_{27}(1,6,8,10), C_{27}(2,7,11,12), C_{27}(3,4,5,13)\}$. 
		
	 $\Rightarrow$  $C_{27}(2,6,7,11),C_{27}(4,5,6,13)\notin Ad_{27}(C_{27}(1,6,8,10))$.
	 
	$\Rightarrow$ $C_{27}(1,6,8,10)$, $C_{27}(4,5,6,13)$,  $C_{27}(2,6,7,11)$ are isomorphic of Type-2 w.r.t. $m$ = 6.
		
	\item [\rm (3)]	$\theta_{27,3,1}(C_{27}(1,8,10,12))$ = $C_{27}(4,5,12,13)$ and $\theta_{27,3,2}(C_{27}(1,8,10,12))$ = $C_{27}(2,7,11,12)$.
		
	$\Rightarrow$ $C_{27}(1,8,10,12)$ $\cong$ $C_{27}(4,5,12,13)$ $\cong$ $C_{27}(2,7,11,12)$. 
		
	Also,	$Ad_{27}(C_{27}(1,8,10,12))$ = $\{C_{27}(1,8,10,12), C_{27}(2,3,7,11), C_{27}(4,5,6,13)\}$. 
		
	$\Rightarrow$  $C_{27}(2,7,11,12),C_{27}(4,5,12,13)\notin Ad_{27}(C_{27}(1,8,10,12))$.
		
	$\Rightarrow$ $C_{27}(1,8,10,12)$, $C_{27}(4,5,12,13)$,  $C_{27}(2,7,11,12)$ are isomorphic of Type-2 w.r.t. $m$ = 3.
		
\item [\rm (4)]	$\theta_{27,3,1}(C_{27}(1,3,6,8,10))$ = $C_{27}(3,4,5,6,13)$ and $\theta_{27,3,2}(C_{27}(1,3,6,8,10))$ = $C_{27}(2,3,6,7,11)$.

$\Rightarrow$ $C_{27}(1,3,6,8,10)$ $\cong$ $C_{27}(3,4,5,6,13)$ $\cong$ $C_{27}(2,3,6,7,11)$. 

Also,	$Ad_{27}(C_{27}(1,3,6,8,10))$ = $\{C_{27}(1,3,6,8,10), C_{27}(2,6,7,11,12), C_{27}(3,4,5,12,13)\}$. 

$\Rightarrow$  $C_{27}(2,3,6,7,11),C_{27}(3,4,5,6,13)\notin Ad_{27}(C_{27}(1,3,6,8,10))$.

$\Rightarrow$ $C_{27}(1,3,6,8,10)$, $C_{27}(3,4,5,6,13)$,  $C_{27}(2,3,6,7,11)$ are Type-2 isomorphic w.r.t. $m$ = 3.

\item [\rm (5)]	$\theta_{27,3,1}(C_{27}(1,3,8,9,10))$ = $C_{27}(3,4,5,9,13)$ and $\theta_{27,3,2}(C_{27}(1,3,8,9,10))$ = $C_{27}(2,3,7,9,11)$.

$\Rightarrow$ $C_{27}(1,3,8,9,10)$ $\cong$ $C_{27}(3,4,5,9,13)$ $\cong$ $C_{27}(2,3,7,9,11)$. 

Also,	$Ad_{27}(C_{27}(1,3,8,9,10))$ = $\{C_{27}(1,3,8,9,10), C_{27}(2,6,7,9,11), C_{27}(4,5,9,12,13)\}$. 

$\Rightarrow$  $C_{27}(2,3,7,9,11),C_{27}(3,4,5,9,13)\notin Ad_{27}(C_{27}(1,3,8,9,10))$.

$\Rightarrow$ $C_{27}(1,3,8,9,10)$, $C_{27}(3,4,5,9,13)$,  $C_{27}(2,3,7,9,11)$ are Type-2 isomorphic w.r.t. $m$ = 3.

\item [\rm (6)]	$\theta_{27,3,1}(C_{27}(1,3,8,10,12))$ = $C_{27}(3,4,5,12,13)$ and $\theta_{27,3,2}(C_{27}(1,3,8,10,12))$ = $C_{27}(2,3,7,11,12)$.

$\Rightarrow$ $C_{27}(1,3,8,10,12)$ $\cong$ $C_{27}(3,4,5,12,13)$ $\cong$ $C_{27}(2,3,7,11,12)$. 

Also,	$Ad_{27}(C_{27}(1,3,8,10,12))$ = $\{C_{27}(1,3,8,10,12),  C_{27}(2,3,6,7,11), C_{27}(4,5,6,12,13)\}$. 

$\Rightarrow$  $C_{27}(2,3,7,11,12), C_{27}(3,4,5,12,13)\notin Ad_{27}(C_{27}(1,3,8,10,12))$.
\\
$\Rightarrow$ $C_{27}(1,3,8,10,12)$, $C_{27}(3,4,5,12,13)$,  $C_{27}(2,3,7,11,12)$ are Type-2 isomorphic w.r.t. $m$ = 3.

\item [\rm (7)]	$\theta_{27,3,1}(C_{27}(1,6,8,9,10))$ = $C_{27}(4,5,6,9,13)$ and $\theta_{27,3,2}(C_{27}(1,6,8,9,10))$ = $C_{27}(2,6,7,9,11)$.

$\Rightarrow$ $C_{27}(1,6,8,9,10)$ $\cong$ $C_{27}(4,5,6,9,13)$ $\cong$ $C_{27}(2,6,7,9,11)$. 

Also,	$Ad_{27}(C_{27}(1,6,8,9,10))$ = $\{C_{27}(1,6,8,9,10), C_{27}(2,7,9,11,12), C_{27}(3,4,5,9,13)\}$. 

$\Rightarrow$  $C_{27}(2,6,7,9,11),C_{27}(4,5,6,9,13)\notin Ad_{27}(C_{27}(1,6,8,9,10))$.

$\Rightarrow$ $C_{27}(1,6,8,9,10)$, $C_{27}(4,5,6,9,13)$,  $C_{27}(2,6,7,9,11)$ are Type-2 isomorphic w.r.t. $m$ = 3.

\item [\rm (8)]	$\theta_{27,3,1}(C_{27}(1,6,8,10,12))$ = $C_{27}(4,5,6,12,13)$ and $\theta_{27,3,2}(C_{27}(1,6,8,10,12))$ = $C_{27}(2,6,7,11,12)$.

$\Rightarrow$ $C_{27}(1,6,8,10,12)$ $\cong$ $C_{27}(4,5,6,12,13)$ $\cong$ $C_{27}(2,6,7,11,12)$. 

Also,	$Ad_{27}(C_{27}(1,6,8,10,12))$ = $\{C_{27}(1,6,8,10,12), C_{27}(2,3,7,11,12), C_{27}(3,4,5,6,13)\}$. 

$\Rightarrow$  $C_{27}(2,6,7,11,12), C_{27}(4,5,6,12,13)\notin Ad_{27}(C_{27}(1,6,8,10,12))$.
\\
$\Rightarrow$ $C_{27}(1,6,8,10,12)$, $C_{27}(4,5,6,12,13)$,  $C_{27}(2,6,7,11,12)$ are Type-2 isomorphic w.r.t. $m$ = 3.

\item [\rm (9)] $\theta_{27,3,1}(C_{27}(1,8,9,10,12))$ = $C_{27}(4,5,9,12,13)$ and $\theta_{27,3,2}(C_{27}(1,8,9,10,12))$ = $C_{27}(2,7,9,11,12)$.

$\Rightarrow$ $C_{27}(1,8,9,10,12)$ $\cong$ $C_{27}(4,5,9,12,13)$ $\cong$ $C_{27}(2,7,9,11,12)$. 

Also,	$Ad_{27}(C_{27}(1,8,9,10,12))$ = $\{C_{27}(1,8,9,10,12), C_{27}(2,3,7,9,11), C_{27}(4,5,6,9,13)\}$. 

$\Rightarrow$  $C_{27}(2,7,9,11,12), C_{27}(4,5,9,12,13)\notin Ad_{27}(C_{27}(1,8,9,10,12))$.
\\
$\Rightarrow$ $C_{27}(1,8,9,10,12)$, $C_{27}(4,5,9,12,13)$,  $C_{27}(2,7,9,11,12)$ are Type-2 isomorphic w.r.t. $m$ = 3.
		
\item [\rm (10)] $\theta_{27,3,1}(C_{27}(1,3,6,8,9,10))$ = $C_{27}(3,4,5,6,9,13)$ and 

 $\theta_{27,3,2}(C_{27}(1,3,6,8,9,10))$ = $C_{27}(2,3,6,7,9,11)$.

$\Rightarrow$ $C_{27}(1,3,6,8,9,10)$ $\cong$ $C_{27}(3,4,5,6,9,13)$ $\cong$ $C_{27}(2,3,6,7,9,11)$. Also,

$Ad_{27}(C_{27}(1,3,6,8,9,10))$ = $\{C_{27}(1,3,6,8,9,10), C_{27}(2,6,7,9,11,12), C_{27}(3,4,5,9,12,13)\}$. 

$\Rightarrow$  $C_{27}(2,3,6,7,9,11),C_{27}(3,4,5,6,9,13)\notin Ad_{27}(C_{27}(1,3,6,8,9,10))$.

$\Rightarrow$ $C_{27}(1,3,6,8,9,10)$, $C_{27}(3,4,5,6,9,13)$,  $C_{27}(2,3,6,7,9,11)$ are Type-2 isomorphic w.r.t. $m$ = 3.
	
\item [\rm (12)] $\theta_{27,3,1}(C_{27}(1,3,8,9,10,12))$ = $C_{27}(3,4,5,9,12,13)$ and

 $\theta_{27,3,2}(C_{27}(1,3,8,9,10,12))$ = $C_{27}(2,3,7,9,11,12)$.

$\Rightarrow$ $C_{27}(1,3,8,9,10,12)$ $\cong$ $C_{27}(3,4,5,9,12,13)$ $\cong$ $C_{27}(2,3,7,9,11,12)$. Also,

$Ad_{27}(C_{27}(1,3,8,9,10,12))$ = $\{C_{27}(1,3,8,9,10,12), C_{27}(2,3,6,7,9,11), C_{27}(4,5,6,9,12,13)\}$. 

$\Rightarrow$  $C_{27}(2,3,7,9,11,12), C_{27}(3,4,5,9,12,13)\notin Ad_{27}(C_{27}(1,3,8,9,10,12))$.

$\Rightarrow$ $C_{27}(1,3,8,9,10,12)$, $C_{27}(3,4,5,9,12,13)$,  $C_{27}(2,3,7,9,11,12)$ are Type-2 isomorphic w.r.t. $m$ = 3.
		
\item [\rm (13)] $\theta_{27,3,1}(C_{27}(1,6,8,9,10,12))$ = $C_{27}(4,5,6,9,12,13)$ and 

$\theta_{27,3,2}(C_{27}(1,6,8,9,10,12))$ = $C_{27}(2,6,7,9,11,12)$.

$\Rightarrow$ $C_{27}(1,6,8,9,10,12)$ $\cong$ $C_{27}(4,5,6,9,12,13)$ $\cong$ $C_{27}(2,6,7,9,11,12)$. Also, 

	$Ad_{27}(C_{27}(1,6,8,9,10,12))$ = $\{C_{27}(1,6,8,9,10,12), C_{27}(2,3,7,9,11,12), C_{27}(3,4,5,6,9,13)\}$. 

$\Rightarrow$  $C_{27}(2,6,7,9,11,12), C_{27}(4,5,6,9,12,13)\notin Ad_{27}(C_{27}(1,6,8,9,10,12))$.

$\Rightarrow$ $C_{27}(1,6,8,9,10,12)$, $C_{27}(4,5,6,9,12,13)$,  $C_{27}(2,6,7,9,11,12)$ are Type-2 isomorphic w.r.t. $m$ = 3.
\end{enumerate}	

On the other hand, even though the following 3 circulant graphs,

	$C_{27}(1,8,9,10)$, $C_{27}(2,7,9,11)$ and $C_{27}(4,5,9,13)$
 are Adam's isomorphic and thereby they are not Type-2 isomorphic since 
 	
  $C_{27}(2(1,8,9,10))$ = $C_{27}(2,7,9,11)$ and $C_{27}(4(1,8,9,10))$ = $C_{27}(4,5,9,13)$. 
  
  Thus, there are 12 triples of Type-2 isomorphic circulant graphs of order 27 and all these triples of isomorphic circulant graphs are Type-2 w.r.t. m = 3. \hfill $\Box$

\section{Results on Type-2 isomorphic circulant graphs w.r.t. $m$ = 2}

 In this section, we present some important result on isomorphic circulant graphs of Type-2 w.r.t. $m$ = 2. We begin with  definition of the {\it symmetric equidistance condition} w.r.t. a vertex of $C_n(R)$.

\begin{definition}{\rm \cite{vc13}}\quad \label{a14} For a set $R = \{r_1,r_2,\ldots,r_k\}$ in $C_n(R)$, the {\it symmetric equidistance condition} w.r.t. $v_i$ is that  $v_{i}$ is adjacent to $v_{i+j}$  if and only if $v_{i}$ is adjacent to $v_{n-j+i}$, using subscript arithmetic modulo $n,$ $0 \leq i,j \leq n-1$. 
\end{definition} 

\begin{theorem}{\rm \cite{vc13} \quad \label{ab14} For a set $R$ = $\{r_1,r_2,\ldots,r_k\}  \subseteq [1, \frac{n}{2}],$ $r\in R$, $m > 1$ is a divisor of $\gcd(n, r)$, $1 \leq i \leq k$ and $0 \leq t \leq \frac{n}{m}-1,$ $\Theta_{n,m,t}(C_n(R))$ = $C_n(S)$ for some $S \subseteq [1, \frac{n}{2}]$ if and only if $\Theta_{n,m,t}(C_n(R))$ satisfies the symmetric equidistance condition w.r.t. $v_0.$}
\end{theorem}
\begin{proof}\quad Obviously, every circulant graph $C_n(R)$ for a set $R$ = $\{r_1,r_2,$ $\ldots,r_k\}$ satisfies the symmetric equidistance condition w.r.t. $v_0.$ 
	
	Conversely, let $\Theta_{n,m,t}(C_n(R))$ satisfy the symmetric equidistance condition w.r.t. $v_0.$ We have to prove that $\Theta_{n,m,t}(C_n(R))$ = $C_n(S)$ for some $S \subseteq [1, \frac{n}{2}].$ In $C_n(R),$ $v_q v_{q+s} \in E(\Gamma_j)$ if and only if $v_q v_{q+s} v_{q+2s} ... v_q \in \Gamma_j,$ using subscript arithmetic modulo $n$ where $\Gamma_j$ is the periodic cycle of period $s$ and of length $\frac{n}{\gcd(n, s)}$ in $C_n(R)$ and contains $v_j,$ $0 \leq j \leq$ $m-1,$ $r,s \in R$ and $0 \leq q \leq n-1.$ By the transformation $\Theta_{n,m,t},$ for some $t,$ each $\Gamma_j$ simply rotates $tm$ positions (points) w.r.t. $\Gamma_{j-1}$ in the regular $n$-gon in the clockwise direction and thereby the vertices of each $V_j$ = $V(\Gamma_j)$ are taking the positions of the vertices of $V_j$ only, $0 \leq tm \leq \frac{n}{2}$ and $0 \leq j \leq m-1.$ Thus by this transformation the consecutive edges, except the last one, of any periodic cycle of $C_n(R)$ connecting different $\Gamma_j$'s (not belonging to $\Gamma_j$), starting from any vertex of $\Gamma_0,$ change uniformly, $j$ = $0,1, \ldots, m-1.$ Since the transformed graph satisfies the symmetric equidistance condition w.r.t. $v_0$ (and hence w.r.t. each vertices of $\Gamma_0$), the resultant changes on the last edges of each periodic cycle (each starting from $\Gamma_0$) of $C_n(R)$ also follows the same uniform (or periodic) movements and thereby all the transformed cycles corresponding to the periodic cycles of $C_n(R)$ are also periodic cycles in the transformed graph $\Theta_{n,m,t}(C_n(R))$ and thereby the resultant graph $\Theta_{n,m,t}(C_n(R))$ = $C_n(S)$ for some $S \subseteq [1, \frac{n}{2}].$ Hence the result. 
\end{proof}
  
The following results on isomorphic circulant graphs of Type-2 w.r.t. $r$ = 2 are obtained in \cite{v17,v20}. 

\begin{theorem}{\rm \cite{v20}}\quad \label{a17b} {\rm For $n,s \in \mathbb{N}$, $1 \leq 2s-1 \leq 2n-1$, $R$ = $\{2,2s-1,4n-(2s-1)\}$ and $S$ = $\{ 2,2n-(2s-1),2n+2s-1 \}$,  $\Theta_{8n,2,n}(C_{8n}(R))$ = $C_{8n}(S)$ = $\Theta_{8n,2,3n}(C_{8n}(R))$ and $\Theta_{8n,2,n}(C_{8n}(S))$ = $C_{8n}(R)$ = $\Theta_{8n,2,3n}$ $(C_{8n}(S)).$ }
\end{theorem}
\begin{proof}\quad Using the definition of $\Theta_{n,r,t},$ we get, $\Theta_{8n,2,n}(2)$ = $2,$ $\Theta_{8n,2,n}(2s$ $-1)$ = $2n+2s-1,$ $\Theta_{8n,2,n} (4n-(2s-1))$ = $6n-(2s-1),$ $\Theta_{8n,2,n}(4n+(2s-1))$ = $6n+2s-1,$ $\Theta_{8n,2,n}(8n-(2s-1))$ = $2n-(2s-1),$ $\Theta_{8n,2,n}(8n-2)$ = $8n-2,$ $1 \leq 2s-1 \leq 2n-1$. Also, $8n- \Theta_{8n,2,n}(4n-(2s-1))$ = $\Theta_{8n,2,n}(2s-1),$ $8n- \Theta_{8n,2,n}(8n-2)$ = $\Theta_{8n,2,n}(2)$ and $8n- \Theta_{8n,2,n}(8n-(2s-1))$ = $\Theta_{8n,2,n}(4n+(2s-1))$. Thus $\Theta_{8n,2,n}(C_{8n}(R))$ for $R$ = $\{2,2s-1, 4n-(2s-1)\}$ satisfies the symmetric equidistance condition w.r.t.  $v_0$ and $\Theta_{8n,2,n}(C_{8n}(R))$ = $C_{8n}(S)$ where $S$ = $\{2,2n-(2s-1), 2n+2s-1\},$ $1 \leq 2s-1 \leq 2n-1$ and $n,s \in \mathbb{N}$. Similarly, we can prove $\Theta_{8n,2,3n}(C_{8n}(R))$ = $C_{8n}(S)$ and $\Theta_{8n,2,n}(C_{8n}(S))$ = $C_{8n}(R)$ = $\Theta_{8n,2,3n}(C_{8n}(S))$ where $R$ = $\{2,2s-1, 4n-(2s-1)\},$ $S$ = $\{2, 2n-(2s-1), 2n+2s-1\},$ $1 \leq 2s-1 \leq 2n-1$ and $n,s \in \mathbb{N}$. Hence the result. 
\end{proof}

\begin{theorem}{\rm \cite{v20}}\quad \label{a17c} {\rm For $n \geq 2$, $1 \leq 2s-1 \leq 2n-1$, $n \neq 2s-1$, $R$ = $\{2,2s-1, 4n-(2s-1)\}$ and $S$ = $\{ 2,$ $2n-(2s-1),$ $2n+2s-1 \},$ $\Theta_{8n,2,n}(C_{8n}(R))$ = $C_{8n}(S)$ = $\Theta_{8n,2,3n}(C_{8n}(R)),$ $\Theta_{8n,2,n}(C_{8n}(S))$ = $C_{8n}(R)$ = $\Theta_{8n,2,3n}(C_{8n}(S))$ and circulant graphs $C_{8n}(R)$ and $C_{8n}(S)$ are Type-2 isomorphic  w.r.t. $m$ = 2. When $n$ = $2s-1$, the two circulant graphs are the same. }
\end{theorem}
\begin{proof}\quad When $R$ = $\{ 2,2s-1,4n-(2s-1)\}$ and $S$ = $\{ 2,2n-(2s-1),2n+2s-1 \},$ $\Theta_{8n,2,n}(C_{8n}(R))$ = $C_{8n}(S)$ = $\Theta_{8n,2,3n}(C_{8n}(R))$ and $\Theta_{8n,2,n}(C_{8n}(S))$ = $C_{8n}(R)$ = $\Theta_{8n,2,3n}(C_{8n}(S)),$ using Theorem \ref{a17b} where $1 \leq 2s-1 \leq 2n-1$ and $n,s \in \mathbb{N}$. The two circulant graphs $C_{8n}(R)$ and $C_{8n}(S)$ are the same when $2s-1$ = $2n-(2s-1)$ and $4n-(2s-1)$ = $2n+2s-1,$ $1 \leq 2s-1 \leq 2n-1$. That is when $n$ = $2s-1$, the graphs become $C_{8n}(2,n,3n),$ $n \in \mathbb{N}.$ When $n$ = $s$ = $1$, the graphs become $C_8(1,2,3)$. For $n \neq 2s-1$, $n,s\in \mathbb{N}$ and $n \geq 2$, the set of jump sizes of the two graphs are different and thereby the two circulant graphs $C_{8n}(R)$ and $C_{8n}(S)$ are not same.   \\
\noindent
{\it \bf {Claim:}} $C_{8n}(R)$ and $C_{8n}(S)$ are of Type-2 isomorphic w.r.t. $r$ = 2 when $R$ = $\{2,2s-1,4n-(2s-1)\},$ $S$ = $\{2,2n-(2s-1),2n+2s-1\},$ $n \neq 2s-1,$ $n \geq 2$ and $n,s \in \mathbb{N}$.
	
	If not, graphs $C_{8n}(R)$ and $C_{8n}(S)$ are Adam's isomorphic. Then, there exists $s' \in \mathbb{N}$ such that $\gcd(8n,2s'-1)$ = 1 and $C_{8n}((2s'-1)R)$ = $C_{8n}(S)$ which implies, $(2s'-1)\{ 2,2s-1,4n-(2s-1),4n+2s-1,$ $8n-(2s-1), 8n-2 \}$ = $\{ 2,2n-(2s-1),$ $2n+2s-1, 6n-(2s-1), 6n+2s-1,$ $8n-2 \}$, using arithmetic modulo $8n,$ $1 \leq 2s'-1 \leq 8n-1$. Now, $2(2s'-1),$ $(8n-2)(2s'-1), 2+8np_1$ and $8n-2+8np_2$ are the only even numbers in the two sets for some $p_1, p_2 \in \mathbb{N}_0$. The following two cases arise.  
	
	\noindent
	{\bf Case (i):} $2(2s'-1)$ = $2+8np_1,$ $1 \leq 2s'-1 \leq 8n-1$ and $p_1 \in \mathbb{N}_0$.  
	
	In this case, the possible values of $p_1$ are $0$ or $1$ since $1 \leq 2s'-1 \leq 8n-1$ and $s',n \in \mathbb{N}$. When $p_1$ = $0,$ $2s'-1$ = $1$ and so we get the same graph only. When $p_1$ = $1,$ $2s'-1$ = $4n+1$. In this case, the two graphs are the same. The jump sizes of the circulant graph corresponding to Adam's isomorphism when $2s'-1$ = $4n+1$ are given in Table 6. See Table 6. \\ 
	\noindent
	\textbf{Case (ii):} $2(2s'-1)$ = $8n-2+8np_2,$ $1 \leq 2s'-1 \leq 8n-1$ and $p_2 \in \mathbb{N}_0$.  
	
	In this case, the possible values of $p_2$ are $0$ or $1.$ When $p_2$ = $0,$ $2s'-1$ = $4n-1$. When $p_2$ = $1,$ $2s'-1$ = $8n-1.$ In the case $2s'-1 = 4n-1$ as well as when $2s'-1$ = $8n-1$ (which is same as when $2s'-1 = 8n-(8n-1)$ = $1$), we get the same circulant graph $C_{8n}(R)$ only. The jump sizes of the circulant graph corresponding to Adam's isomorphism when $2s'-1$ = $4n-1$ as well as when $2s'-1$ = $8n-(4n-1)$ = $4n+1$ are given in Table 6. See Table 6.

	\begin{table}
		\caption{{\small Calculation of $s(2s'-1)$ under arithmetic modulo $8n.$}}
		\begin{center}
			\scalebox{0.85}{
				\begin{tabular}{||c||*{5}{c|}c||}\hline \hline
					\backslashbox{Multiplier \\ $2s'-1$}{Jump size \\ $s$}
					& 2 & $2s-1$ & $4n-(2s-1)$ & $4n+2s-1$ & $8n-(2s-1)$ & $8n-2$\\\hline \hline
					& &  &   &  &  &  \\
					$4n-1$ & $8n-2$ & $4n-(2s-1)$ & $2s-1$ & $8n-(2s-1)$ & $4n+2s-1$ & 2  \\\hline
					& &  &   &  &  &  \\
					$4n+1$ & 2 & $4n+2s-1$ & $8n-(2s-1)$ & $2s-1$ & $4n-(2s-1)$ & $8n-2$ \\\hline
					& &  &   &  &  &  \\
					$8n-1$ & $8n-2$ & $8n-(2s-1)$ & $4n+2s-1$ & $4n-(2s-1)$ & $2s-1$ & 2  \\\hline \hline 
			\end{tabular}}
		\end{center}
	\end{table}
	
	Thus, the isomorphic circulant graphs $C_{8n}(R)$ and $C_{8n}(S)$ for $R$ = $\{ 2,$ $2s-1,$ $4n-(2s-1) \},$ $S$ = $\{ 2,$ $2n-(2s-1),$ $2n+2s-1 \},$ $1 \leq 2s-1 \leq 2n-1,$ $n \geq 2$  and  $n \neq 2s-1$ (and thereby $R \neq S,$ $2n-(2s-1) \neq 2s-1$ and $2n+2s-1 \neq 4n-(2s-1),$ $1 \leq 2s-1 \leq 2n-1$) are different, not of Type-1 and $\Theta_{8n,2,n}(C_{8n}(R))$ = $C_{8n}(S).$ And hence the result follows. 
\end{proof}

\begin{theorem} \label{a17d} {\rm Let $n \geq 2,$ $k \geq 3,$ $1 \leq 2s-1 \leq 2n-1,$ $n \neq 2s-1,$ $p_1,p_2,\ldots,p_{k-2} \in \mathbb{N}$ $\ni$ $\gcd(p_1,p_2,\ldots,p_{k-2})$ = 1,  $R$ = $\{ 2s-1,$ $4n-2s+1,$ $2p_1,$ $2p_2,$ $\ldots,$ $2p_{k-2} \}$, $S$ = $\{2n-(2s-1),$ $2n+2s-1,$ $2p_1,2p_2,\ldots,2p_{k-2}\}$, $2y\in R,S$ and $\gcd(4n, 2y)$ = 2. Then, for a given set of values of $n,y,p_1,p_2,\ldots,p_{k-2}$ and $s$, $C_{8n}(R)$ and $C_{8n}(S)$ are isomorphic of either Adam's or Type-2 w.r.t. $2y$. Moreover, for all such possible values of $p_1,p_2,\ldots,p_{k-2}$ and $y$, the set $\{ C_n(S) = \Theta_{8n,2,n}(C_{8n}(R)): p_1,p_2,\ldots,p_{k-2} \in \mathbb{N}\}$ contains all possible isomorphic circulant graphs of $C_{8n}(R)$ of Type-2 w.r.t. 2. }
\end{theorem}
\begin{proof}\quad Using Theorem \ref{a17c}, $\Theta_{8n,2,n}(C_{8n}(R))$ = $C_{8n}(S)$ = $\Theta_{8n,2,3n}(C_{8n}(R)),$ $\Theta_{8n,2,n}(C_{8n}(S))$ = $C_{8n}(R)$ = $\Theta_{8n,2,3n}(C_{8n}(S))$ and circulant graphs $C_{8n}(R)$ and $C_{8n}(S)$ are Type-2 isomorphic  w.r.t. $r = 2$ when $R$ = $\{2,2s-1,4n-(2s-1)\},$ $S$ = $\{ 2, 2n-(2s-1), 2n+2s-1 \},$ $n$ $\neq$ $2s-1,$ $n$ $\geq$ $2$ and $n \in \mathbb{N}.$ And using Remark \ref{r12}, for a given set of values of $n,s,p_1,p_2,\ldots,p_{k-2}$ and $y$,   $\Theta_{8n,2,n}(C_{8n}(R))$ = $C_{8n}(S)$ = $\Theta_{8n,2,3n}(C_{8n}(R)),$ $\Theta_{8n,2,n}(C_{8n}(S))$ = $C_{8n}(R)$ = $\Theta_{8n,2,3n}(C_{8n}(S))$ and the isomorphic circulant graphs $C_{8n}(R)$ and $C_{8n}(S)$ are either Adam's or Type-2 w.r.t. $m$ = 2 when $R$ = $\{ 2s-1,$ $4n-2s+1,$ $2p_1,$ $2p_2,$ $\ldots,$ $2p_{k-2} \}$, $S$ = $\{2n-(2s-1),2n+2s-1, 2p_1,2p_2,\ldots,2p_{k-2}\},$ $n \geq 2,$ $k \geq 3,$ $1 \leq 2s-1 \leq 2n-1,$ $n \neq 2s-1,$ $\gcd(4n, 2y)$ = $m$ = 2,  $\gcd(p_1,p_2,\ldots,p_{k-2})$ = 1, $2y\in R,S$ and $n,s, y,p_1,p_2,\ldots,p_{k-2} \in \mathbb{N}$.
	
Moreover, for all such possible values of $p_1,p_2,\ldots,p_{k-2}$ and $y$, the set $\{ C_n(S) = \Theta_{8n,2,n}(C_{8n}(R)):$ $p_1,p_2,\ldots,p_{k-2} \in \mathbb{N}\}$ contains all possible isomorphic circulant graphs of $C_{8n}(R)$ of Type-2 w.r.t. $m$ = 2.  Hence we get the result. 
\end{proof}

\begin{cor}\quad \label{a17e} {\rm Let $n \geq 2,$ $k \geq 3,$ $1 \leq 2s-1 \leq 2n-1,$ $n \neq 2s-1,$ $p_1,p_2,\ldots,p_{k-2} \in \mathbb{N}$ $\ni$ $\gcd(p_1,p_2,\ldots,p_{k-2})$ = 1,  $R$ = $\{ 2s-1,$ $4n-2s+1,$ $2p_1,$ $2p_2,$ $\ldots,$ $2p_{k-2} \}$, $S$ = $\{2n-(2s-1),$ $2n+2s-1,$ $2p_1,2p_2,\ldots,2p_{k-2}\}$, $2y\in R,S$ and $\gcd(4n, 2y)$ = 2. Then, for a given set of values of $n,y,p_1,p_2,\ldots,p_{k-2}$ and $s$, $C_{8n}(R)$ and $C_{8n}(S)$ are either Adam's isomorphic or without CI-property.  \hfill  $\Box$}
\end{cor}
When the order of the circulant graph is $2^n$, corresponding to the above results, we obtain the following, $n \in \mathbb{N}$.

\begin{cor}{\rm \label{a17f} Let $n \geq 4,$ $k \geq 3,$ $1 \leq 2s-1 \leq 2^{n-2}-1,$ $n,s\in \mathbb{N}$, $R$ = $\{2, 2s-1, 2^{n-1}-(2s-1)\}$ and $S$ = $\{2, 2^{n-2}-(2s-1), 2^{n-2}+2s-1\}$. Then $C_{2^n}(R)$ and $C_{2^n}(S)$ are isomorphic of Type-2 w.r.t. $m$ = 2. }
\end{cor}
\begin{proof} The proof follows from Theorem \ref{a17c}. 
\end{proof}

\begin{cor}{\rm \label{a17fb} Let $n \geq 2,$ $k \geq 3,$ $1 \leq 2s-1 \leq 2^{n-2}-1$, $p_1,p_2,\ldots,p_{k-2} \in \mathbb{N}$ $\ni$ $\gcd(p_1,p_2,\ldots,p_{k-2})$ = 1,  $R$ = $\{ 2s-1,$ $2^{n-1}-(2s-1),$ $2p_1,$ $2p_2,$ $\ldots,$ $2p_{k-2} \}$, $S$ = $\{2^{n-2}-(2s-1), 2^{n-2}+2s-1,$ $2p_1,2p_2,\ldots,2p_{k-2}\}$, $2y\in R,S$ and $\gcd(4n, 2y)$ = 2. Then, for a given set of values of $n,y,p_1,p_2,\ldots,p_{k-2}$ and $s$, $C_{2^n}(R)$ and $C_{2^n}(S)$ are either Adam's isomorphic or Type-2 isomorphic w.r.t. $2y$. Moreover, for all such possible values of $p_1,p_2,\ldots,p_{k-2}$ and $y$, the set $\{ C_{2^n}(S) = \Theta_{{2^n},2,{2^{n-3}}}(C_{2^n}(R)): p_1,p_2,\ldots,p_{k-2} \in \mathbb{N}\}$ contains all possible isomorphic circulant graphs of $C_{2^n}(R)$ of Type-2 w.r.t. $r$ = $2y$. }
\end{cor}
\begin{proof} The proof follows from Theorem \ref{a17d}. 
\end{proof}

\begin{cor} \label{a17g} {\rm Let $n \geq 2,$ $k \geq 3,$ $1 \leq 2s-1 \leq 2^{n-2}-1$, $p_1,p_2,\ldots,p_{k-2} \in \mathbb{N}$ $\ni$ $\gcd(p_1,p_2,\ldots,p_{k-2})$ = 1,  $R$ = $\{ 2s-1,$ $2^{n-1}-(2s-1),$ $2p_1,$ $2p_2,$ $\ldots,$ $2p_{k-2} \}$, $S$ = $\{2^{n-2}-(2s-1), 2^{n-2}+2s-1,$ $2p_1,2p_2,\ldots,2p_{k-2}\}$, $2y\in R,S$ and $\gcd(4n, 2y)$ = 2. Then, for a given set of values of $n,y,p_1,p_2,\ldots,p_{k-2}$ and $s$, $C_{2^n}(R)$ and $C_{2^n}(S)$ are isomorphic of either Adam's or  without CI-property. \hfill $\Box$}
\end{cor}

For a given circulant graph $C_n(R)$ with $R$ = $\{2,2s-1, 2s'-1\}$, the following theorem gives conditions on $n, t$ and jump sizes for which the transformed graph $\Theta_{n,2,t}(C_n(R))$ and $C_n(R)$ are isomorphic of Type-2 w.r.t. $m$ = 2. Proof given here is different from the one given in \cite{v20}.

\begin{theorem}{\rm \cite{v20}}\quad \label{a18} {\rm For $n \geq 2,$ $1 \leq 2s-1 < 2s'-1 \leq [\frac{n}{2}],$ $0 \leq t \leq [\frac{n}{2}],$ $R$ = $\{2,2s-1, 2s'-1\}$ and $n,s,s'\in \mathbb{N},$ if $\Theta_{n,2,t}(C_n(R))$ and $C_n(R)$ are  isomorphic circulant graphs of Type-2 w.r.t. $m$ = 2 for some $t,$ then $n \equiv 0~(mod ~ 8),$ $2s-1+2s'-1$ = $\frac{n}{2},$ $2s-1 \neq \frac{n}{8},$ $t$ = $\frac{n}{8}$ or $\frac{3n}{8},$ $1 \leq 2s-1 \leq \frac{n}{4}$ and $n \geq 16.$ In particular, $\Theta_{8n,2,n}(C_{8n}(R))$ = $C_{8n}(S)$ = $\Theta_{8n,2,3n}(C_{8n}(R))$ and $C_{8n}(R)$ and $C_{8n}(S)$ are Type-2 isomorphic w.r.t. $m$ = 2 when $R$ = $\{2, 2s-1, 4n-(2s-1)\}$, $S$ = $\{2, 2n-(2s-1), 2n+2s-1\}$, $n\geq 2$ and $n,s\in \mathbb{N}$. }
\end{theorem}
	
\begin{proof}\quad Using the definition of $\Theta_{n,r,t},$ we have $m$ = $2$ is a divisor of $\gcd(n,r)$ = $\gcd(n,2)$ which implies, $n$ is even. Let $n$ = $2a$ and so the circulant graph $C_n(R)$ = $C_{2a}(R)$ and $\Theta_{n,r,t}$ = $\Theta_{2a,2,t}$ where $R$ = $\{ 2, 2s-1, 2s'-1 \}$, $0\leq t \leq \left[\frac{n}{2}\right]$ and $a,n,s,s',x\in \mathbb{N}$. Also, $\Theta_{8x,2,x}(\{2, 2s-1, 4x-(2s-1),$ $4x+2s-1, 8x-(2s-1)$, $8x-2\})$ = $\{2, 2s-1+2x, 6x-(2s-1)$, $6x+2s-1, 2x-(2s-1), 8x-2\}$, $\Theta_{8x,2,2x}(\{2, 2s-1, 4x-(2s-1)$, $4x+2s-1, 8x-(2s-1), 8x-2\})$ = $\{2, 2s-1+4x, 8x-(2s-1), 2s-1$, $4x-(2s-1), 8x-2\}$ and $\Theta_{8x,2,3x}(\{2, 2s-1, 4x-(2s-1), 4x+2s-1$, $8x-(2s-1), 8x-2\})$ = $\{2, 2s-1+6x, 2x-(2s-1)$, $2x+2s-1, 6x-(2s-1), 8x-2\}$. This implies, $\Theta_{8x,2,ix}(C_{8x}(R))$ satisfies the symmetric equidistance condition for $i$ = 0,1,2,3; $\Theta_{8x,2,2x}(C_{8x}(R))$ = $C_{8x}(R)$ = $\Theta_{8x,2,0}(C_{8x}(R))$ and $\Theta_{8x,2,x}(C_{8x}(R))$ = $C_{8x}(S)$ = $\Theta_{8x,2,3x}(C_{8x}(R))$ when $R$ = $\{2, 2s-1,$ $4x-(2s-1)\}$ and $S$ = $\{2, 2x-(2s-1), 2x+2s-1\}$,  $s,x\in\mathbb{N}$.

Consider the transformed graph under the transformation $\Theta_{2a,2,t}$ acting on $C_{2a}(R)$ for $R$ = $ \{ 2,$ $2s-1$, $2s'-1 \}$. Let Figure 17 corresponds to $C_{2a}(R)$, $V(C_n(R))$ = $\{v_0, v_1, \ldots,$ $v_{2a-1} \},$ cycles $C$ = $(v_0 v_1 \ldots v_{2a-1}),$ $C_0$ = $(v_0 v_2 \ldots v_{2a-2})$ and $C_1$ = $(v_1 v_3 \ldots v_{2a-1})$. $C_{2a}(R)$ may contain cycle $C$ and only to prove the theorem we consider it, even if it doesn't exist in $C_{2a}(R)$. Let $u_1$ = $v_{2s-1},$ $u_2$ = $v_{2s'-1},$ $u_3$ = $v_{2a-(2s'-1)}$ and $u_4$ = $v_{2a-(2s-1)}$. Clearly, $v_0\in V(C_0)$ and $u_1,u_2,u_3,u_4\in C_1$.  Let $d(u,v)$ denote the distance between the vertices $u$ and $v$ measured on the regular $2a$-gon with vertices $w_0, w_1, \ldots, w_{2a-1}$ which coincide with the vertices of $C_{2a}(R)$. In $C_{2a}(R)$, $d(v_0,u_1)$ = $d(v_0,u_4)$ = $2s-1$ and $d(v_0,u_2)$ = $d(v_0,u_3)$ = $2s'-1$ which implies $d(u_1,u_2)$ = $d(u_4,u_3)$. Thus, 
	
(1) Under Figure 17: $d(v_0,u_1)$ = $d(v_0,u_4)$,  $d(v_0,u_2)$ = $d(v_0,u_3)$ and $d(u_1,u_2)$ = $d(u_4,u_3)$.	
	
When $\Theta_{2a,2,t}$ acts on $C_{2a}(R),$ cycle $C_0$ doesn't change but $C_1$ simply rotates uniformly for sucessive values of $t$ so that the relative positions of $u_1, u_2, u_3$ and $u_4$ remain the same in $C_1$. Also for a particular value of $t$ if $\Theta_{2a,2,t}(C_{2a}(R))$ = $C_{2a}(S)$ for some $S$, then two out of the 4 vertices $u_1,u_2,u_3$ and $u_4$ lie to the right of $v_0 v_a$ and the other two to the left and $\Theta_{2a,2,t}(C_{2a}(R))$ satisfies the symmetric equidistance condition w.r.t. $v_0$ by Theorem \ref{ab14}, $0 \leq t \leq \frac{2a}{\gcd(2a, 2)} = a$. Under the above conditions, 3 cases corresponding to Figures 18, 19 and 20 arise due to the rotation of $C_1$ when $\Theta_{2a,2,t}(C_{2a}(R))$ = $C_{2a}(T)$ for different values of $t$, $1 \leq t \leq  a$. In each case, we get at the most one possible circulant graph of the form $C_{2a}(T)$ for some $T \subseteq [1, a]$ using the symmetric equidistance condition. See Figures 17 to 20. 
	
When the transformed graph is circulant in its representation, then using the symmetric equidistance condition,   similar to the case corresponding to Figure 17, we get the following. 
	
	(2) Under Figure 18: $d(v_0,u_4)$ = $d(v_0,u_3)$, $d(v_0,u_1)$ = $d(v_0,u_2)$ and $d(u_4,u_1)$ = $d(u_3,u_2)$; 
	
	(3) Under Figure 19: $d(v_0,u_3)$ = $d(v_0,u_2)$, $d(v_0,u_4)$ = $d(v_0,u_1)$ and $d(u_3,u_4)$ = $d(u_1,u_2)$; and 
	
	(4) Under Figure 20: $d(v_0,u_2)$ = $d(v_0,u_1)$, $d(v_0,u_3)$ = $d(v_0,u_4)$ and $d(u_2,u_3)$ = $d(u_4,u_1)$.  

Graph $C_{2a}(R)$ is vertex-transitive and the vertices of $C_1$ move uniformly for consecutive values of $t$, $t$ = 0,1,.... And so if $t_1$ is the smallest positive integer value of $t \ni \theta_{2a,2,t_1}(C_{2a}(R))$ = $C_{2a}(S_1)$ for some $S_1$, then also $\theta_{2a,2,it_1}(C_{2a}(R))$ = $C_{2a}(S_i)$ for some $S_i \subseteq [1, a]$ and satisfies the symmetric equidistance condition for $i$ = 1,2,.... When the graph under Figure 19 satisfies the symmetric equidistance condition, we get,  $t$ = $a$ and  $\theta_{2a,2,a}(C_{2a}(R))$ = $C_{2a}(R)$. 

When $\Theta_{n,2,t}(C_n(R))$ and $C_n(R)$ are  isomorphic circulant graphs of Type-2 w.r.t. $r$ = 2 for some $t,$ then $\Theta_{n,2,t}(C_n(R))$ has to take the form corresponding to Figure 18 or 20. These two cases correspond to $\Theta_{8x,2,x}(C_{8x}(R))$ = $C_{8x}(S)$  and $\Theta_{8x,2,3x}(C_{8x}(R))$ = $C_{8x}(S)$ as it is shown at the begining where $R$ = $\{2, 2s-1, 4x-(2s-1)\}$ and $S$ = $\{2, 2x-(2s-1), 2x+2s-1\}$, $s,x\in\mathbb{N}$. By Theorem \ref{a17c}, the two circulant graphs $C_{8x}(R)$ and  $C_{8x}(S)$ are Type-2 w.r.t. $r$ = 2 when $R$ = $\{2, 2s-1, 4x-(2s-1)\}$ and $S$ = $\{2, 2x-(2s-1), 2x+2s-1\}$, $x \geq 2$, $x \neq 2s-1$, $2s'-1$ = $4x-(2s-1)$ and $s,s',x\in\mathbb{N}$. Hence the result follows. 
\end{proof} 	
\begin{center}
\scalebox{0.82}{ \begin{tikzpicture}[thick, set/.style = { circle, minimum size = .02cm}]
	
	\node (0) at (-4,11.5) [circle,fill=red!30]{$v_{0}$};
	\node (1) at (-2.75,10.75) [circle][circle,fill=blue!30]{};
	\node (2) at (-2.25,10) [circle,fill=blue!30][label=2:]{};
	\node (3) at (-2,9) [circle,fill=blue!30][label=3:]{$u_{1}$};
	\node (4) at (-2,7.75) [circle,fill=blue!30][label=4:]{};
	\node (5) at (-2.5,6.5) [circle,fill=blue!30][label=5:]{$u_{2}$};
	\node (6) at (-3,5.5) [circle,fill=blue!30][label=6:]{};		
	
	\node (7) at (-4,4.75)[circle,fill=red!30][label=7:]{$u_{a}$};
	\node (8) at (-4.75,5.5) [circle,fill=blue!30][label=8:]{};
	\node (9) at (-5.5,6.5) [circle,fill=blue!30][label=9:]{$u_{3}$};
	\node (10) at (-6,7.75) [circle,fill=blue!30][label=10:]{};
	\node (11) at (-6,9) [circle,fill=blue!30][label=11:]{$u_{4}$};
	\node (12) at (-5.75,10) [circle,fill=blue!30][label=12:]{};
	\node (13) at (-5.25,10.75) [circle,fill=blue!30][label=13:]{};

\draw[->] [brown](-2.5,10.4) to [out=120,in=220] (-1.9,10.6);
	\node (14) at (-1.6,10.7) [label=14:]{$C_1$};
	
	\draw [dashed](1) -- (2);
	\draw [dashed](2) -- (3);
	\draw [dashed](3) -- (4);
	\draw [dashed](4) -- (5);
	\draw [dashed](5) -- (6);
	\draw [dashed](6) -- (8);
	\draw [dashed](8) -- (9);
	\draw [dashed](9) -- (10);
	\draw [dashed](10) -- (11);
	\draw [dashed](11) -- (12);
	\draw [dashed](12) -- (13);
	\draw [dashed](13) -- (1);
	
	\draw (0) -- (3);
	\draw (0) -- (5);
	\draw (0) -- (9);
	\draw (0) -- (11);
	
	\draw [red][dashed](0) -- (7);
	
	\node (00) at (4,11.5) [circle,fill=red!30]{$v_{0}$};
	\node (01) at (5.25,10.75) [circle][circle,fill=blue!30]{$u_{4}$};
	\node (02) at (5.75,10) [circle,fill=blue!30][label=02:]{};
	\node (03) at (6,9) [circle,fill=blue!30][label=03:]{};
	\node (04) at (6,7.75) [circle,fill=blue!30][label=04:]{};
	\node (05) at (5.5,6.5) [circle,fill=blue!30][label=05:]{};
	\node (06) at (5,5.5) [circle,fill=blue!30][label=06:]{$u_{1}$};		
	
	\node (07) at (4,4.75)[circle,fill=red!30][label=07:]{$u_{a}$};
	\node (08) at (3,5.5) [circle,fill=blue!30][label=08:]{};
	\node (09) at (2.5,6.5) [circle,fill=blue!30][label=09:]{$u_{2}$};
	\node (010) at (2,7.75) [circle,fill=blue!30][label=010:]{};
	\node (011) at (2,9) [circle,fill=blue!30][label=011:]{};
	\node (012) at (2.25,10) [circle,fill=blue!30][label=012:]{$u_{3}$};
	\node (013) at (2.75,10.75) [circle,fill=blue!30][label=013:]{};
	
	\draw [dashed](01) -- (02);
	\draw [dashed](02) -- (03);
	\draw [dashed](03) -- (04);
	\draw [dashed](04) -- (05);
	\draw [dashed](05) -- (06);
	\draw [dashed](06) -- (08);
	\draw [dashed](08) -- (09);
	\draw [dashed](09) -- (010);
	\draw [dashed](010) -- (011);
	\draw [dashed](011) -- (012);
	\draw [dashed](012) -- (013);
	\draw [dashed](013) -- (01);
	
	\draw (00) -- (01);
	\draw (00) -- (06);
	\draw (00) -- (09);
	\draw (00) -- (012);
	
	\draw [red][dashed](00) -- (07);
	
	\end{tikzpicture}}\\
	\vspace{0.2cm} 
	\hspace{.2cm} Figure $17$ \hspace{5.2cm} Figure $18$  \label{fig 9}			
\end{center}
\begin{center}
\scalebox{0.82}{	\begin{tikzpicture}[thick,
	set/.style = { circle, minimum size = .02cm}]
	
	\node (0) at (-4,4.5) [circle,fill=red!30]{$v_{0}$};
	\node (1) at (-2.75,3.75) [circle][circle,fill=blue!30]{};
	\node (2) at (-2.25,3) [circle,fill=blue!30][label=2:]{$u_{3}$};
	\node (3) at (-2,2) [circle,fill=blue!30][label=3:]{};
	\node (4) at (-2,.75) [circle,fill=blue!30][label=4:]{$u_{4}$};
	\node (5) at (-2.5,-.5) [circle,fill=blue!30][label=5:]{};
	\node (6) at (-3,-1.5) [circle,fill=blue!30][label=6:]{};		
	
	\node (7) at (-4,-2.25)[circle,fill=red!30][label=7:]{$u_{a}$};
	\node (8) at (-4.75,-1.5) [circle,fill=blue!30][label=8:]{};
	\node (9) at (-5.5,-.5) [circle,fill=blue!30][label=9:]{};
	\node (10) at (-6,.75) [circle,fill=blue!30][label=10:]{$u_{1}$};
	\node (11) at (-6,2) [circle,fill=blue!30][label=11:]{};
	\node (12) at (-5.75,3)[circle,fill=blue!30][label=12:]{$u_{2}$};
	\node (13) at (-5.25,3.75) [circle,fill=blue!30][label=13:]{};
	
	\draw [dashed](1) -- (2);
	\draw [dashed](2) -- (3);
	\draw [dashed](3) -- (4);
	\draw [dashed](4) -- (5);
	\draw [dashed](5) -- (6);
	\draw [dashed](6) -- (8);
	\draw [dashed](8) -- (9);
	\draw [dashed](9) -- (10);
	\draw [dashed](10) -- (11);
	\draw [dashed](11) -- (12);
	\draw [dashed](12) -- (13);
	\draw [dashed](13) -- (1);
	
	\draw (0) -- (2);
	\draw (0) -- (4);
	\draw (0) -- (10);
	\draw (0) -- (12);
	
	\draw [red][dashed](0) -- (7);
	
	\node (00) at (4,4.5) [circle,fill=red!30]{$v_{0}$};
	\node (01) at (5.25,3.75) [circle][circle,fill=blue!30]{};
	\node (02) at (5.75,3) [circle,fill=blue!30][label=02:]{$u_{2}$};
	\node (03) at (6,2) [circle,fill=blue!30][label=03:]{};
	\node (04) at (6,.75) [circle,fill=blue!30][label=04:]{};
	\node (05) at (5.5,-.5)[circle,fill=blue!30][label=05:]{$u_{3}$};
	\node (06) at (5,-1.5) [circle,fill=blue!30][label=06:]{};		
	\node (07) at (4,-2.25)[circle,fill=red!30][label=07:]{$u_{a}$};
	\node (08) at (3,-1.5) [circle,fill=blue!30][label=08:]{$u_{4}$};
	\node (09) at (2.5,-.5) [circle,fill=blue!30][label=09:]{};
	\node (010) at (2,.75) [circle,fill=blue!30][label=010:]{};
	\node (011) at (2,2) [circle,fill=blue!30][label=011:]{};
	\node (012) at (2.25,3) [circle,fill=blue!30][label=012:]{};
	\node (013) at (2.75,3.75) [circle,fill=blue!30][label=013:]{$u_{1}$};
	
	\draw [dashed](01) -- (02);
	\draw [dashed](02) -- (03);
	\draw [dashed](03) -- (04);
	\draw [dashed](04) -- (05);
	\draw [dashed](05) -- (06);
	\draw [dashed](06) -- (08);
	\draw [dashed](08) -- (09);
	\draw [dashed](09) -- (010);
	\draw [dashed](010) -- (011);
	\draw [dashed](011) -- (012);
	\draw [dashed](012) -- (013);
	\draw [dashed](013) -- (01);
	
	\draw (00) -- (02);
	\draw (00) -- (05);
	\draw (00) -- (08);
	\draw (00) -- (013);
	
	\draw [red][dashed](00) -- (07);
	
	\end{tikzpicture} }\\
	\vspace{0.2cm} 
	\hspace{.15cm}  Figure $19$ \hspace{5cm} Figure $20$  \label{fig 9b}			
\end{center}
\begin{cor}{\rm \cite{v20}}\quad \label{a18b} {\rm For $n \geq 2,$ $k \geq 3,$ $1 \leq 2s-1 < 2s'-1 \leq [\frac{n}{2}],$ $0 \leq t \leq [\frac{n}{2}],$ $\gcd(n, r)$ = $m$ = 2, $r\in R$, $R$ = $\{2s-1,$ $2s'-1, 2p_1,2p_2,...,2p_{k-2}\}$, $\gcd(p_1,p_2,...,p_{k-2})$ = $1$ and $n,s,s',p_1,p_2,\ldots,p_{k-2} \in \mathbb{N},$ if for a given set of values of $p_1,p_2,\ldots,p_{k-2},s$ and $s'$,  $\Theta_{n,2,t}(C_n(R))$ and $C_n(R)$ are  isomorphic circulant graphs of Type-2 w.r.t. $m$ = 2 for some $t,$ then $n$ $\equiv$ $0~(mod ~ 8),$ $2s-1+2s'-1$ = $\frac{n}{2},$ $2s-1$ $\neq$ $\frac{n}{8},$ $t$ = $\frac{n}{8}$ or $\frac{3n}{8},$ $1 \leq 2s-1 \leq \frac{n}{4}$ and $n \geq 16.$  \hfill $\Box$ }
\end{cor}

\section{On Abelian groups $(T1_{n}(C_n(R)), \circ)$, $(V_{n,r}(C_n(R)), \circ)$ and $(T2_{n,r}(C_n(R)), \circ)$}

In this section, for any circulant graph $C_n(R)$ with $r\in R$ and $m > 1$ is a divisor of $\gcd(n, r)$,  we define $V_{n,r}(C_n(R))$ and $T2_{n,r}(C_n(R))$ similar to $Ad_n(C_n(R))$ and prove that  $(V_{n,r}(C_n(R)), \circ)$ and $(T2_{n,r}(C_n(R)), \circ)$ are Abelian groups and $T2_{n,r}(C_n(R))$ $\subseteq$ $V_{n,r}(C_n(R))$ where $`\circ'$ is defined as given below. We call the group $(T2_{n,r}(C_n(R)), \circ)$ as the {\em Type-2 group of $C_n(R)$} w.r.t. $r$ and the group $(Ad_{n}(C_n(R)), \circ)$ = $(T1_{n,r}(C_n(R)), \circ)$ as the {\em Type-1 group of $C_n(R)$} or {\em Adam's group of $C_n(R)$}.

\begin{definition}{\rm \cite{v20}}\quad \label{a19} Let $V(C_n(R)) = \{v_0, v_1, \ldots, v_{n-1}\}$, $V(K_n)$ = $\{u_0, u_1,\ldots, u_{n-1}\}$, $x = qm+j,$ $0 \leq j \leq m-1$, $m > 1$ is a divisor of $\gcd(n, r)$, $0 \leq q,t,t' \leq \frac{n}{m}-1$, $m,q,t,t',x\in \mathbb{Z}_n$ and $r \in R$. Define $\Theta_{n,m,t}:$ $V(C_n(R))$ $\rightarrow$  $V(K_n)$ $\ni$ $\Theta_{n,m,t}(v_x)$ = $u_{x+jmt}$ and $\Theta_{n,m,t}((v_x, v_{x+s}))$ = $(\Theta_{n,m,t}(v_x), \Theta_{n,m,t}(v_{x+s}))$, $\forall$  $x\in \mathbb{Z}_n$ and $s\in R,$ under subscript arithmetic modulo $n$. Let $V_{n,r}$ = $\{\Theta_{n,m,t}: t = 0,1,...,\frac{n}{m}-1\},$ $V_{n,r}(v_s)$ = $\{\Theta_{n,m,t}(v_s):$ $t = 0,1,\ldots, \frac{n}{m}-1\}$, $s \in \mathbb{Z}_n$ and $V_{n,r}(C_n(R))$ = $\{\Theta_{n,m,t}(C_n(R)) = C_n(\Theta_{n,m,t}(R)):$ $t$ = $0,1,\ldots, \frac{n}{m}-1\}$ where $\theta_{n,m,t}(R)$ in $C_n(\theta_{n,m,t}(R))$  is calculated under reflexive modulo $n$. In $V_{n,r}$, define $'\circ'$ $\ni$ $\Theta_{n,m,t} ~\circ ~ \Theta_{n,m,t'}$ =  $\Theta_{n,m,t+t'}$ and $(\Theta_{n,m,t} ~\circ ~ \Theta_{n,m,t'})(C_n(R))$ = $\Theta_{n,m,t}(C_n(R)) ~\circ ~ \Theta_{n,m,t'}(C_n(R))$ = $\Theta_{n,m,t+t'}(C_n(R))$, $\forall$ $\Theta_{n,m,t},\Theta_{n,m,t'}\in V_{n,r}$ where $t+t'$ is calculated under arithmetic modulo ~$\frac{n}{m}$.
\end{definition}

$V_{n,r}(C_n(R))$ = $\{\Theta_{n,m,t}(C_n(R)): t = 0,1,\ldots, \frac{n}{m}-1\}$ and for $t$ = 0 to $\frac{n}{m}-1$, the $\frac{n}{m}$ graphs $\Theta_{n,m,t}(C_n(R))$ are isomorphic and 
$V_{n,r}(C_n(R))$ contains all isomorphic circulant graphs of Type-2 of $C_n(R)$ w.r.t. $m,$ if exist, under the transformation $\Theta_{n,m,t}$ on $C_n(R)$ where $r\in R$ and $m > 1$ is a divisor of $\gcd(n, r)$. Now, consider an important algebraic property on $V_{n,r}(C_n(R))$. 

\begin{theorem}{\rm  \quad \label{ab19}  Under the above definition of $`\circ'$, $(V_{n,r}(C_n(R)), \circ)$ is an abelian  group.}
\end{theorem}
\begin{proof}   We have, $V_{n,r}(C_n(R))$ = $\{\Theta_{n,m,t}(C_n(R)): t = 0,1,..., \frac{n}{m}-1\}$ and for $t$ = 0 to $\frac{n}{m}-1$, the $\frac{n}{m}$ graphs $\Theta_{n,m,t}(C_n(R))$ are isomorphic where $r\in R$ and $m > 1$ is a divisor of $\gcd(n, r)$.  Also, $\forall t,t'\in \{0,1,\ldots,$ $\frac{n}{m}-1\}$, $\Theta_{n,m,t}(C_n(R)) \circ  \Theta_{n,m,t'}(C_n(R))$ = $\Theta_{n,m,t+t'}(C_n(R))\in V_{n,r}$ where $t+t'$ is calculated under addition modulo $\frac{n}{m}.$ This implies, $V_{n,r}(C_n(R))$ is closed under $`\circ'$. $\Theta_{n,m,0}(C_n(R))$ = $C_n(R)$ is the identity element. $`\circ'$ is abelian since $\Theta_{n,m,t}(C_n(R)) \circ \Theta_{n,m,t'}(C_n(R))$ = $\Theta_{n,m,t+t'}(C_n(R))$ = $\Theta_{n,m,t'+t}(C_n(R))$ = $\Theta_{n,m,t'}(C_n(R)) \circ \Theta_{n,m,t}(C_n(R))$, $\forall t,t'\in\{0,1,...,\frac{n}{m}-1\}$. The associative law can be proved very easily.  	 $\Theta_{n,m,\frac{n}{m}-t}(C_n(R))$ is the inverse of $\Theta_{n,m,t}(C_n(R))$, $\forall t\in \{0,1,\ldots, \frac{n}{m}-1\}$ since $\Theta_{n,m,\frac{n}{m}-t}(C_n(R)) \circ \Theta_{n,m,t}(C_n(R))$ = $\Theta_{n,m,(\frac{n}{m}-t)+t}(C_n(R))$ = $\Theta_{n,m,0}(C_n(R))$. Hence the result. 
\end{proof}

\begin{theorem} \cite{v20}  \label{a16} {\rm Let $n,r\in \mathbb{N}$, $m > 1$ be a divisor of $\gcd(n, r)$, $r \notin R$, $\Theta_{n,m,t}(C_n(R \cup \{ r \}))$ = $C_n(S)$ and $C_n(R \cup \{ r \})$ and $C_n(S)$ be Type-2 isomorphic w.r.t. $m$, 	$1 \leq t \leq \frac{n}{m} - 1$. Then, there doesn't exist circulant graph that is isomorphic of Type-2 w.r.t. $m$ to $C_n(R)$. That is $C_n(R)$ and $\Theta_{n, m,t}(C_n(R))$ = $C_n(S)$ are isomorphic but not of Type-2 w.r.t. $m$, $m > 1$ is a divisor of $\gcd(n, r)$ and $0 \leq t \leq \frac{n}{m} - 1$.}
\end{theorem}
\begin{proof}\quad Let $R = \{r_1,r_2,\ldots,r_k\}$, $m > 1$ be a divisor of $\gcd(n, r)$ and $\gcd(n,r_i) \neq m$ for $i = 1,2,\ldots,k$.
	
	Given, $\Theta_{n,m,t}(C_n(R \cup \{ r \}))$ = $C_n(S)$ $\neq$ $C_n(x(R\cup \{r\}))$ for every $x\in \varphi_n$  and $C_n(R \cup \{ r \})$ and $C_n(S)$ are Type-2 isomorphic w.r.t. $m$, $1 <x < n-1.$
	
	This implies, $\Theta_{n,m,t}(C_n(R))$ $\cup$ $\Theta_{n,m,t}(C_n(r))$ = $C_n(S)$ and $r\in S,R\cup\{r\}$ $\ni$ $m > 1$ is a divisor of $\gcd(n, r)$ which implies that  $\Theta_{n,m,t}(C_n(R))$ = $C_n(S \setminus \{r\})$ and $r\notin R,S \setminus \{r\}$ since $\Theta_{n,m,t}(r)$ = $r$ and $\Theta_{n,m,t}(C_n(r))$ = $C_n(\Theta_{n,m,t}(r))$ = $C_n(r)$. ~\hfill $(a)$
	
	Now, for every $s\in S \setminus \{r\},$ $m > 1$ is not a divisor of $\gcd(n, s)$. And so $C_n(S \setminus \{r\})$ can not be Type-2 isomorphic w.r.t. {$m$ to any} circulant graph $C_n(R)$, $m > 1$ is a divisor of $\gcd(n, r)$.
	
	Also, if $\Theta_{n,m,t}(C_n(U))$ = $C_n(W)$ for some $t$, then 
	
	\hspace{.7cm} $\Theta_{n,m,\frac{n}{m}-t}(C_n(W)) = \Theta_{n,m,\frac{n}{m}-t}(\Theta_{n,m,t}(C_n(U))) = \Theta_{n,m,\frac{n}{m}-t+t}(C_n(U))  = C_n(U)$. \hfill $(b)$
	
	Using $(b)$ in $(a)$, we get, $\Theta_{n,m,\frac{n}{m}-t}(C_n(S \setminus \{r\}))$ = $C_n(R)$ which implies that $C_n(S \setminus \{r\})$ and $C_n(R)$ are isomorphic circulant graphs but they are not of Type-2 w.r.t. $m$ since $m > 1$ is not a divisor of $\gcd(n, s)$ for every $s\in S \setminus \{r\}$. This implies, $C_n(R)$ and $\Theta_{n, m,t}(C_n(R))$ are isomorphic but not of Type-2 w.r.t. m, $0 \leq t \leq \frac{n}{m} - 1$.
\end{proof}

It is interesting to know how Type-2 isomorphism is defined, how the transformation is taking place and under what conditions Type-2 isomorphic circulant graph(s) are obtained. These are given in the following paragraphs \cite{v20}. 

Here, we take $m_i$ = $\gcd(n, r_i)$ for at least one $i,$ $r_i\in R,$ $1 \leq i \leq k,$ so that $C_n(R)$ for a set $R$ = $\{ r_1, r_2, \ldots, r_k \}$ contains $m_i$ number of disjoint isomorphic circulant subgraphs, say, $\Gamma_0, \Gamma_1, \ldots, \Gamma_{m_i-1},$ each having one periodic cycle of period $r_i$ and of length $n_i$ = $\frac{n}{\gcd(n, r_i)},$ using Theorem \ref{a1} where $v_j \in V(\Gamma_j)$ (and thereby $v_{j+qm_i} \in V(\Gamma_j),$ using subscript modulo $n,$ $0 \leq q \leq n_i-1),$ $0 \leq j \leq m_i-1$. For some value of $t$, the transformed graph $\Theta_{n,m_i,t}(C_n(R))$ may be in the form of $C_n(S)$ for some $S,$ $0 \leq t \leq n_i-1$. Suppose $\Theta_{n,m_i,t}(C_n(R))$ = $C_n(S)$ for some $S$ = $\{s_1, s_2, \ldots, s_k\},$ then by the necessary condition for isomorphism between $C_n(R)$ and $C_n(S),$ for each $j$ there exists $q$ such that $\gcd(n, r_j)$ = $\gcd(n, s_q)$ using Theorem \ref{a3}, $1 \leq r_j,s_q \leq \frac{n}{2}$ and $1 \leq j,q \leq k$. Under this transformation, each $\Gamma_j$ simply rotates $tm_i$ positions (points) w.r.t. $\Gamma_{j-1}$ in the regular $n$-gon in the clockwise direction and thereby the vertices of each $V_j$ = $V(\Gamma_j)$ are taking the positions of the vertices of $V_j$ only, $0 \leq t \leq \frac{n}{m_i}-1$ and $0 \leq j \leq m_i-1.$ By this transformation acting on $C_n(R),$ we may get a circulant graph of the form $C_n(S)$ for some $S \subseteq [1,\frac{n}{2}].$  

While considering Type-2 isomorphism of $C_n(R),$ under the transformation $\Theta_{n,m_i,t},$ sets of elements $V(\Gamma_j)$ move uniformly instead of individual elements (points) under Adam's isomorphism. Also, the sets of elements $\Gamma_j$ are isomorphic circulant subgraphs of the given circulant graph and the {\it jump sizes of these isomorphic circulant subgraphs do not change and all others move and may change} but still have to satisfy the necessary condition for circulant graph isomorphism (Theorem \ref{a3}). The movements are so arranged that the resultant graph may be circulant in its representation.

	Under Type-2 isomorphism acting on $C_n(R)$, the vertices of $C_n(R)$ move uniformly on the vertices of regular $n$-gon $\ni$ sets of vertices of circulant subgraphs $\Gamma_j$ of $C_n(R)$ move uniformly and the transformed graph $\Theta_{n,m_i,t}(C_n(R))$ becomes $C_n(S)$ for some $t$ and $S$  where $\Gamma_j$ is the periodic cycle of period $r_i$ and length $\frac{n}{m_i}$ in $C_n(R)$ and contains $v_j,$ $S \subseteq [1, \frac{n}{2}]$, $\gcd(n, r_i)$ = $m_i > 1$, $0 \leq j \leq m_i-1$ and $0 \leq t \leq \frac{n}{m_i}-1$. And when $C_n(R)$ and $C_n(S)$ are Type-2 isomorphic w.r.t. $m_i$, $C_n(R)$ and $C_n(S)$ are isomorphic but not of Adam's, $S \neq R$. 
	
	 Following result is related to Remark \ref{r8} on Type-2 isomorphic circulant graphs.
	
	\begin{theorem}\quad {\rm  \label{a17a} Let $n,r \in \mathbb{N}$, $V(C_n(R))$ = $\{v_1,v_2,...,v_{n-1}\}$ and $V(K_n)$ = $\{u_1,u_2,$ $...,u_{n-1}\}$ such that $n > n-r > r > 1$ and $n \neq yr$ for any $y\in\mathbb{N}$. Define mapping $\Theta_{n,r,t}:$ $V(C_n(R)) \rightarrow V(K_n)$ such that $\Theta_{n,r,t}(v_x)$ = $u_{x+jrt}$, $\forall$ $v_x \in V(C_n(R))$ under subscript arithmetic modulo $n$ where $x = qr+j,$ $0 \leq j \leq r-1,$ $0 \leq q,t \leq \frac{n}{\gcd(n,r)}-1$ and $j,q,t \in \mathbb{Z}_n$. Then, $C_n(R)$ doesn't have isomorphic circulant graph of Type-2 w.r.t. $r$.}
	\end{theorem}
	\begin{proof} 	When $n > n-r > r > 1$ and $n \neq yr$ for any $y\in\mathbb{N}$ $\ni$ $1 \leq y \leq \frac{n}{2}$ and $n,r,y \in \mathbb{Z}_n$, $\frac{n}{r}$ is not an integer and thereby $rC_y(1) = C_{ry}(r) \ncong C_n(s)$ for any $s,y\in\mathbb{N}$. This implies, the question of getting $r$ copies of a circulant subgraph of $C_n(R)$ with vertex set = $V(C_n(R))$ is not possible. This implies, getting Type-2 isomorphic circulant graph(s) to $C_n(R)$ is not possible in this case. Hence the result. 
	\end{proof}
		
	Some properties of $\Theta_{n,m_i,t}(C_n(R))$ are given below. 

\vspace{.1cm} 
\noindent
{\bf Properties of $\Theta_{n,m,t}(C_n(R)):$}	
\begin{enumerate}
	\item [\rm (a)]  \label{r11a}  Let $\Theta_{n,m,t}(C_n(R)) = C_n(S)$, $m > 1$ be a divisor of $\gcd(n,r)$ and $r_i \in [1, \frac{n}{2}]$ $\ni$ $\gcd(n, r_i)$ = $\gcd(n,r)$. Then, $r_i \in R$ if and only if $r_i \in S$ follows from the definition of $\Theta_{n,m,t}.$
	
	\item [\rm (b)]  \label{r11b} For a given circulant graph $C_n(R)$ and for a particular value of $t$ if $\Theta_{n,m,t}(C_n(R))$ = $C_n(S)$ for some $S$, then $\forall$ $t'$ $\ni$ $0 \leq t,t' \leq \frac{n}{m} -1$, $\Theta_{n,m,t+t'}(C_n(R))$ = $\Theta_{n,m,t'+t}(C_n(R))$ 
	\\
	= $\Theta_{n,m,t'}(\Theta_{n,m,t}(C_n(R)))$ = $\Theta_{n,m,t'}(C_n(S))$ where $m > 1$ is a divisor of $\gcd(n,r)$.
	
	\item [\rm (c)]  \label{r11c} Let $C_n(R)$ $\cong$ $C_n(S)$ and $0 \leq t \leq \frac{n}{m}-1$. Then, for some $t$, $C_n(S)$ = $\Theta_{n,m,t}(C_n(R))$  if and only if $C_n(R)$ = $\Theta_{n,m,\frac{n}{m}-t}(C_n(S)).$ This follows from $\Theta_{n,m,\frac{n}{m}-t}(C_n(S)) = \Theta_{n,m,\frac{n}{m}-t}(\Theta_{n,m,t}(C_n(R))) = \Theta_{n,m,(\frac{n}{m}-t)+t}(C_n(R))$ = $\Theta_{n,m,0}(C_n(R))$ = $C_n(R).$

\vspace{.1cm} 
Circulant graph	$C_n(R)$ is vertex-transitive and  for a particular value of $t,$ $\Theta_{n,m,t}(C_n(R))$ is the resultant of uniform rotation of $m$ copies of circulant subgraphs $\Gamma_j$ of $C_n(R)$ and by Theorem \ref{ab14}, $\Theta_{n,m,t}(C_n(R))$ satisfies the symmetric equidistance condition w.r.t. each of its vertices if and only if $\Theta_{n,m,t}(C_n(R))$ = $C_n(S')$ for some $S' \subseteq [1, \frac{n}{2}]$, $r \in R$, $m > 1$ is a divisor of $\gcd(n,r)$, $0 \leq j \leq m-1$ and $0 \leq t \leq \frac{n}{m}-1$. Based on the above, we obtain some more properties of $\Theta_{n,m,t}(C_n(R))$ as given below. 

\vspace{.1cm} 
\item [\rm (d)] \label{r11d} 	If $t_1$ is the smallest positive integer value of $t$ such that $\Theta_{n,m,t_1}(C_n(R))$ = $C_n(S_1)$ for some $S_1 \neq R$, $q_1t_1$ = $\frac{n}{m}$ and $q_1\in\mathbb{N}$, then $\forall i\in\mathbb{N}_0 \ni$ $0 \leq it_1 \leq \frac{n}{m}-1$, $\Theta_{n,m,it_1}(C_n(R))$ = $C_n(S_i)$ for some $S_i \subseteq [1, \frac{n}{2}]$, $S_0$ = $R$, $r\in R$, $m > 1$ is a divisor of $\gcd(n,r)$, $0 \leq i \leq q_1$ and $1 \leq t_1 \leq \frac{n}{m}$. 

\item [\rm (e)]  \label{r11e}  If $t_2$ is the smallest positive integer value of $t$ such that $\Theta_{n,m,t_2}(C_n(R))$ = $C_n(S_1)$ for some $S_1 \neq R$, $C_n(S_1)\in T1_n(C_n(R))$ = $Ad_n(C_n(R))$, $q_2t_2$ = $\frac{n}{m}$ and $q_2\in\mathbb{N}$, then $\forall i\in\mathbb{N}_0$ such that $0 \leq it_2 \leq \frac{n}{m}-1$, $\Theta_{n,m,it_2}(C_n(R))$ = $C_n(S_i)$ and $C_n(S_i)\in T1_n(C_n(R))$ for some $S_i \subseteq [1, \frac{n}{2}]$, $S_0$ = $R$, $r\in R$, $m > 1$ is a divisor of $\gcd(n,r)$, $0 \leq i \leq q_2$ and $1 \leq t_2 \leq \frac{n}{m}$. 

\vspace{.1cm}
Let $T2_{n,r}(C_n(R))$ = $\{C_n(R)\} \cup \{C_n(S): C_n(S)$ is Type-2 isomorphic to $C_n(R)$ w.r.t. $m\}$, $r\in R$ and $m > 1$ is a divisor of $\gcd(n,r)$. This implies that $T2_{n,r}(C_n(R))$ = $\{C_n(R)\} \cup \{\Theta_{n,m,t}(C_n(R)): \Theta_{n,m,t}(C_n(R))$ = $C_n(S)$ and $C_n(S)$ is Type-2 isomorphic to $C_n(R)$ w.r.t. $m$, $r\in R$ and $m > 1$ is a divisor of $\gcd(n,r)$, $1 \leq t \leq \frac{n}{m}-1\}$ = $\{\Theta_{n,m,0}(C_n(R))\} \cup \{\Theta_{n,m,t}(C_n(R)):$ $\Theta_{n,m,t}(C_n(R))$ = $C_n(S)$ and $C_n(S)$ is Type-2 isomorphic to $C_n(R)$ w.r.t. $m$, $1 \leq t \leq \frac{n}{m}-1\}$, $r\in R$ and $m > 1$ is a divisor of $\gcd(n,r)$. Clearly, $T2_{n,r}(C_n(R))$ $\subseteq V_{n,r}(C_n(R))$ and  $T1_n(C_n(R))$ $\cap$ $T2_{n,r}(C_n(R))$ = $\{C_n(R)\}$ where $T1_n(C_n(R))$ = $Ad_n(C_n(R))$. We present here some properties related to $T2_{n,r}(C_n(R))$.

\vspace{.1cm}
\item [\rm (f)]  \label{r11f}  If $t_3$ is the smallest positive integer value of $t$ $\ni$ $\Theta_{n,m,t_3}(C_n(R))$ = $C_n(S_1)$ for some $S_1 \neq R$, $C_n(S_1)\in T2_{n,r}(C_n(R))$, $q_3t_3$ = $\frac{n}{m}$ and $q_3\in\mathbb{N}$, then $\forall i\in\mathbb{N}_0 \ni 0 \leq it_3 \leq \frac{n}{m}-1$, $\Theta_{n,m,it_3}(C_n(R))$ = $C_n(S_i)$ for some $S_i \subseteq [1, \frac{n}{2}]$, $C_n(S_i)\in T2_{n,r}(C_n(R))$, $S_0$ = $R$, $m > 1$ is a divisor of $\gcd(n,r)$, $0 \leq i \leq q_3$ and $1 \leq t_3 \leq \frac{n}{m}-1$. 

\item [\rm (g)]  \label{r11g}  If $t_4$ is the smallest positive integer value of $t$ $\ni$ $\Theta_{n,m,t_4}(C_n(R))$ = $C_n(R)$ = $\Theta_{n,m,0}(C_n(R))$, $q_4t_4$ = $\frac{n}{m}$, $1 \leq t_4 \leq \frac{n}{m}-1$ and $q_4\in\mathbb{N}$, then $\forall i\in\mathbb{N}_0 \ni 0 \leq it_4 \leq \frac{n}{m}-1$, $\Theta_{n,m,it_4}(C_n(R))$ = $C_n(R)$, $r\in R$, $m > 1$ is a divisor of $\gcd(n,r)$, $0 \leq i \leq q_4$ and $1 \leq t_4 \leq \frac{n}{m}-1$.  

\vspace{.1cm}
Similar to Theorem \ref{a7b} related to $T1_n(C_n(R))$, we present here a property of  $T2_{n,r}(C_n(R))$.

\vspace{.1cm}		
\item [\rm (h)]  \label{r11h} Let $C_n(R)$ $\cong$ $C_n(S)$, $R \neq S$ and $|R| = |S| \geq 3$. Then, $C_n(S)\in$ $T2_{n,r}(C_n(R))$ if and only if  $T2_{n,r}(C_n(R))$ = $T2_{n,r}(C_n(S))$ if and only if  $C_n(R)\in T2_{n,r}(C_n(S))$. This is true because of the following. When $T2_{n,r}(C_n(R))$ = $T2_{n,r}(C_n(S))$, $T2_{n,r}(C_n(R))$ $\subseteq$ $T2_{n,r}(C_n(S))$ and $T2_{n,r}(C_n(S))$ $\subseteq$ $T2_{n,r}(C_n(R))$ which implies, $C_n(R)\in T2_{n,r}(C_n(S))$ and $C_n(S)\in T2_{n,r}(C_n(R))$. On the other hand, assume that $C_n(S)\in T2_{n,r}(C_n(R))$, $S \neq R$. Our aim is to prove that $T2_{n,r}(C_n(R))$ = $T2_{n,r}(C_n(S))$. That is to prove $T2_{n,r}(C_n(R))$ $\subseteq$ $T2_{n,r}(C_n(S))$ and $T2_{n,r}(C_n(S))$ $\subseteq$ $T2_{n,r}(C_n(R))$. Let $C_n(T)\in T2_{n,r}(C_n(R))$ for some $T$, $T \neq R$. Given, $C_n(S)\in T2_{n,r}(C_n(R))$, $S \neq R$. This implies, $C_n(T)$ = $\Theta_{n,m,t}(C_n(R))$, $C_n(S)$ = $\Theta_{n,m,t'}(C_n(R))$ (which implies, $C_n(R)$ = $\Theta_{n,m,\frac{n}{m}-t'}(C_n(S))$ using property $(c)$,  $0 \leq t' \leq \frac{n}{m}-1$) for some $t$ and $t'$,  $C_n(T)$ and $C_n(R)$ are Type-2 isomorphic w.r.t. $m$ as well as $C_n(S)$ and $C_n(R)$, $S,T \neq R$ and $1 \leq t,t' \leq \frac{n}{m}-1$. This implies that $C_n(T)$ = $\Theta_{n,m,t}(C_n(R))$ = $\Theta_{n,m,t}(\Theta_{n,m,\frac{n}{m}-t'}(C_n(S)))$ = $\Theta_{n,m,t+(\frac{n}{m}-t')}(C_n(S))$ = $\Theta_{n,m,t_2}(C_n(S))$ for some $t$,$t_1$ = $\frac{n}{m}-t'$ and $t_2$ = $t+(\frac{n}{m}-t')$ and $C_n(S)$, $C_n(T)$ and $C_n(R)$ are isomorphic of Type-2 w.r.t. $m$, $S,T \neq R$ and $1 \leq t,t',t_1,t_2 \leq \frac{n}{m}-1$. This implies, $C_n(T)$ = $\Theta_{n,m,t_2}(C_n(S))$ for some $t_2$ and $C_n(T)$ and $C_n(S)$ are Type-2 isomorphic w.r.t. $m$, $1 \leq t_2 \leq \frac{n}{m}-1$. This implies, $C_n(T) \in T2_{n,r}((C_n(S)))$ which implies, $T2_{n,r}(C_n(R))$ $\subseteq$ $T2_{n,r}(C_n(S))$. Similarly, we can prove that $T2_{n,r}(C_n(S))$ $\subseteq$ $T2_{n,r}(C_n(R))$. This implies, $T2_{n,r}(C_n(R))$ = $T2_{n,r}(C_n(S))$. 
	
\item [\rm (i)]  \label{r11i}  $T2_{n,r}(C_n(R))$ is closed under $`\circ'$ follows from property $(f)$ where $\circ$ is defined as $\forall$ $C_n(S)$ = $\Theta_{n,m,t_1}(C_n(R)), C_n(T) = \Theta_{n,m,t_2}(C_n(R))\in T2_{n,r}(C_n(R))$, $\Theta_{n,m,t_1}(C_n(R)) \circ \Theta_{n,m,t_2}(C_n(R))$ = $\Theta_{n,m,t_3}(C_n(R))$, $t_3$ = $t_1+t_2$, $0 \leq t_1,t_2,t_3 \leq \frac{n}{m}-1$ and $t_3$ in $\Theta_{n,m,t_3}$ is calculated under arithmetic modulo $\frac{n}{m}$. 
		
\item [\rm (j)]  \label{r11j} Let $C_n(R)$ $\cong$ $C_n(S)$, $R \neq S$ and $|R| = |S| \geq 3$. Then, either $T2_{n,r}(C_n(R))$ $\cap$ $T2_{n,r}(C_n(S))$ = $\emptyset$ or $T2_{n,r}(C_n(R))$ = $T2_{n,r}(C_n(S))$. This follows from property $(h)$.		
\end{enumerate}

$C_n(R)$ has Type-2 isomorphic circulant graph w.r.t. $m$ if and only if $T2_{n,r}(C_n(R))$ $\neq$ $\{C_n(R)\}$ if and only if $T2_{n,r}(C_n(R$ $)) \setminus \{C_n(R)\} \neq \emptyset$ if and only if $|T2_{n,r}(C_n(R))| > 1$ \cite{v20}. In the next theorem, we  prove that $(T2_{n,r}(C_n(R)), \circ)$  is a subgroup of $(V_{n,r}(C_n(R)), \circ)$. 

\begin{theorem}\cite{v20} \label{a20} {\rm $(T2_{n,r}(C_n(R)),  \circ)$ is a subgroup of $(V_{n,r}(C_n(R)), \circ)$ where $r\in R$ and $m > 1$ is a divisor of $\gcd(n,r)$.}
\end{theorem}
\begin{proof}\quad $T2_{n,r}(C_n(R))$ is closed under $`\circ'$ by property $(i)$, $\Theta_{n,m,0}(C_n(R))$ = $C_n(R)\in T2_{n,r}(C_n(R))$ is the identity element in $T2_{n,r}(C_n(R))$, $T2_{n,r}(C_n(R)) \subseteq V_{n,r}(C_n(R))$ and  by Theorem \ref{ab19}, $(V_{n,r}(C_n(R)), \circ)$ is an abelian group. When $T2_{n,r}(C_n(R))$ = $\{ \Theta_{n,m,0}(C_n(R)) = C_n(R) \},$ $(T2_{n,r}(C_n(R)), \circ)$ is the trivial subgroup of $(V_{n,r}(C_n(R)), \circ)$. When $T2_{n,r}(C_n(R))$ $\neq$ $\{ C_n(R) = \Theta_{n,m,0}(C_n(R))\}$, then by property $(f)$, there exists $t_1$, the smallest positive integer value of $t$ in $\theta_{n,m,t}(C_{n}(R))$ $\ni$ $\theta_{n,m,t_1}(C_{n}(R))$ = $C_{n}(S_1)$ and $C_n(S_{1})\in T2_{n,r}(C_n(R))$ for some $S_1 \neq R$, $0 \leq t,t_1 \leq \frac{n}{m}-1$, $q_1t_1$ = $\frac{n}{m}$ and $m > 1$ is a divisor of $\gcd(n,r)$. Let $C_n(S),C_n(T)\in T2_{n,r}(C_n(R))$. Then by property $(f)$, there exists $i$ and $j$ $\ni$ $C_n(S)$ = $\Theta_{n,m,it_1}(C_n(R))$ and $C_n(T)$ = $\Theta_{n,m,jt_1}(C_n(R))$, $0 \leq t_1 \leq \frac{n}{m}-1$, $q_1t_1$ = $\frac{n}{m}$, $0 \leq i,j \leq q_1-1$ and $0 \leq it_1,jt_1 \leq \frac{n}{m}-1$. This implies, $\Theta_{n,m,it_1}(C_n(R)), \Theta_{n,m,jt_1}(C_n(R))\in T2_{n,r}(C_n(R))$, $0 \leq t_1 \leq \frac{n}{m}-1$, $q_1t_1$ = $\frac{n}{m}$, $0 \leq i,j \leq q_1-1$ and $0 \leq it_1,jt_1 \leq \frac{n}{m}-1$. This implies, $\Theta_{n,m,it_1}(C_n(R)), \Theta_{n,m,(q_1-j)t_1}(C_n(R))\in T2_{n,r}(C_n(R))$ using property $(f)$, $q_1t_1$ = $\frac{n}{m}$, $0 \leq i,j \leq q_1-1$ and $0 \leq it_1,jt_1 \leq \frac{n}{m}-1$. This implies, $\Theta_{n,m,it_1}(C_n(R)) \circ \Theta_{n,m,(q_1-j)t_1}(C_n(R)) = \Theta_{n,m,it_1+(q_1-j)t_1}(C_n(R)) = \Theta_{n,m,kt_1}(C_n(R))\in T2_{n,r}(C_n(R))$ using property $(f)$ and $(i)$ where $\Theta_{n,m,(q_1-j)t_1}(C_n(R))$ is the inverse element of $\Theta_{n,m,jt_1}(C_n(R))$ = $C_n(T)$ in $T2_{n,r}(C_n(R))$, $q_1t_1$ = $\frac{n}{m}$, $k$ = $q_1-j+i$, $0 \leq i,j,k \leq q_1-1$ and $0 \leq it_1,jt_1,kt_1 \leq \frac{n}{m}-1$. This implies, $(T2_{n,r}(C_n(R)),  \circ)$ is a subgroup of $(V_{n,r}(C_n(R)), \circ)$ where $r\in R$, $m > 1$ is a divisor of $\gcd(n,r)$ and $m,n,r\in\mathbb{N}$.
\end{proof} 

\begin{definition}{\rm \cite{v20}}\quad \label{a21} With usual notation, group $(T2_{n,r}(C_n(R)), \circ)$ is called the Type-2 group of $C_n(R)$ w.r.t.  $r$.  
\end{definition}

We already established that $C_{16}(1,2,7)$ and $C_{16}(2,3,5)$ are of Type-2 isomorphic w.r.t. $r$ = $2$, $Ad_{16}(C_{16}(1,2,7))$ $\neq$ $Ad_{16}(C_{16}(2,3,5)),$ $V_{16,2}(C_{16}(1,2,7))$ = $V_{16,2}(C_{16}(2,3,5))$ and $T2_{16,2}(C_{16}(1,2,7))$ = $\{\Theta_{16,2,t}(C_{16}(1,2,7)): t = 0,2\}$ = $\{\Theta_{16,2,0}(C_{16}(2,3,5))$ = $C_{16}(2,3,5),$ $\Theta_{16,2,2}(C_{16}(2,$ $3,5))$ = $C_{16}(1,2,7)\}$ = $\{\Theta_{16,2,t}(C_{16}(2,3,5)): t = 0,2\}$ = $T2_{16,2}(C_{16}(2,3,5)).$ Thus, we get, $(T2_{16,2}(C_{16}(1,2,7)), \circ)$ = $(T2_{16,2}(C_{16}(2,3,5)), \circ)$ is the Type-2 group on $C_{16}(1,2,7)$ or on $C_{16}(2,3,5)$ w.r.t. $r$ = 2. In the following, we present more general results on Type-2 isomorphic circulant graphs w.r.t. $m$ = 2.
	
\begin{theorem}{\rm \cite{v20} \quad \label{a22} For $n \geq 2,$ $k \geq 3,$ $1 \leq 2s-1 \leq 2n-1,$ $n \neq 2s-1,$ $R = \{2s-1$, $4n-(2s-1)$, $2p_1,2p_2,\ldots,2p_{k-2}\}$ and $S = \{ 2n-(2s-1)$, $2n+2s-1, 2p_1, 2p_2, \ldots, 2p_{k-2} \},$ $\Theta_{8n,2,n+t}(C_{8n}(R))$ = $\Theta_{8n,2,t}(C_{8n}(S))$ = $\Theta_{8n,2,3n+t}(C_{8n}(R))$ and $\Theta_{8n,2,n+t}(C_{8n}(S))$ = $\Theta_{8n,2,t}$ $(C_{8n}(R))$ = $\Theta_{8n,2,3n+t}$ $(C_{8n}(S))$ where $n,s,p_1,p_2,\ldots,p_{k-2} \in \mathbb{N}$ and $\gcd(p_1,p_2,\ldots,p_{k-2})$ = 1. }
\end{theorem}
\begin{proof}\quad We have $V_{n,r}(C_{n}(R))$ = $\{\Theta_{n,m,t}(C_n(R)):~ t = 0,1,\ldots,\frac{n}{m}-1\}.$ Let $n \geq 2,$ $k \geq 3,$ $n \neq 2s-1$ and $1 \leq 2s-1 \leq 2n-1$. To simplify our work, let $R$ = $\{2, 2s-1, 4n-(2s-1), 4n+2s-1, 8n-(2s-1), 8n-2\}$ = $R \cup (8n-R)$ and $S$ = $\{ 2, 2n-(2s-1), 2n+2s-1, 6n-(2s-1), 6n+2s-1, 8n-2 \}$ = $S \cup (8n-S).$ 
	
For $0 \leq t \leq 4n-1,$ $\Theta_{8n,2,n+t}(R)$ = $\Theta_{8n,2,n+t}(\{2, 2s-1, 4n-(2s-1), 4n+2s-1, 8n-(2s-1), 8n-2\})$ = $\{2, 2s-1+2n+2t, 4n-(2s-1)+2n+2t, 4n+2s-1+2n+2t, 8n-(2s-1)+2n+2t, 8n-2\}$ = $\{2, 2n+2s-1+2t, 6n-(2s-1)+2t, 6n+2s-1+2t, 2n-(2s-1)+2t, 8n-2\}$ = $\Theta_{8n,2,t}(S),$ $\Theta_{8n,2,3n+t}(R)$ = $\Theta_{8n,2,3n+t}(\{2, 2s-1, 4n-(2s-1), 4n+2s-1, 8n-(2s-1), 8n-2\})$ = $\{2, 2s-1+6n+2t,$ $4n-(2s-1)+6n+2t, 4n+2s-1+6n+2t, 8n-(2s-1)+6n+2t, 8n-2\}$ = $\{2, 6n+2s-1+2t,$ $2n-(2s-1)+2t, 2n+2s-1+2t, 6n-(2s-1)+2t, 8n-2\}$ = $\Theta_{8n,2,t}(S).$ Similarly, we can prove $\Theta_{8n,2,n+t}(S)$ = $\Theta_{8n,2,t}(R)$ = $\Theta_{8n,2,3n+t}(S).$ This also implies, $\Theta_{8n,2,n+t}(C_{8n}(R))$ = $\Theta_{8n,2,t}(C_{8n}(S))$ = $\Theta_{8n,2,3n+t}(C_{8n}(R)),$ $\Theta_{8n,2,n+t}(C_{8n}(S))$ = $\Theta_{8n,2,t}(C_{8n}(R))$ = $\Theta_{8n,2,3n+t}(C_{8n}(S))$ when $R$ = $\{2,2s-1,$ $4n-(2s-1)\}$ and $S$ = $\{ 2, 2n-(2s-1), 2n+2s-1 \}.$ Then the result follows from Remark \ref{r12}.
\end{proof}   

Corresponding to Results \ref{a17c}, \ref{a17f} and \ref{a17fb}, we obtain Results \ref{a23}, \ref{a24} and \ref{a25} on Type-2 group of $C_n(R)$ w.r.t. $r$ = 2, respectively as given below.

\begin{theorem}{\rm \quad \label{a23} For $n \geq 2,$ $k \geq 3,$ $1 \leq 2s-1 \leq 2n-1,$ $n \neq 2s-1,$ $R$ = $\{2, 2s-1, 4n-(2s-1)\}$ and $S$ = $\{2, 2n-(2s-1), 2n+2s-1\},$ $C_{8n}(R)$ and $C_{8n}(S)$ are Type-2 isomorphic w.r.t. $r$ = 2,  $T2_{8n,2}(C_{8n}(R))$ = $T2_{8n,2}(C_{8n}(S))$ = $\{C_{8n}(R), C_{8n}(S)\}$ and $(T2_{8n,2}(C_{8n}(R)), \circ)$ = $(T2_{8n,2}(C_{8n}(S)), \circ)$ is a Type-2 group, $n,s\in\mathbb{N}$. }
\end{theorem}
\begin{proof}\quad  Using Theorem \ref{a17c}, $C_{8n}(R)$ and $C_{8n}(S)$ are Type-2 isomorphic w.r.t. $r = 2.$ This implies, $C_{8n}(R),C_{8n}(S)\in T2_{8n,2}(C_{8n}(R))$ and $T2_{8n,2}(C_{8n}(R))$ = $T2_{8n,2}(C_{8n}(S))$ using property $(h)$. Our aim is to prove that $T2_{8n,2}(C_{8n}(R))$ =  $\{C_{8n}(R), C_{8n}(S) \}$ = $T2_{8n,2}(C_{8n}(S))$. Now, it is enough to prove that $\Theta_{8n,2,t}(C_{8n}(R)) \notin T2_{8n,2}(C_{8n}(R))$ for $t \neq 0,n,2n,3n,$ $0 \leq t \leq 4n-1.$ That is to prove that $\Theta_{8n,2,t}(C_{8n}(R))\notin T2_{8n,2}(C_{8n}(R))$ for all $t$ = $jn+ q,$ $0 \leq j \leq 3,$ $1 \leq  q \leq n-1$ since $0 \leq t \leq 4n-1.$
	
Let $t$ = $jn+q,$ $0 \leq j \leq 3$ and $1 \leq q \leq n-1.$ Consider $\Theta_{8n,2,jn+q}(R \cup (8n-R))$ = $\Theta_{8n,2,jn+q}(\{2,$ $2s-1, 4n-(2s-1), 4n+(2s-1), 8n-(2s-1), 8n-2\})$ = $\{2, 8n-2\} \cup (2(jn+q)+\{2s-1, 4n-(2s-1),$ $4n+(2s-1), 8n-(2s-1)\})$ = $\{2, 8n-2, 2(jn+q)+2s-1, 2(jn+q)+4n-(2s-1), 2(jn+q)+4n+(2s-1),$ $2(jn+q)+8n-(2s-1)\}$ = $\{2, 2jn+ 2q+2s-1, 2jn+2q+4n-(2s-1), 2jn+2q+4n+(2s-1),$ $2jn+2q-(2s-1), 8n-2\}$ = $\{2, 2jn+2q-(2s-1), 2jn+2q+2s-1, 2jn+2q+4n-(2s-1),$ $2jn+2q+4n+(2s-1), 8n-2\},$ under arithmetic modulo $8n.$ 
	
Elements of the above set satisfies the symmetric equidistance condition if and only if $2jn+ 2q-(2s-1)$ + $(2jn+ 2q +4n+(2s-1))$ = $8n$ and $2jn+ 2q +2s-1$ + $(2jn+ 2q +4n-(2s-1))$ = $8n$ if and only if $4jn+ 4q$ = $4n$ if and only if $jn+ q$ = $n$ which is not possible since $0 \leq j \leq 3$ and $1 \leq q \leq n-1.$ This implies, for $0 \leq j \leq 3$ and $1 \leq q \leq n-1,$ the elements of the set $\Theta_{8n,2,jn+ q}(R \cup (8n-R))$ does not satisfy the symmetric equidistance condition and hence $\Theta_{8n,2,jn+ q}(C_{8n}(R \cup (8n-R)))$ $\neq$ $C_{8n}(T \cup (8n-T))$ for any $T$ $\subseteq$ $[1, 4n].$ This implies, $T2_{8n,2}(C_{8n}(R))$ = $\{C_{8n}(R), C_{8n}(S)\}$ = $T2_{8n,2}(C_{8n}(S))$. 
	
Clearly, $(T2_{8n,2}(C_{8n}(R)), \circ)$ = $(T2_{8n,2}(S), \circ)$ is the Type-2 group of $C_n(R)$ w.r.t. $r$ = 2. 
\end{proof}

\begin{cor}{\rm \quad \label{a23b} For $n \geq 2,$ $k \geq 3,$ $1 \leq 2s-1 \leq 2n-1,$ $n \neq 2s-1,$ $n,s,p_1,p_2,\ldots,p_{k-2} \in \mathbb{N}$, $\gcd(p_1,p_2,\ldots,p_{k-2})$ = 1, $R$ = $\{2s-1, 4n-(2s-1), 2p_1,2p_2,\ldots,2p_{k-2}\}$ and $S$ = $\{2n-(2s-1), 2n+2s-1,$ $2p_1, 2p_2, \ldots, 2p_{k-2} \},$ if for a given set of values of $k,p_1,p_2,\ldots,p_{k-2},s$ and $n$, $C_{8n}(R)$ and $C_{8n}(S)$ are Type-2 isomorphic w.r.t. $m$ = 2, then $T2_{8n,2}(C_{8n}(R))$ = $T2_{8n,2}(C_{8n}(S))$ = $\{C_{8n}(R), C_{8n}(S)\}$ and $(T2_{8n,2}(C_{8n}(R)), \circ)$ = $(T2_{8n,2}(C_{8n}(S)), \circ)$ is the Type-2 group of $C_n(R)$ w.r.t. $r$ = 2. \hfill $\Box$}
\end{cor}

\begin{cor}{\rm \cite{v20} \quad \label{a24} For $n \geq 4,$ $k \geq 3,$ $1 \leq 2s-1 \leq 2^{n-2}-1,$ $n,s\in \mathbb{N}$, $R$ = $\{2, 2s-1, 2^{n-1}-(2s-1)\}$ and $S$ = $\{ 2, 2^{n-2}-(2s-1), 2^{n-2}+2s-1\}$, $C_{8n}(R)$ and $C_{8n}(S)$ are Type-2 isomorphic w.r.t. $r$ = 2, $T2_{2^n,2}(C_{2^n}(R))$ = $\{C_{8n}(R), C_{8n}(S)\}$ = $T2_{2^n,2}(C_{2^n}(S))$ and $(T2_{2^n,2}(C_{2^n}(R)), \circ)$ = $(T2_{2^n,2}(C_{2^n}(S)), \circ)$ is the Type-2 group of $C_n(R)$ w.r.t. $r$ = 2. \hfill $\Box$}
\end{cor}

\begin{cor}{\rm \cite{v20} \quad \label{a25} For $n \geq 4,$ $k \geq 3,$ $1 \leq 2s-1 \leq 2^{n-2}-1,$ $n,s,p_1,p_2,\ldots,p_{k-2} \in \mathbb{N}$, $\gcd(p_1,p_2,\ldots,p_{k-2})$ = 1, $R$ = $\{2s-1, 2^{n-1}-(2s-1),$ $2p_1, 2p_2, \ldots, 2p_{k-2} \}$ and $S$ = $\{ 2^{n-2}-(2s-1),$ $2^{n-2}+2s-1, 2p_1, 2p_2, \ldots, 2p_{k-2} \},$ if for a given set of values of $k,p_1,p_2,\ldots,p_{k-2},s$ and $n$, $C_{8n}(R)$ and $C_{8n}(S)$ are Type-2 isomorphic w.r.t. $m$ = 2, then $T2_{2^n,2}(C_{2^n}(R))$ = $\{C_{8n}(R), C_{8n}(S)\}$ = $T2_{2^n,2}(C_{2^n}(S))$ and $(T2_{2^n,2}(C_{2^n}(R)), \circ)$ = $(T2_{2^n,2}(C_{2^n}(S)), \circ)$ is the Type-2 group of $C_n(R)$  w.r.t. $r$ = 2. \hfill $\Box$}
\end{cor}

\section{Examples of Type-1 and Type-2 isomorphic circulant graphs and related groups}

In the previous section, we obtained Type-1 and Type-2 isomorphic circulant graphs of $C_n(R)$ w.r.t. to $r$ = 2 and discussed Type-1 and Type-2 groups of $C_n(R)$ which are formed on the set of all Adam's isomorphic circulant graphs of $C_n(R)$ and on the set of all Type-2 isomorphic circulant graphs of $C_n(R)$ w.r.t. $m$ = 2, respectively. In this section, we bring out more examples of Type-1 and Type-2 isomorphic circulant graphs and related groups to enlighten the topic. We start with an example.

\begin{exm} \quad \label{e1} {\rm For the circulant graph $C_{54}(R_1)$ with $R_1$ = $\{2,3,16,20\}$,
\\		
 $(i)$ $T2_{54,3}(C_{54}(R_1))$ = $\{C_{54}(R_1)$, $C_{54}(R_2), C_{54}(R_3)\}$ = $T2_{54,3}(C_{54}(R_2))$ = $T2_{54,3}(C_{54}(R_3))$, 
\\ 
$(ii)$ $T1_{54}(C_{54}(R_1))$ = $\{C_{54}(R_1), C_{54}(8,10,15,26),$ $C_{54}(4,14,21,22)\}$ = $\{C_{54}(xR_1): x = 1, 5, 7\}$,
 
 $T1_{54}(C_{54}(R_2))$ = $\{C_{54}(R_2), C_{54}(2,15,16,20), C_{54}(8,10,21,26)\}$ = $\{C_{54}(xR_2): x = 1, 5, 7\}$ and 
 
 $T1_{54}(C_{54}(R_3))$ = $\{C_{54}(R_3), C_{54}(4,14,15,22), C_{54}(2,16,20,21)\}$ = $\{C_{54}(xR_3): x = 1, 5, 7\}$ 
 \\
 where $R_2$ = $\{3,4,14,22\}$ and $R_3$ = $\{3,8,10,26\}$.}   
\end{exm}
\noindent
$(i)$ Here, $n$ = 54, $r = 3$ and $m$ = $\gcd(n, r)$ = $\gcd(54, 3)$ = 3 = $r$. 

Let $S_1$ = $R_1 \cup (54-R_1)$ = $\{2,3,16,20, 34,38,51,52\}$, $S_2$ = $R_2 \cup (54-R_2)$ = $\{3,4,14,22,$ $32,40,50,51\}$ and $S_3$ = $R_3 \cup (54-R_3)$ = $\{3,8,10,26, 28,44,46,51\}$.

\noindent
We have $V_{54,3}(C_{54}(2,3,16,20)) = \{\Theta_{54,3,t}(C_{54}(2,3,16,20)): t = 0,1,...,\frac{54}{\gcd(54,3)}-1 = 17\}$.

 For various values of $t$, it is easy to see from Table 7
 that $\Theta_{54,3,i+6t}(C_{54}(S_1))$ = $\Theta_{54,3,i}(C_{54}(S_1))$ and $\Theta_{54,3,2(j-1)}(C_{54}(S_1))$ = $C_{54}(S_j)\in V_{54,3}(C_{54}(S_1))$, $0 \leq i \leq 5$, $0 \leq i+6t \leq$ $\frac{54}{\gcd(54,3)}-1 = 17$ and $1 \leq j \leq 3$. This implies, $V_{54,3}(C_{54}(2,3,16,20))$ = $\{\Theta_{54,3,t}(C_{54}(2,3,16,$ $20)): t = 0,1,...,5 \}$ and $C_{54}(S_i)$ (= $C_{54}(R_i)$) are isomorphic circulant graphs for $i$ = 1 to 3. Moreover, 2 is the only relative prime to 3 between 1 and 3 and thereby Adam's isomorphic circulant graphs of $S_1, S_2, S_3$ are 

$2S_1$ = $2 \times \{2,3,16,20, 34,38,51,52\}$ = $\{4,6,32,40, 14,22,48,50\}$ $\neq$ $S_2,S_3$,

$2S_2$ = $2 \times \{3,4,14,22, 32,40,50,51\}$ = $\{6,8,28,44, 10,26,46,48\}$ $\neq$ $S_1,S_3$ and 

$2S_3$ = $2 \times \{3,8,10,26, 28,44,46,51\}$ = $\{6,16,20,52, 2,34,38,48\}$ $\neq$ $S_1,S_2$. 

\vspace{.1cm}
$\Rightarrow$ $C_{54}(R_i)$ and $C_{54}(R_j)$ are not Adam's isomorphic for $i \neq j$ and $1 \leq i,j \leq 3$.

\vspace{.1cm}
$\Rightarrow$ $C_{54}(R_i)$ are Type-2 isomorphic w.r.t. $r = 3$ for $i$ = 1 to 3 and 

$T2_{54,3}(C_{54}(R_1))$~ = $\{C_{54}(R_1),$~ $C_{54}(R_2),~ C_{54}(R_3)\}$ = $T2_{54,3}(C_{54}(R_2))$ ~= $T2_{54,3}(C_{54}(R_3))$ where $(T2_{54,3}(C_{54}(2,3,16,20)), \circ)$ is a subgroup of $(V_{54,3}(C_{54}(2,3,16,20)), \circ)$. 

\vspace{.2cm}
\noindent
$(ii)$ We have

$T1_{54}(C_{54}(R_1))$ = $\{C_{54}(xR_1): x\in\varphi_{54}(R_1)\}$
= $\{C_{54}(xR_1): x = 1, 5, 7, 11, 13, 17, 19, 23\}$

\hspace{2.25cm} = $\{C_{54}(R_1), C_{54}(8,10,15,26),$ $C_{54}(4,14,21,22)\}$ = $\{C_{54}(xR_1): x = 1, 5, 7\}$.

Similarly, we get the results in the other cases. \hfill $\Box$

\begin{table}\label{5} 
\caption{\small{Calculation of $\Theta_{54,3,t}(R_1 \cup (54-R_1))$, $R_1$ = $\{2,3,16,20\}$, $t$ = 0 to 6.}}
\begin{center}
\scalebox{0.8}{
\begin{tabular}{||c||c||c|c|c|c|c|c|c|c|c||} \hline \hline
~ \hspace{.05cm} $t$ \hspace{.1cm} & \backslashbox{$\Theta_{54,3,t}(x)$}{Jump \\ size \\ $x$}
& \hspace{.1cm} 2 \hspace{.1cm} & \hspace{.1cm} 3 \hspace{.1cm} & \hspace{.1cm} 16 \hspace{.1cm} & \hspace{.1cm} 20 \hspace{.1cm} & \hspace{.1cm} 34 \hspace{.1cm} & \hspace{.1cm} 38 \hspace{.1cm} & \hspace{.1cm} 51 \hspace{.1cm} & \hspace{.1cm} 52 \hspace{.1cm} & -- \\\hline \hline 
& & &  &   &  &  & & & &  \\
t & $\Theta_{54,3,t}(x)$ & 2+6t & 3 & 16+3t & 20+6t & 34+3t & 38+6t & 51 & 52+3t & {$T1$ or  $T2$ or  NS} \\ \hline \hline
& & &  &   &  &  & & & &  \\
0 & $\Theta_{54,3,0}(x)$ & 2 & 3 & 16 & 20 & 34 & 38 & 51 & 52 & Identity  \\\hline
 & & &  &   &  &  & & & & \\
1 & $\Theta_{54,3,1}(x)$ & 8 & 3 & 19 & 26 & 37 & 44 & 51 & 1 & NS \\\hline 
& & &  &   &  &  & & & & \\
2 & $\Theta_{54,3,2}(x)$ & 14 & 3 & 22 & 32 & 40 & 50 & 51 & 4  & Yes (Type-2) \\\hline
& & &  &   &  &  & & &  & \\
3 & $\Theta_{54,3,3}(x)$ & 20 & 3 & 25 & 38 & 43 & 2 & 51 & 7 & NS \\\hline
& & &  &   &  &  & & & & \\
4 & $\Theta_{54,3,4}(x)$ & 26 & 3 & 28 & 44 & 46 & 8 & 51 & 10 & Yes (Type-2) \\\hline
& & &  &   &  &  & & & & \\
5 & $\Theta_{54,3,5}(x)$ & 32 & 3 & 31 & 50 & 49 & 14 & 51 & 13 & NS \\\hline
& & &  &   &  &  & & & & \\
6 & $\Theta_{54,3,6}(x)$ & 38 & 3 & 34 & 2 & 52 & 20 & 51 & 16 & Identity \\\hline\hline
\end{tabular}}
\end{center}
\footnotesize{T1: Type-1; T2: Type-2 isomorphic w.r.t. $r$ = 3; NS: Non-symmetric.}
\end{table} 

\begin{prm} \quad \label{b3} {\rm Find $(i)$ Type-2 group of $C_{81}(3,7,20,34)$ w.r.t.  $r$ = 3 and $(ii)$ Type-1 group of $C_{81}(3,7,20,34)$.} 
\end{prm}

\noindent
{\bf Solution.}\quad $(i)$ Here, $n$ = 81, $r$ = 3. This implies, $m$ = $\gcd(n, r)$ = $\gcd(81, 3)$ = 3. 
Let $R_1$ = $\{3,7,20,34\}$ and $S_1$ = $R_1 \cup (81 - R_1)$ = $\{3,7,20,34, 47,61,74,78\}$. We calculate $V_{81,3}(C_{81}(R_1))$ = $\{\Theta_{81,3,t}(C_{81}(R_1)):$ $t$ = 0 to $\frac{n}{\gcd(n,r)}$-1 = 26$\}$. From Table 8, we get  $\Theta_{81,3,0}(S_1)$ = $S_1$, $\Theta_{81,3,3}(S_1)$ = $\{3,11,16,38, 43,65,70,78\}$,  $\Theta_{81,3,6}(S_1)$ = $\{2,3,~25,29,~ 52,56,~78,79\}$, $\Theta_{81,3,i+9t}(S_1)$ = $\Theta_{81,3,i}(S_1)$, $\Theta_{81,3,3(j-1)}(S_1)$ = $S_j$ and $C_{81}(S_j)\in V_{81,3}(C_{81}(S_1))$, $1 \leq j \leq 3$ and $0 \leq i \leq 8$. 

Let $R_2$ = $\{3,11,16,38\}$ and $R_3$ = $\{2,3,25,29\}$ so that $S_2$ = $R_2 \cup (81 - R_2)$ and $S_3$ = $R_3 \cup (81 - R_3)$. This implies, $\Theta_{81,3,i+9t}(C_{81}(R_1))$ = $\Theta_{81,3,i}(C_{81}(R_1))$, $\Theta_{81,3,3(j-1)}(C_{81}(R_1))$ = $C_{81}(R_j)$,  $V_{81,3}(C_{81}(R_k))$ = $\{\Theta_{81,3,t}(C_{81}(R_k)):$ $t$ = 0 to 8$\}$ and $C_{81}(R_j)\in V_{81,3}(C_{81}(R_k))$, $1 \leq j,k \leq 3$ and $0 \leq i \leq 8$. This implies, $C_{81}(R_i)$ are isomorphic circulant graphs for $i$ = 1 to 3. Moreover, 2 is the only integer relative prime to 3 between 1 and 3 and thereby their Adam's isomorphic circulant graphs are    

$2S_1$ = $2 \times \{3,7,20,34, 47,61,74,78\}$ = $\{6,14,40,68, 13,41,67,75\}$ $\neq$ $S_2, S_3$,

$2S_2$ = $2 \times \{3,11,16,38, 43,65,70,78\}$ = $\{6,22,32,76, 5,49,59,75\}$ $\neq$ $S_1, S_3$ and 

$2S_3$ = $2 \times \{2,3,25,29, 52,56,78,79\}$ = $\{4,6,50,58, 23,31,75,77\}$ $\neq$ $S_1, S_2$. 

\vspace{.1cm}
$\Rightarrow$ $C_{81}(R_i)$ and $C_{81}(R_j)$ are not Adam's isomorphic for $i \neq j$ and $1 \leq i,j \leq 3$.

\vspace{.1cm}
$\Rightarrow$ $C_{81}(R_i)$ are Type-2 isomorphic w.r.t. $r = 3$ for $i$ = 1 to 3.

$\Rightarrow$ $T2_{81,3}(C_{81}(R_1))$ = $\{C_{81}(R_1), C_{81}(R_2), C_{81}(R_3)\}$ = $T2_{81,3}(C_{81}(R_2))$ = $T2_{81,3}(C_{81}(R_3))$ follows from property $(f)$. Thus, $(T2_{81,3}(C_{81}(3,7,20,34)), \circ)$ is the required group which is a subgroup of $(V_{81,3}(C_{81}(3,7,20,34)), \circ)$.  

\vspace{.2cm}
\noindent
$(ii)$ Consider,

$T1_{81}(C_{81}(R_1))$ = $\{C_{81}(xR_1): x\in\varphi_{81}\}$ 

\hspace{2.25cm} = $\{C_{81}(xR_1): x = 1,2, 4,5, 7,8, 10,11, 13,14, 16,17, 19,20,$ 

\hfill $22,23, 25,26, 28,29, 31,32, 34,35, 37,38, 39,40\}$

\hspace{2.25cm}  = $\{C_{81}(R_1), C_{81}(6, 13, 14, 40), C_{81}(1, 12, 26, 28), C_{81}(8, 15, 19, 35), C_{81}(5, 21, 22, 32),$

\hfill $C_{81}(2, 24, 25, 29), C_{81}(11, 16, 30, 38), C_{81}(4, 23, 31, 33), C_{81}(10, 17, 37, 39)\}$ 
 
\hspace{2.25cm}  = $\{C_{81}(xR_1): x = 1,2, 4,5, 7,8, 10,11, 13\}$.

$\therefore$ $(T1_{81}(C_{81}(R_1)), \circ)$ =  $(\{C_{81}(xR_1): x = 1,2, 4,5, 7,8, 10,11, 13\}, \circ)$ is the Type-1 group of $C_{81}(R_1)$ where $R_1$ = $\{3, 7, 20, 34\}$. \hfill $\Box$

\begin{table} \label{6} 
\caption{\small{Calculation of $\Theta_{81,3,t}(R \cup (81-R))$, $R$ = $\{3,7,20,34\}$, $r$ = 0 to 8.}}
\begin{center}
\scalebox{0.75}{
\begin{tabular}{||c||c||c|c|c|c|c|c|c|c|c||} \hline \hline
~ \hspace{.05cm} $t$ \hspace{.1cm} & \backslashbox{$\Theta_{81,3,t}(x)$}{Jump size \\ $x$}
& \hspace{.1cm} 3 \hspace{.1cm} & \hspace{.1cm} 7 \hspace{.1cm} & \hspace{.1cm} 20 \hspace{.1cm} & \hspace{.1cm} 34 \hspace{.1cm} & \hspace{.1cm} 47 \hspace{.1cm} & \hspace{.1cm} 61 \hspace{.1cm} & \hspace{.1cm} 74 \hspace{.1cm} & \hspace{.1cm} 78 \hspace{.1cm} & -- \\\hline \hline
& & &  &   &  &  & & & &  \\
t & $\Theta_{81,3,t}(x)$ & 3 & 7+3t & 20+6t & 34+3t & 47+6t & 61+3t & 74+6t & 78 & {$T1$ or  $T2$ or  NS} \\ \hline \hline
& & &  &   &  &  & & & &  \\
0 & $\Theta_{81,3,0}(x)$ & 3 & 7 & 20 & 34 & 47 & 61 & 74 & 78 & Identity  \\\hline
 & & &  &   &  &  & & & & \\
1 & $\Theta_{81,3,1}(x)$ & 3 & 10 & 26 & 37 & 53 & 64 & 80 & 78 & NS \\\hline 
& & &  &   &  &  & & & & \\
2 & $\Theta_{81,3,2}(x)$ & 3 & 13 & 32 & 40 & 59 & 67 & 5 & 78 & NS \\\hline
& & &  &   &  &  & & &  & \\
3 & $\Theta_{81,3,3}(x)$ & 3 & 16 & 38 & 43 & 65 & 70 & 11 & 78 & Yes (Type-2) \\\hline
& & &  &   &  &  & & & & \\
4 & $\Theta_{81,3,4}(x)$ & 3 & 19 & 44 & 46 & 71 & 73 & 17 & 78 & NS \\\hline
& & &  &   &  &  & & & & \\
5 & $\Theta_{81,3,5}(x)$ & 3 & 22 & 50 & 49 & 77 & 76 & 23 & 78 & NS \\\hline
& & &  &   &  &  & & & & \\
6 & $\Theta_{81,3,6}(x)$ & 3 & 25 & 56 & 52 & 2 & 79 & 29 & 78 & Yes (Type-2) \\\hline
& & &  &   &  &  & & & & \\
7 & $\Theta_{81,3,7}(x)$ & 3 & 28 & 62 & 55 & 8 & 1 & 35 & 78 & NS \\\hline
& & &  &   &  &  & & & & \\
8 & $\Theta_{81,3,8}(x)$ & 3 & 31  & 68 & 58 & 14 & 4 & 41 & 78 & NS  \\\hline\hline
\end{tabular}}
\end{center}
\footnotesize{T1: Type-1; T2: Type-2 isomorphic w.r.t. $r$ = 3; NS: Non-symmetric.}
\end{table} 

\vspace{.2cm}
While defining Type-2 isomorphism of $C_n(R)$ w.r.t.  $m$, we consider that $r\in R$ and $m > 1$ is a divisor of $\gcd(n, r)$ because of the following. 

\begin{enumerate}
\item [\rm (1)] Let $m > 1$ be a divisor of $\gcd(n, r)$, $C_n(R)$ $\cong$ $C_n(S)$, $R \neq S$, $lm \notin R$, $l,m,n,r\in \mathbb{N}$, $|R| = |S| \geq 2$ and $C_n(R \cup \{ m \})$ and $C_n(S \cup \{ m \})$ be Type-2 isomorphic w.r.t. $m$. Then, for $q\in\mathbb{N}$, $qm\in\mathbb{Z}_{\frac{n}{2}}$, $\gcd(n, qm) = m_1m > m$ and $\gcd(m, m_1) = 1$, $C_n(R \cup \{ qm \})$ and $C_n(S \cup \{ qm \})$ may be of Type-2 isomorphic w.r.t. $qm$, $m_1 \geq 2$ and $m_1\in\mathbb{N}$.
\end{enumerate}

\begin{exm}\quad \label{e2} {\rm Circulant graphs $C_{48}(1,2,23)$ and $C_{48}(2,11,13)$ are isomorphic of Type-2 w.r.t. $r = 2$  and also $C_{48}(1,6,23)$ and $\theta_{48,2,6}(C_{48}(1,6,23))$ = $C_{48}(6,11,13)$ (See Problem \ref{p4}.) which are given in Figures 15 and 16.  

 The second set of graphs is obtained from the first set by replacing 2 by 6. Here, $r = m = 2$, $\gcd(n, qm) = \gcd(48, 6)$ = 6 = $2\times 3$, $q = 3 = m_1$ and $\gcd(m, m_1) = \gcd(2, 3) = 1$. See Problem \ref{p4} for their isomorphism of Type-2.  \hfill $\Box$}
\end{exm}
\begin{exm} \label{e3} {\rm Circulant graphs $C_{54}(1,3,17,19)$, $C_{54}(3,7,11,25)$ and $C_{54}(3,5,13,23)$ are given in Figures 21, 22 and 23 and are Type-2 isomorphic w.r.t. $r = 3$, see the Annexure, and also $C_{54}(1,6,17,19)$, $C_{54}(6,7,11,25)$ and $C_{54}(5,6,13,23)$. }

{\rm The second set of graphs is obtained from the first set by replacing 3 by 6. Here, $r = m = 3$, $\gcd(n, qm) = \gcd(54, 6) = 6 = 2\times 3$, $q = 2 = m_1$ and $\gcd(m, m_1) = \gcd(3, 2) = 1$. Their Type-2 isomorphism w.r.t. $r$ = 3 can be verified similarly  as done in Problem \ref{b3}. \hfill $\Box$}
\end{exm}
\begin{figure}[ht]
	\centerline{\includegraphics[width=6.3in]{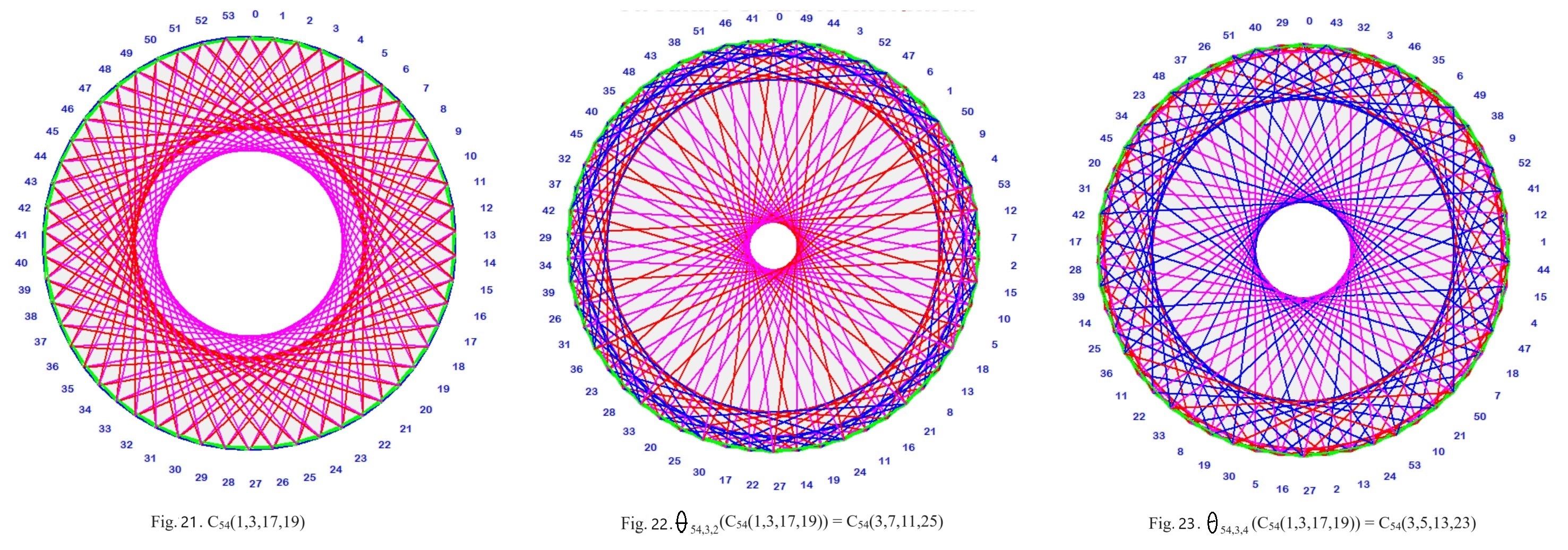}}
\end{figure}
\begin{exm} \label{e4} {\rm Graphs $C_{108}(3,5,31,41)$, $C_{108}(3,7,29,43)$ and $C_{108}(3,17,19,53)$ are Type-2 isomorphic circulant graphs w.r.t. $r = 3$, see the Annexure, and also are graphs $C_{108}(5,12,31,41)$, $C_{108}(7,12,29,43)$ and $C_{108}(12,17,19,53)$. 

The second set of graphs is obtained from the first set by replacing 3 by 12. Here, $r = m$ = 3, $\gcd(n, qm) = \gcd(108, 12) = 12 = 3\times 4$, $q = 4 = m_1$ and $\gcd(m, m_1) = \gcd(3, 4) = 1$. \hfill $\Box$}
\end{exm}

\begin{exm}\quad \label{e5} {\rm Graphs $C_{108}(3,4,32,40)$, $C_{108}(3,16,20,52)$ and $C_{108}(3,8,28,44)$ are Type-2 isomorphic w.r.t. $r$ = 3, see the Annexure, and also $C_{108}(4,12,32,40)$, $C_{108}(12,16,20,52)$ and $C_{108}(8,12,28,44)$. 

The second set of graphs is obtained from the first set by replacing 3 by 12. Here, $r = m$ = 3, $\gcd(n, qm) = \gcd(108, 12)$ = 12 = $3\times 4$, $q = 4 = m_1$ and $\gcd(m, m_1) = \gcd(3, 4) = 1$. \hfill $\Box$}
\end{exm}

\begin{enumerate}
	\item [\rm (2)]  {\rm Let $m > 1$ be a divisor of $\gcd(n,r)$, $C_n(R) \cong C_n(S)$, $R \neq S$, $lm \notin R$, $l,m,n,r\in \mathbb{N}$, $|R| = |S| \geq 2$ and $C_n(R \cup \{ m \})$ and $C_n(S \cup \{ m \})$ be Type-2 isomorphic w.r.t. $m$. Then, for $m_1,q\in\mathbb{N}$ $\ni$ $\gcd(m, m_1) >$ 1, $\gcd(n, qm)$ = $m_1m > m$, $qm\in\mathbb{Z}_{\frac{n}{2}}$ and $m_1 \geq 2$, $C_n(R \cup \{ qm \})$ and $C_n(S \cup \{ qm \})$ need not be of Type-2 isomorphic w.r.t. $qm$.}
\end{enumerate}
 
\begin{exm}\quad \label{e6} {\rm $C_{48}(1,4,23)$ and $C_{48}(4,11,13)$ are Adam's isomorphic (and not of Type-2 w.r.t.  $m$ = 4 or 2) since $C_{48}(4,11,13)$ = $C_{48}(11(1,4,23))$ = $C_{48}(13(1,$ $4,23))$ even though $C_{48}(1,2,23)$ and $C_{48}(2,11,13)$ are  isomorphic of Type-2 w.r.t. $m$ = 2. See Example \ref{e2}. \hfill $\Box$}
\end{exm}

\begin{exm}\quad \label{e7} {\rm Graphs $C_{54}(1,17,18,19)$, $C_{54}(7,11,18,25)$ and $C_{54}(5,13,18,23)$ are Adam's isomorphic (and not of Type-2 w.r.t. $m$ = 18 or 3) since $C_{54}(7,11,$ $18,25)$ = $C_{54}(7(1,17,18,19))$ and $C_{54}(5,13,18,23)$ = $C_{54}(5(1,17,18,19))$ even though the graphs $C_{54}(1,3,17,19)$, $C_{54}(3,7,11,25)$ and $C_{54}(3,5,13,23)$ are Type-2 isomorphic w.r.t. $m$ = 3. See Example \ref{e3} and Table 9. \hfill $\Box$}
\end{exm}

\begin{exm} \label{e8} {\rm Circulant graphs $C_{108}(5,18,31,41)$, $C_{108}(7,18,29,43)$ and $C_{108}(17,18,19,53)$ are Adam's isomorphic (and not of Type-2 w.r.t. $m$ = 18 or 3) since $C_{108}(7,18, 29,43)$ = $C_{108}(13(5,18, 31,41))$ and $C_{108}(17,18,19,53)$ = $C_{108}(11(5,18,31,41))$ whereas graphs $C_{108}(3,5,31,41)$, $C_{108}(3,7,29,43)$ and $C_{108}(3,17,19,53)$ are Type-2 isomorphic  w.r.t. $r$ = 3. See Example \ref{e4} and Table 10. \hfill $\Box$}
\end{exm}
  
\begin{exm} \label{e9} {\rm Circulant graphs $C_{108}(4,18,32,40)$, $C_{108}(16,18,20,52)$ and $C_{108}(8,18,28,44)$ are Adam's isomorphic (and not of Type-2 w.r.t. $m$ = 18 or 3) since $C_{108}(5(4, 18,32,40))$ = $C_{108}(16,18,20,52)$ and $C_{108}(7(4,18,32,40))$ = $C_{108}(8,18,28,44)$ whereas graphs $C_{108}(3,4,32,40)$, $C_{108}(3,16,20,52)$ and $C_{108}(3,8,28,44)$ are Type-2 isomorphic w.r.t. $r$ = 3. See Example \ref{e5} and Table 11. \hfill $\Box$}
\end{exm}

\begin{table} \label{7} 
\caption{Calculation of $rx$, $r\in R$, $x\in\varphi_{54}$ and $rx\in\mathbb{Z}_{54}$.}
\begin{center}
\scalebox{0.9}{
\begin{tabular}{||c||c|c|c|c|c|c|c|c||} \hline \hline
\backslashbox{\\Multiplier $x$}{Jump size $r$}
& \hspace{.1cm} 1 \hspace{.1cm} & \hspace{.1cm} 17 \hspace{.1cm} & \hspace{.1cm} 18 \hspace{.1cm} & \hspace{.1cm} 19 \hspace{.1cm} & \hspace{.1cm} 35 \hspace{.1cm} & \hspace{.1cm} 36 \hspace{.1cm} & \hspace{.1cm} 37 \hspace{.1cm} & \hspace{.1cm} 53 \hspace{.1cm} \\\hline \hline
& &  &   &  &  & & &  \\
5 & 5 & 31 & 36 & 41 & 13 & 18 & 23 & 49  \\\hline
& &  &   &  &  & & & \\
7 & 7 & 11 & 18 & 25 & 29 & 36 & 43 & 47  \\\hline \hline
\end{tabular}}
\end{center}
\end{table} 

\begin{table} \label{8}
\caption{Calculation of $rx$, $r\in R$, $x\in\varphi_{108}$ and $rx\in\mathbb{Z}_{108}$.}
\begin{center}
\scalebox{0.9}{
\begin{tabular}{||c||c|c|c|c|c|c|c|c||} \hline \hline
\backslashbox{\\Multiplier $x$}{Jump size $r$}
& \hspace{.1cm} 5 \hspace{.1cm} & \hspace{.1cm} 18 \hspace{.1cm} & \hspace{.1cm} 31 \hspace{.1cm} & \hspace{.1cm} 41 \hspace{.1cm} & \hspace{.1cm} 67 \hspace{.1cm} & \hspace{.1cm} 77 \hspace{.1cm} & \hspace{.1cm} 90 \hspace{.1cm} & \hspace{.1cm} 103 \hspace{.1cm} \\\hline \hline
& &  &   &  &  & & &  \\
11 & 55 & 90 & 17 & 19 & 89 & 91 & 18 & 53  \\\hline
& &  &   &  &  & & & \\
13 & 65 & 18 & 79 & 101 & 7 & 29 & 90 & 43  \\\hline \hline
\end{tabular}}
\end{center}
\end{table} 

\begin{table} \label{9} 
\caption{Calculation of $rx$, $r\in R$, $x\in\varphi_{108}$ and $rx\in\mathbb{Z}_{108}$.}
\begin{center}
\scalebox{0.9}{
\begin{tabular}{||c||c|c|c|c|c|c|c|c||} \hline \hline
\backslashbox{\\Multiplier $x$}{Jump size $r$}
& \hspace{.1cm} 4 \hspace{.1cm} & \hspace{.1cm} 18 \hspace{.1cm} & \hspace{.1cm} 32 \hspace{.1cm} & \hspace{.1cm} 40 \hspace{.1cm} & \hspace{.1cm} 68 \hspace{.1cm} & \hspace{.1cm} 76 \hspace{.1cm} & \hspace{.1cm} 90 \hspace{.1cm} & \hspace{.1cm} 104 \hspace{.1cm} \\\hline \hline
& &  &   &  &  & & &  \\
5 & 20 & 90 & 52 & 92 & 16 & 56 & 18 & 88  \\\hline
& &  &   &  &  & & & \\
7 & 28 & 18 & 8 & 64 & 44 & 100 & 90 & 80  \\\hline \hline
\end{tabular}}
\end{center}
\end{table} 

\begin{rem} \label{r23} From the above examples, it is noted that  there are Adam's isomorphic circulant graphs $C_n(R)$ and $C_n(S)$ such that $\Theta_{n,r,t}(C_n(R))$ = $C_n(S)$ for some $t\in\mathbb{N}$ with $R \neq S$, $r\in R,S$, $\gcd(n,r) > 1$ and $1 \leq t \leq \frac{n}{m} - 1$. And so a separate study is needed to find such type of isomorphic circulant graphs. 
\end{rem}

Based on the above observations, we propose the following conjectures.

 \begin{con}\quad \label{con2}  {\rm Let $m > 1$ be a divisor of $\gcd(n,r)$, $C_n(R) \cong C_n(S)$, $R \neq S$, $|R| = |S| \geq 2$, $lm \notin R,S$,  and $l,m,n,r\in \mathbb{N}$. Let $m_1,q\in\mathbb{N}$ $\ni$ $\gcd(m, m_1) = 1$, $\gcd(n, qm) = m_1m > m$, $qm\in\mathbb{Z}_{\frac{n}{2}}$, $m_1 \geq 2$ and $C_n(R \cup \{ m \})$ and $C_n(S \cup \{m\})$ be of Type-2 isomorphic w.r.t. $m$. Then, $C_n(R \cup \{ qm \})$ and $C_n(S \cup \{ qm \})$ are Type-2 isomorphic w.r.t. $qm$.} \hfill $\Box$
\end{con}
 
\begin{con}\quad\label{con3} {\rm Let $m > 1$ be a divisor of $\gcd(n,r)$, $C_n(R) \cong C_n(S)$, $R \neq S$, $|R| = |S| \geq 2$, $lm \notin R,S$,  and $l,m,n,r\in \mathbb{N}$. Let $m_1,q\in\mathbb{N}$ $\ni$ $\gcd(m, m_1) > 1$, $\gcd(n, qm) = m_1m > m$, $qm\in\mathbb{Z}_{\frac{n}{2}}$, $m_1 \geq 2$ and $C_n(R \cup \{ m \})$ and $C_n(S \cup \{m\})$ be of Type-2 isomorphic w.r.t. $m$. Then, $C_n(R \cup \{ qm \})$ and $C_n(S \cup \{ qm \})$ are not Type-2 isomorphic w.r.t. $qm$.} \hfill $\Box$
\end{con}

\section{$C_{54}(1, 3, 17, 19)$ and $C_{54}(5, 13, 21, 23)$ are isomorphic but are neither  Type-1 nor Type-2}

In this section, we present isomorphic circulant graphs $C_n(R)$ and $C_n(S)$ which are neither Type-1 nor Type-2. Let $C_n(R)$ and $C_n(S)$ be any two isomorphic circulant graphs such that $R \neq S$ and $|R| = |S| \geq 3$. Then one of the following statements is true.
\begin{enumerate} 
\item  $C_n(R)$ and $C_n(S)$ are Adam's isomorphic and $R \neq S$.  That is  for some $x\in \varphi_n$, $C_n(S) = C_n(xR)$, $x \neq 1,n-1$ and $R \neq S$. In this case, $T2_{n,r}(C_n(R))$ $\cap$ $T2_{n,r}(C_n(S))$ = $\emptyset$, $R \neq S$.

\item  $C_n(R)$ and $C_n(S)$ are isomorphic of Type-2  w.r.t. $m$ where $m > 1$ is a divisor of $\gcd(n, r)$. In this case, $T2_{n,r}(C_n(R))$ = $T2_{n,r}(C_n(S))$, $R \neq S$.

\item There are isomorphic circulant graphs $C_n(R)$ and $C_n(S)$ which are neither Adam's isomorphic  nor Type-2 isomorphic w.r.t. any particular $r\in \mathbb{Z}_\frac{n}{2}.$ That is $C_n(S) \neq C_n(xR)$ for all $x\in \varphi_n$ and also $C_n(S)$ and $C_n(R)$ are not Type-2 isomorphic w.r.t. any particular $r\in \mathbb{Z}_\frac{n}{2}.$ But their isomorphism may be connected by a sequence of Type-2 isomorphisms w.r.t. different $r$'s or Type-2 isomorphisms w.r.t. different $r$'s as well as Adam's isomorphism and in this case, $T2_{n,r}(C_n(R))$ $\cap$ $T2_{n,r}(C_n(S))$ = $\emptyset$. This is illustrated by the following example.
\end{enumerate}

\begin{exm} \quad \label{e10} {\rm  $C_{54}(1,3,17,19)$ and $C_{54}(5,13,21,23)$ are isomorphic but they are neither of Adam's nor of Type-2 w.r.t. 3 (or 21 or w.r.t. any $r$ whose $\gcd$ with 54 is $ > 1$) by the following.}
\end{exm}
\begin{enumerate}
\item [\rm (a)]  $Ad_{54}(C_{54}(1,3,17,19))$ = $\{\varphi_{54,x}(C_{54}(1,3,17,19)): x\in\varphi_{54}\}$ 

 = $\{C_{54}(x(1, 3,17,19)): x = 1,5,7,$ $11,13,17,19,23,25,29,31,35,37,41,43,47,49,53\}$ 

= $\{C_{54}(1,3,17, 19),$ $C_{54}(5,13,15,23),$ $C_{54}(7,$ $11,21,25)\}$ = $\{C_{54}($ $x(1,3,17,19)) : x = 1,5,7\}$. $\Rightarrow$ $C_{54}(5,13,21,23) \notin Ad_{54}( C_{54}(1,3,$ $17,19))$ and thereby $C_{54}(1,3,17,19)$ and $C_{54}(5,13,21,23)$ are not Adam's isomorphic. 

\item [\rm (b)] $V_{54,3}(\{1,3,17,19,35,37,51,53\})$ = $\{\Theta_{54,3,t}(\{1,3,17,19,35,37,51,53\}):$ $t = 0,1,...,18-1\}$ 

= $\{\{1,3,17,19,35,37,51,53\},$ $\{4,3,23,22,41,40,51,5\}$, $\{7,3, 29,25,47,43,51,11\}$, 

\hfill $\{10,3,35,28,$ $53,46,51,17\}$, $\{13,3,41, 31,5,49,51,23\}$, $\{16,3,47,34,11,52,51,29\}\}$ 

= $\{\Theta_{54,3,t}(\{1,3,17,19, 35,$ $37,51,53\}):$ $t$ = $0,1,2,3,4,5\}$. See Table 12. 

\begin{table} \label{10}
\caption{Calculation of $\Theta_{54,3,t}(\{1,3,17,19,35,37,51,53\})$.}
\begin{center}
\scalebox{0.75}{
\begin{tabular}{||c|c|c|c|c|c|c|c|c|c|c||} \hline \hline
~ \hspace{.1cm} $t$ \hspace{.2cm} & \backslashbox{$\Theta_{54,3,t}(x)$}{Jump size \\ $x$}
& \hspace{.2cm} 1 \hspace{.2cm} & \hspace{.2cm} 3 \hspace{.2cm} & \hspace{.2cm} 17 \hspace{.2cm} & \hspace{.2cm} 19 \hspace{.2cm} & \hspace{.2cm} 35 \hspace{.2cm} & \hspace{.2cm} 37 \hspace{.2cm} & \hspace{.2cm} 51 \hspace{.2cm} & \hspace{.2cm} 53 & -- \\\hline \hline
& & &  &   &  &  & & &  & \\
t & $\Theta_{54,3,t}(x)$ & $1+3t$ & 3 & $17+6t$ & $19+3t$ & $35+6t$ & $37+3t$ & 51 & $53+6t$ & $T1$ or $T2$ or NS  \\\hline \hline
 & & &  &   &  &  & & & & \\
0 & $\Theta_{54,3,0}(x)$ & 1 & 3 & 17 & 19 & 35 & 37 & 51 & 53 & Identity  \\\hline
 & & &  &   &  &  & & & & \\
1 & $\Theta_{54,3,1}(x)$ & 4 & 3 & 23 & 22 & 41 & 40 & 51 & 5 & NS  \\\hline 
& & &  &   &  &  & & & & \\
2 & $\Theta_{54,3,2}(x)$ & 7 & 3 & 29 & 25 & 47 & 43 & 51 & 11 & Yes (Type-2) \\\hline 
& & &  &   &  &  & & & & \\
3 & $\Theta_{54,3,3}(x)$ & 10 & 3 & 35 & 28 & 53 & 46 & 51 & 17 & NS  \\\hline 
& & &  &   &  &  & & & & \\
4 & $\Theta_{54,3,4}(x)$ & 13 & 3 & 41 & 31 & 5 & 49 & 51 & 23 & Yes (Type-2)  \\\hline 
& & &  &   &  &  & & & & \\
5 & $\Theta_{54,3,5}(x)$ & 16 & 3 & 47 & 34 & 11 & 52 & 51 & 29 & NS \\\hline\hline 
& & &  &   &  &  & & & & \\
6 & $\Theta_{54,3,6}(x)$ & 19 & 3 & 53 & 37 & 17 & 1 & 51 & 35 & Identity  \\\hline\hline
\end{tabular}}
\end{center}
\footnotesize{T1: Type-1; T2: Type-2 isomorphic w.r.t. $r$ = 3; NS: Non-symmetric.}
\end{table} 

This implies that $\Theta_{54,3,0}(C_{54}(1,3,~17,19)$ = $C_{54}(1,3,~17,19)$, $\Theta_{54,3,2}(C_{54}(1,3,~17,19))$ = $C_{54}(3,7,11,25)$ and $\Theta_{54,3,4}(C_{54}(1,3,17,19))$ = $C_{54}(3,5,13,23)$ are the only circulant graphs of the form $C_{54}(R)$ contained in $V_{54,3}(C_{54}(1,3,17,19))$. Also, see Table 12. This implies that $C_{54}(5,13,21,23)\notin V_{54,3}(C_{54}(1,3,17,19))$ which implies, $C_{54}(5,13,21,23)\notin T2_{54,3}(C_{54}(1,3,17,$ $19))$ since $T2_{n,r}(C_{n}(R)) \subseteq V_{n,r}(C_{n}(R))$ for any $C_{n}(R)$. This implies that $C_{54}(1,3,17,19)$ and $C_{54}(5,13,21,23)$ are not Type-2 isomorphic w.r.t. $r$ = 3. Similarly, we can show that $C_{54}(1,3,17,19)$ and $C_{54}(5,13,21,23)$ are not Type-2 isomorphic w.r.t. $r$ = 21. Here, $\gcd(54, 1)$ = $\gcd(54, 17)$ = $\gcd(54, 19)$ = 1.

\item [\rm (c)] Now, we establish that $C_{54}(1,3,17,19) \cong C_{54}(5,13,21,23)$ by showing $C_{54}(1,$ $3,17,19) \cong C_{54}(3,7,11,25)$ and $C_{54}(3,7,11,25) \cong C_{54}(5,13,21,23)$.

Consider, $\Theta_{54,3,2}(C_{54}(1,3,17,19)) = \Theta_{54,3,2}(C_{54}(1,3,17,19,35,37,51,53))$ = $C_{54}(7,3,29,25,$ $47,43,51,11)$ = $C_{54}(3,7,11,25,29,43,47,51) = C_{54}(3,7,11,25)$ (See Table 12) which implies, $C_{54}(1,3,17,19)$ $\cong$ $C_{54}(3,7,11,25)$. 

Also, we have $\varphi_{54,5}(C_{54}(5,13,21,23))$ = $\varphi_{54,5}(C_{54}(5,13,21,23,31,33,41,49))$ = $C_{54}(5(5,13,$ $21,23,31,33,41,49))$ = $C_{54}(25,11,51,7,47,3,43,29)$ = $C_{54}(3,$ $7,11,25)$ which implies $C_{54}(5,13,$ $21,23)$ and $C_{54}(3,7,11,25)$ are Adam's isomorphic. This implies that $C_{54}(1,3,17,19)$ and $C_{54}(5,13,21,23)$ are isomorphic but they are neither of Adam's isomorphic nor of Type-2 w.r.t. 3 (or 21 or w.r.t. any particular $r$ such that $\gcd(54,r) >$ 1). 
\end{enumerate}

Isomorphic circulant graphs $C_{54}(1,3,17,19)$, $C_{54}(3,7,11,25)$ and $C_{54}(5,13,21,23)$ are shown in Fig. 21, 22, 23. Figures 22 and 26 show $C_{54}(3,7,11,25)$ as $\Theta_{54,3,2}(C_{54}(1,3, 17,19))$ = $C_{54}(3,7, 11,25)$ and $\varphi_{54,5}(C_{54}(5,13, 21,23))$ = $C_{54}(3,7, 11,25)$, respectively. That is graph $C_{54}(3,7,11,25)$ is obtained from graph $C_{54}(5,13,21,23)$ under the transformation $\Theta_{54,3,2}$ as well as under $\varphi_{54,5}$ are shown in Figures 22 and 26, respectively. From the above, we observe the following.
\begin{figure}[ht]
	\centerline{\includegraphics[width=6in]{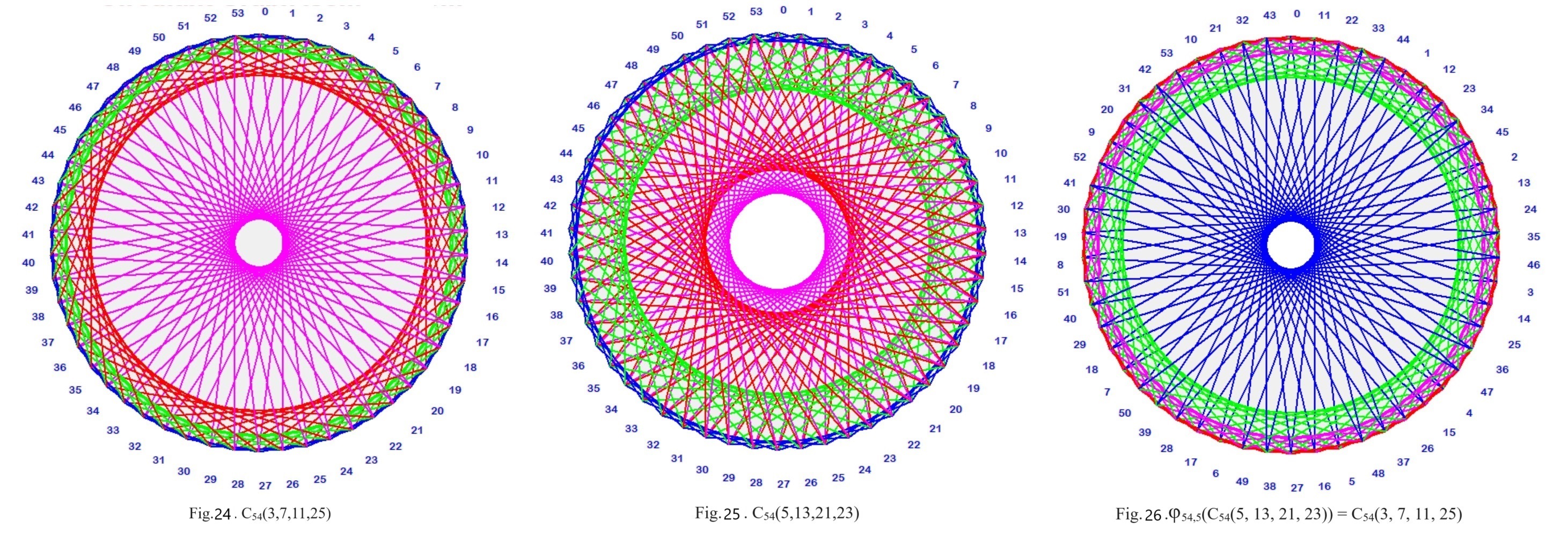}}
\end{figure}
 Given a circulant graph $C_n(R)$, it may be possible to form a sequence of isomorphisms involving Adam's as well as Type-2 w.r.t. different $r$'s and get isomorphic circulant graph(s) $C_n(S)$ which may neither be Adam's isomorphic nor be Type-2 w.r.t. a particular $r$ to $C_n(R).$ Thus a new study is required to find the sequence of isomorphisms involved among isomorphic circulant graphs. 

For any circulant graph $C_{n}(R)$,  $T2_{n,r}(C_{n}(R)) \subseteq V_{n,r}(C_{n}(R))$. In Example \ref{e10}, we could notice that for the circulant graph $C_{54}(1,3,17,19)$, $T2_{54,3}(C_{54}(1,3,17,19)) \subsetneq V_{54,3}(C_{54}(1,3,17,19))$. This need not be the case always. 
In Problem \ref{p6}, for the circulant graph $C_{27}(1,3,8,10)$, we have $T2_{27,3}(C_{27}(1,3,8,10))$ = $V_{27,3}(C_{27}(1,3,8,10))$. Thus, in general, we get the following result.

\begin{theorem}\quad {\rm \label{c40} Let $C_n(R)$ and $C_n(S)$ be Type-2 isomorphic w.r.t. $r$ and $C_n(xR)$ = $C_n(T)$, $r\in R,S$,  $x\in \varphi_n$ and $T \neq R$. Then, $C_n(S)$ and $C_n(T)$ are isomorphic but they are neither of Adam's nor of Type-2 w.r.t. any  $r'$. }
\end{theorem}
\begin{proof} \quad Given, $C_n(R)$ and $C_n(S)$ are Type-2 isomorphic w.r.t. $r$ and $C_n(xR)$ = $C_n(T)$, $r\in R,S$, $x\in \varphi_n$ and $T \neq R$. This implies, $C_n(R)$ $\cong$ $C_n(S)$ $\cong$ $C_n(T)$, $C_n(R),C_n(S)\in T2_{n,r}(C_n(R))$ = $T2_{n,r}(C_n(S))$ and $C_n(R),C_n(T)\in Ad_n(C_n(R))$ = $T1_n(C_n(R))$ = $T1_n(C_n(T))$ where $(T1_n(C_n(R)), \circ')$ and $(T2_{n,r}(C_n(R)), \circ)$ are Abelian groups on $C_n(R)$, see Theorem \ref{a20}.
	
Suppose, $C_n(S)$ and $C_n(T)$ be Adam's isomorphic, $S \neq T$. This implies, $C_n(S),C_n(T)\in T1_n(C_n(S))$ = $T1_n(C_n(T))$, $S \neq T$. And already $C_n(R),C_n(T)\in T1_n(C_n(R))$ = $T1_n(C_n(T))$, $R \neq T$. This implies, $C_n(R),C_n(S),C_n(T)\in T1_n(C_n(R))$ = $T1_n(C_n(S))$ = $T1_n(C_n(T))$, $S \neq T$, $R \neq T$ and $R \neq S$ since $C_n(R)$ and $C_n(S)$ are Type-2 isomorphic w.r.t. $r$. This implies, $C_n(R)$ and $C_n(S)$ are Adam's isomorphic since $(T1_n(C_n(R)), \circ')$ is an Abelian groups on $C_n(R)$, $R \neq S$. This is a contradiction to the given condition that $C_n(R)$ and $C_n(S)$ are Type-2 isomorphic w.r.t. $r$, $r\in R,S$. Hence $C_n(S)$ and $C_n(T)$ are isomorphic but they are not Adam's isomorphic, $S \neq T$. 
	
On the otherhand, suppose $C_n(S)$ and $C_n(T)$ be Type-2 isomorphic w.r.t. $r'$, $r'\in S,T$. This implies that $C_n(S),C_n(T)\in T2_{n,r'}(C_n(S))$ = $T2_{n,r'}(C_n(T))$ using Theorem \ref{a20}, $S \neq T$. By the given condition $C_n(R)$ and $C_n(S)$ are Type-2 isomorphic w.r.t. $r$, $r\in R,S$. This implies, $C_n(R),C_n(S)\in T2_{n,r'}(C_n(R))$ = $T2_{n,r'}(C_n(S))$. Combining the above two statements, we get, $C_n(R),C_n(S),C_n(T)\in T2_{n,r'}(C_n(R))$ = $T2_{n,r'}(C_n(S))$ = $T2_{n,r'}(C_n(T))$, $R \neq S$, $S \neq T$ and $R \neq T$. This implies that $C_n(R)$ and $C_n(T)$ are Type-2 isomorphic w.r.t. $r'$ since $(T2_{n,r'}(C_n(R)), \circ)$ is an Abelian group, $r'\in S,T$. This is a contradiction to the given condition that $C_n(R)$ and $C_n(T)$ are Adam's isomorphic. This implies that our assumption is wrong. This implies, $C_n(S)$ and $C_n(T)$ are isomorphic but they are not Type-2 isomorphic w.r.t. $r'$, $r'\in S,T$. 

Hence we get the result.  
\end{proof}

We propose the following open problems for future study.

\begin{oprm}\quad {\rm \label{op4} Given a circulant graph $C_n(R)$, find all $C_n(S)$ such that $C_n(R)$ $\cong$ $C_n(S)$ and find the sequence of type of isomorphisms involved among $C_n(R)$ and $C_n(S)$ for each $S$. \hfill $\Box$} 
\end{oprm}

\begin{oprm}\quad {\rm \label{op5} Find circulant graphs $C_n(R)$ for which  

(i) $T2_{n,r}(C_{n}(R))$ = $V_{n,r}(C_{n}(R))$;  

(ii)  $T2_{n,r}(C_{n}(R))$ $\neq$ $V_{n,r}(C_{n}(R))$. $i.e.$,  $T2_{n,r}(C_{n}(R)) \subset V_{n,r}(C_{n}(R))$. \hfill $\Box$}
\end{oprm}

\section{Results on isomorphic circulant graphs of Type-2 w.r.t. $r$ = 3,5,7 }  

  In Section 5, we presented results on Type-2 isomorphic circulant graphs of order $n$ w.r.t. $r$ = 2. In this section, we present important results obtained in \cite{vw1}-\cite{vw3} on families of Type-2 isomorphic circulant graphs of order $n$ w.r.t. $r$ = 3,5,7. One can find proof of Theorems \ref{c41}, \ref{c43} and \ref{c45} in \cite{vw1}-\cite{vw3}. Also, these isomorphic circulant graphs are particular forms of isomorphic circulant graphs of order $np^3$ and of Type-2 w.r.t. $r$ = $p$ that are presented in the next section where $p$ is a prime number and $n\in\mathbb{N}$.
 
 \begin{theorem} \cite{vw1} \label{c41} {\rm For $R$ = $\{1, 3, 9n-1, 9n+1\}$, $S$ = $\{3, 3n+1, 6n-1, 12n+1\}$, $T$ = $\{3, 3n-1, 6n+1$, $12n-1\}$ and $n\in\mathbb{N}$, $\theta_{27n,3,n}(C_{27n}(R))$ = $C_{27n}(S)$, $\theta_{27n,3,n}(C_{27n}(S))$ = $C_{27n}(T)$, $\theta_{27n,3,n}(C_{27n}(T))$ = $C_{27n}(R)$ and $C_{27n}(R)$, $C_{27n}(S)$ and $C_{27n}(T)$ are Type-2 isomorphic circulant graphs w.r.t. $r$ = 3. \hfill $\Box$}
 \end{theorem}

 \begin{theorem} \label{c42} {\rm Let $k \geq 3$, $R$ = $\{1, 9n-1, 9n+1, 3p_1, 3p_2, \ldots, 3p_{k-2}\}$, $S$ = $\{3n+1, 6n-1, 12n+1,$ $3p_1, 3p_2, \ldots, 3p_{k-2}\}$, $T$ = $\{3n-1, 6n+1, 12n-1, 3p_1, 3p_2, \ldots, 3p_{k-2}\}$, $\gcd(p_1,p_2,...,p_{k-2}) = 1$ and $k,n,p_1,p_2,\ldots,p_{k-2}\in\mathbb{N}$. Then, $(i)$  $\theta_{27n,3,n}(C_{27n}(R))$ = $C_{27n}(S)$, $\theta_{27n,3,n}(C_{27n}(S))$ = $C_{27n}(T)$ and $\theta_{27n,3,n}(C_{27n}(T))$ = $C_{27n}(R)$ and $(ii)$ for a given set of values of $k,p_1,p_2,...,p_{k-2}$ and $n$, $C_{27n}(R)$, $C_{27n}(S)$ and $C_{27n}(T)$ are either all Type-2 isomorphic w.r.t. $m$ = $3$ or all Adam's isomorphic.}
\end{theorem}
\begin{proof} The result follows from Theorem \ref{c41}, Remark \ref{r12} and definition of Type-2 isomorphism. 
\end{proof}

 \begin{theorem} \cite{vw2} \label{c43} {\rm For $R_i$ = $\{5, d_i, 25n-d_i, 25n+d_i, 50n-d_i, 50n+d_i\}$, $d_i$ = $5n(i-1)+1$, $i,j$ = 1 to 5 and $n\in\mathbb{N}$, $\theta_{125n,5,jn}(C_{125n}(R_i))$ = $C_{125n}(R_{i+j})$ where $i+j$ in $R_{i+j}$ is calculated under addition modulo 5 and $C_{125n}(R_i)$ are Type-2 isomorphic circulant graphs w.r.t. $r$ = 5.  		\hfill $\Box$}
\end{theorem}

\begin{theorem} \label{c44} {\rm Let $k \geq 3$, $d_i$ = $5n(i-1)+1$, $1 \leq i \leq 5$, $R_i$ = $\{d_i, 25n-d_i, 25n+d_i, 50n-d_i$, $50n+d_i, 5p_1, 5p_2, . . . , 5p_{k-2}\}$,  $k,n,p_1,p_2,...,p_{k-2}\in\mathbb{N}$ and $\gcd(p_1,p_2,...,p_{k-2}) = 1$. Then, for a given set of values of $k,p_1,p_2,...,p_{k-2}$ and $n$, circulant graphs $C_{125n}(R_i)$ are either all Type-2 isomorphic w.r.t. $m$ = $5$ or all Adam's isomorphic, $1 \leq i \leq 5$.}
\end{theorem}
\begin{proof} The result follows from Theorem \ref{c43}, Remark \ref{r12} and definition of Type-2 isomorphism. 
\end{proof}

\begin{theorem} \cite{vw3} \label{c45} {\rm For $R_i$ = $\{7, d_i, 49n-d_i, 49n+d_i, 98n-d_i, 98n+d_i, 147n-d_i, 147n+d_i\}$, $d_i$ = $7n(i-1)+1$, $i,j$ = 1 to 7 and $n\in\mathbb{N}$, $\theta_{343n,7,jn}(C_{343n}(R_i))$ = $C_{343n}(R_{i+j})$ where $i+j$ is calculated under addition modulo 7 and $C_{343n}(R_i)$ are Type-2 isomorphic circulant graphs w.r.t. $r$ = 7. 	\hfill $\Box$}
\end{theorem}

\begin{theorem} \label{c46} {\rm Let $k \geq 3$, $d_i$ = $7n(i-1)+1$, $1 \leq i \leq 7$, $R_i$ = $\{d_i, 49n-d_i, 49n+d_i, 98n-d_i$, $98n+d_i$, $147n-d_i,$ $147n+d_i,$ $7p_1, 7p_2, . . . , 7p_{k-2}\}$, $k,n,p_1,p_2,...,p_{k-2}\in\mathbb{N}$ and $\gcd(p_1,p_2,...,p_{k-2}) = 1$. Then for a given set of values of $k,p_1,p_2,...,p_{k-2}$ and $n$, circulant graphs $C_{343n}(R_i)$ are either all Type-2 isomorphic w.r.t. $m$ = $7$ or all Adam's isomorphic, $1 \leq i \leq 7$.}
\end{theorem}
\begin{proof} The result follows from Theorem \ref{c45}, Remark \ref{r12} and definition of Type-2 isomorphism. 
\end{proof}

\section{On Type-2 isomorphic circulant graphs of order $np^3$ w.r.t. $r$ = $p$}  

Based on our studies in \cite{v20}-\cite{vw3} on Type-2 isomorphic circulant graphs $C_n(R)$ w.r.t. $r$ = 2,3,5,7 that are presented in Sections 5 and 9, we developed computer programs that produced more families of Type-2 isomorphic circulant graphs $C_n(R)$ for different values of $n$, $n\in\mathbb{N}$. From these, we obtain families of Type-2 isomorphic circulant graphs $C_{np^3}(S)$ w.r.t. $r$ = $p$,  Abelian groups related to these isomorphic graphs, develope its theory and present it in this section where $p$ is a prime number and $n\in\mathbb{N}$. Type-2 isomorphic circulant graphs don't have the CI-property.  We use Remark \ref{r11} to establish Type-2 isomorphism w.r.t.  $m$ among isomorphic circulant graphs $C_n(R)$ and $C_n(S)$. 

\begin{theorem} \label{c1} {\rm Let $p$ be an odd prime number, $1 \leq x \leq p-1$, $1 \leq i \leq p$, $y\in\mathbb{N}_0$, $0 \leq y \leq np-1$, $1 \leq x+yp \leq np^2-1$, $d^{np^3, x+yp}_i = (i-1)xpn+x+yp$,  $R^{np^3, x+yp}_i$ $=$ $\{p$, $d^{np^3, x+yp}_i$, $np^2-d^{np^3, x+yp}_i$, $np^2+d^{np^3, x+yp}_i$, $2np^2-d^{np^3, x+yp}_i$, $2np^2+$ $d^{np^3, x+yp}_i,$ $3np^2-d^{np^3, x+yp}_i$, $3np^2+d^{np^3, x+yp}_i,$ . . . , $(p-1)np^2$ - $d^{np^3, x+yp}_i$, $(p-1)np^2+d^{np^3, x+yp}_i,$ $np^3-d^{np^3, x+yp}_i,$ $np^3-p\}$ and $i,j,n,x\in\mathbb{N}$. Then, for a given set of values of $n$, $p$, $x$ and $y$, $\Theta_{np^3,p,jn} (C_{np^3}(R^{np^3, x+yp}_i))$ = $C_{np^3}(R^{np^3, x+yp}_{i+j})$ and the $p$ circulant graphs $C_{np^3}(R^{np^3, x+yp}_i)$ are isomorphic of Type-2 w.r.t.  $p$, $1 \leq i,j \leq p$ where $i+j$ in $R^{np^3, x+yp}_{i+j}$ is calculated under addition modulo $p$.}
\end{theorem}
\begin{proof} We use Remark \ref{r11} to establish Type-2 isomorphism w.r.t.  $r$ between circulant graphs $C_n(R)$ and $C_n(S)$. 

At first, let us prove $\Theta_{np^3,p,jn} (R^{np^3, x+yp}_i)$ = $R^{np^3, x+yp}_{i+j}$ for $1 \leq i,j \leq p$, $1 \leq x \leq p-1$, $y\in\mathbb{N}_0$, $n\in\mathbb{N}$, $0 \leq y \leq np-1$ and $1 \leq x+yp \leq np^2-1$. 

For a given set of values of $n$, $x$ and $y$, we start with proving the above result for $i$ = 1 and 2. 

When $i$ = $1,$ $1 \leq x \leq p-1$, $0 \leq y \leq np-1$, $y\in\mathbb{N}_0$, $n\in\mathbb{N}$ and $1 \leq x+yp \leq np^2-1$,

$d^{np^3, x+yp}_1 = x+yp$ and 

$R^{np^3, x+yp}_1$ = $\{p, x+yp, np^2-x-yp, np^2+x+yp, 2np^2-x-yp,$ 

\hspace{2.5cm} $2np^2+x+yp, 3np^2-x-yp,$ $3np^2+x+yp,$ $\ldots,$  

~\hfill $(p-1)np^2-x-yp,$ $(p-1)np^2+x+yp,$ $np^3-x-yp,$ $np^3-p\}$.

When $i$ = $2,$ $1 \leq x \leq p-1$, $0 \leq y \leq np-1$, $y\in\mathbb{N}_0$, $n\in\mathbb{N}$ and $1 \leq x+yp \leq np^2-1$, 

$d^{np^3, x+yp}_2 = xpn+x+yp$ and 

$R^{np^3, x+yp}_2 = \{p, xpn+x+yp, np^2-(xpn+x+yp), np^2+xpn+x+yp,$  

\hspace{1.5cm} $2np^2-(xpn+x+yp),$ $2np^2+xpn+x+yp,$ $3np^2-(xpn+x+yp),$  

\hspace{1.5cm} $3np^2+xpn+x+yp,$ $\ldots,$ $(p-1)np^2-(xpn+x+yp),$ 

~\hfill $(p-1)np^2+xpn+x+yp,$ $np^3-(xpn+x+yp),$ $np^3-p\}.$ 

For $1 \leq i,j \leq p$, $1 \leq x \leq p-1$, $0 \leq y \leq np-1$, $1 \leq x+yp \leq np^2-1$, $y\in\mathbb{N}_0$ and $n\in\mathbb{N}$, using the definition of $\Theta_{n,r,t},$ we get, 

$\Theta_{np^3,p,n}(R^{np^3, x+yp}_1)$ = $\Theta_{np^3,p,n}(\{p, np^3-p \})$  

\hfill $\bigcup \Theta_{np^3,p,n}(\{x+yp, np^2+x+yp, 2np^2+x+yp, 3np^2+x+yp, \ldots, (p-1)np^2+x+yp \})$ 

\hfill $ \bigcup \Theta_{np^3,p,n}(\{np^2-x-yp, 2np^2-x-yp, 3np^2-x-yp, \ldots, (p-1)np^2-x-yp,  np^3-x-yp \})$

= $\{p, np^3-p \}$ 

\hfill $ \bigcup (xpn+ \{x+yp, np^2+x+yp, 2np^2+x+yp, 3np^2+x+yp, \ldots, (p-1)np^2+x+yp \})$ 

\hfill $\bigcup ((p-x)pn+ \{np^2-x-yp, 2np^2-x-yp, 3np^2-x-yp,$ 

\hfill $\ldots, (p-1)np^2-x-yp, np^3-x-yp \})$ 

 = $\{p, np^3-p, xpn+x+yp, np^2+xpn+x+yp, 2np^2+xpn+x+yp, 3np^2+xpn+x+yp, \ldots,$  

\hfill $  (p-1)np^2+xpn+x+yp, 2np^2-(xpn+x+yp), 3np^2-(xpn+x+yp),  \ldots,$ 

\hfill $  np^3-(xpn+x+yp), np^2-(xpn+x+yp)\} = R^{np^3,x+yp}_2;$  

 $\Theta_{np^3,p,in}(R^{np^3, x+yp}_1)$ = $\Theta_{np^3,p,in}(\{p, np^3-p \})$ $\bigcup \Theta_{np^3,p,in}(\{x+yp, np^2+x+yp, 2np^2+x+yp,$ 

\hfill $3np^2+x+yp, \ldots, (p-1)np^2+x+yp \}) \bigcup \Theta_{np^3,p,in}(\{np^2-x-yp, 2np^2-x-yp,$ 

\hfill $3np^2-x-yp, \ldots, (p-1)np^2-x-yp,  np^3-x-yp \})$

= $\{p, np^3-p \} \bigcup (ixpn+ \{x+yp, np^2+x+yp, 2np^2+x+yp, 3np^2+x+yp, \ldots,$ 

\hfill $ (p-1)np^2+x+yp \}) \bigcup ~(i(p-x)pn+ \{np^2-x-yp, 2np^2-x-yp, 3np^2-x-yp, \ldots, $ 

~\hfill $(p-1)np^2-x-yp, np^3-x-yp \})$ 

 = $\{p, np^3-p, ixpn+x+yp, np^2+ixpn+x+yp, 2np^2+ixpn+x+yp, 3np^2+ixpn+x+yp, \ldots,$ 

~ \hfill $(p-1)np^2+ixpn+x+yp, (1+i)np^2-(ixpn+x+yp),$ $(2+i)np^2-(ixpn+x+yp),$  

~\hfill  $(3+i)np^2-(ixpn+x+yp), \ldots,$ $(p-1+i)np^2-(ixpn+x+yp),$

~\hfill $(p+i)np^2-(ixpn+x+yp)\} = R^{np^3,x+yp}_{i+1}$  \\
since $\{1+i, 2+i, 3+i, . . . , (p-1)+i, p+i\}$ = $\{1, 2, . . . , p-1, p = 0\}$ under addition modulo $p$ and $d^{np^3,x+yp}_{i+1} = ixpn+x+yp$, $1 \leq i \leq p$. Thus, the above result is true for $i$ = 1.

In a similar way, we can prove the above result for $i$ = 2 and also for the general case that $\Theta_{np^3,p,jn}(R^{np^3,x+yp}_i)$ = $R^{np^3,x+yp}_{i+j}$ where $i+j$ in $R^{np^3,x+yp}_{i+j}$ is calculated under addition modulo $p$, $1 \leq x \leq p-1$, $0 \leq y \leq np-1$, $1 \leq x+yp \leq np^2-1$, $y\in\mathbb{N}_0$, $n\in\mathbb{N}$ and $1 \leq i,j \leq p$. And thereby, we get, $\Theta_{np^3,p,jn}(C_{np^3}(R^{np^3,x+yp}_i))$ = $C_{np^3}(R^{np^3,x+yp}_{i+j})$ where $i+j$ in $R^{np^3,x+yp}_{i+j}$ is calculated under addition modulo $p$, $1 \leq x \leq p-1$, $0 \leq y \leq np-1$, $1 \leq x+yp \leq np^2-1$, $y\in\mathbb{N}_0$, $n\in\mathbb{N}$ and $1 \leq i,j \leq p$. 

From the definition of $\Theta_{n,r,t}$ acting on $C_{n}(R)$, we get, for a given set of values of $n$, $x$ and $y$, the circulant graphs $C_{np^3}(R^{np^3, x+yp}_i)$ are isomorphic, $1 \leq x \leq p-1$, $0 \leq y \leq np-1$, $1 \leq x+yp \leq np^2-1$, $y\in\mathbb{N}_0$, $n\in\mathbb{N}$ and $1 \leq i,j \leq p$. To complete the proof, we have to establish their Type-2 isomorphism. 

\vspace{.2cm}
\noindent
{\bf {\it Claim.}} For a given set of values of $n$, $x$ and $y$, $C_{np^3}(R^{np^3, x+yp}_i)$ are Type-2 isomorphic w.r.t.  $p$, $1 \leq i \leq p$, $1 \leq x \leq p-1$, $0 \leq y \leq np-1$, $y\in\mathbb{N}_0$, $n\in\mathbb{N}$ and $1 \leq x+yp \leq np^2-1$.  

 For $1 \leq i,j\leq p$, $1 \leq x \leq p-1$, $1 \leq x+yp \leq np^2-1$, $0 \leq y \leq np-1$, $y\in\mathbb{N}_0$ and $n\in\mathbb{N}$ and given $n$, $x$ and $y$, $d^{np^3,x+yp}_i = (i-1)xpn+x+yp\in$ $R^{np^3,x+yp}_i$. And $d^{np^3,x+yp}_i$ = $d^{np^3,x+yp}_j$ if and only if $(i-1)xpn+x+yp$ = $(j-1)xpn+x+yp$ if and only if $i$  = $j$ if and only if $R^{np^3, x+yp}_i$ = $R^{np^3, x+yp}_j$. Thus, for given $n$, $x$ and $y$ and different $i,$ all the $p$ sets $R^{np^3, x+yp}_i$ are different and thereby all the $p$ circulant graphs $C_{np^3}(R^{np^3, x+yp}_i)$ are also distinct, $1 \leq i,j\leq p$, $1 \leq x \leq p-1$, $0 \leq y \leq np-1$, $y\in\mathbb{N}_0$, $n\in\mathbb{N}$ and $1 \leq x+yp \leq np^2-1$. We have already proved that for given $n$, $x$ and $y$, $\Theta_{np^3,p,jn}(C_{np^3}(R^{np^3, x+yp}_i))$ = $C_{np^3}(R^{np^3, x+yp}_{i+j})$ and thereby all the $p$ circulant graphs $C_{np^3}(R^{np^3, x+yp}_i)$ are isomorphic where $i+j$ in $R^{np^3, x+yp}_{i+j}$ is calculated under addition modulo $p$. This implies that all the $p$ circulant graphs $C_{np^3}(R^{np^3, x+yp}_i)$ are distinct but isomorphic, $1 \leq i\leq p$, $1 \leq x \leq p-1$, $0 \leq y \leq np-1$, $1 \leq x+yp \leq np^2-1$, $y\in\mathbb{N}_0$ and $n\in\mathbb{N}$.

To prove their Type-2 isomorphism w.r.t.  $p$, it is enough to prove, each pair of circulant graphs $C_{np^3}(R^{np^3, x+yp}_i)$ and $C_{np^3}(R^{np^3, x+yp}_j)$ for $i \neq j$ are not of Type-1 for a fixed $n$, $x$ and $y$ where $1 \leq i,j\leq p$, $1 \leq x \leq p-1$, $0 \leq y \leq np-1$, $1 \leq x+yp \leq np^2-1$, $y\in\mathbb{N}_0$ and $n\in\mathbb{N}$. Let us start with the circulant graph $C_{np^3}(R^{np^3, x+yp}_1)$. 

\vspace{.2cm}
\noindent
{\bf{\it Sub-claim.}} $C_{np^3}(R^{np^3, x+yp}_1)$ and $C_{np^3}(R^{np^3, x+yp}_i)$ are Type-2 isomorphic w.r.t.  $p$ for given $n$, $x$ and $y$ and $2 \leq i \leq p$.

\vspace{.2cm}
If not, they are of Adam's isomorphic. This implies, there exists $s,q \in \mathbb{N}$ such that $C_{np^3}(sR^{np^3, x+yp}_1)$ = $C_{np^3}(R^{np^3, x+yp}_i)$ where $s$ = $qpn-j,$ $\gcd(np^3, s)$ = $1$, $1 \leq j,x \leq p-1,$ $1 \leq qpn-j \leq np^3-1$, $n\in\mathbb{N}$, $2 \leq i \leq p$, $0 \leq y \leq np-1$, $1 \leq x+yp \leq np^2-1$ and $y\in\mathbb{N}_0$. This also implies, $1 \leq j \leq p-1$ and $\gcd(qn, j)$ = 1. Now, consider the case when $j$ = $1.$ In this case, $s$ = $qpn-1,$ $\gcd(np^3, qpn-1)$ = $1,$ $\gcd(qn, 1)$ = 1, $1 \leq qpn-1 \leq np^3-1,$ $C_{np^3}((qpn-1)R^{np^3,x+yp}_1) \cong C_{np^3}(R^{np^3,x+yp}_i),$ $2 \leq i \leq p$, $1 \leq x \leq p-1$, $0 \leq y \leq np-1$,  $1 \leq x+yp \leq np^2-1$, $y\in\mathbb{N}_0$, $n\in\mathbb{N}$. This implies, $(qpn-1)\{p,$ $x+yp,$ $np^2-x-yp$, $np^2+x+yp,$ $2np^2-x-yp,$ $2np^2+x+yp,$ $3np^2-x-yp,$ $3np^2+x+yp,$ $\ldots,$ $(p-1)np^2-x-yp,$ $(p-1)np^2+x+yp,$ $np^3-x-yp,$ $np^3-p\}$ = $\{p,$ $(i-1)xpn+x+yp,$ $np^2-((i-1)xpn+x+yp),$ $np^2+(i-1)xpn+x+yp,$ $2np^2-((i-1)xpn+x+yp),$ $2np^2+(i-1)xpn+x+yp,$ $3np^2-((i-1)xpn+x+yp),$ $3np^2+(i-1)xpn+x+yp,$ $\ldots,$ $(p-1)np^2-((i-1)xpn+x+yp),$ $(p-1)np^2+(i-1)xpn+x+yp,$ $np^3-((i-1)xpn+x+yp),$ $np^3-p\}$ under arithmetic modulo $np^3,$ $2 \leq i \leq p$, $1 \leq x \leq p-1$, $y\in\mathbb{N}_0$, $n\in\mathbb{N}$, $0 \leq y \leq np-1$, $1 \leq x+yp \leq np^2-1$. This implies, $(qpn-1)p,$ $(qpn-1)(np^3-p),$ $p+p_1np^3$ and $np^3-p+p_2np^3$ are the only numbers, each is a multiple of $p,$ in the two sets for some $p_1,p_2\in \mathbb{N}_0,$ $2 \leq i \leq p$, $1 \leq x \leq p-1$, $1 \leq qpn-1 \leq np^3-1$, $\gcd(np^3, qpn-1)$ = $1$, $0 \leq y \leq np-1$, $1 \leq x+yp \leq np^2-1$, $y\in \mathbb{N}_0$ and $n,q\in\mathbb{N}$. Under this, the following two cases arise.

\vspace{.2cm}
\noindent
{\bf {\it Case 1.}} $(qpn-1)p = p+p_1np^3$, $p_1 \in \mathbb{N}_0,$~ $n,q \in \mathbb{N},$ $1 \leq qpn-1 \leq np^3-1.$ 

In this case, the possible values of $p_1$ are $~0,~1,~2,~\ldots,~p-1~$ since $1 \leq qpn-1 \leq np^3-1$ and $q \in \mathbb{N}.$ When $p_1$ = $0$, $qpn-1$ = $1$ and $\gcd(qn, j)$ = 1; $p_1 = 1$, $qpn-1$ = $np^2+1$ and $\gcd(qn, j)$ = 1; $p_1$ = $2$, $qpn-1$ = $2np^2+1$ and $\gcd(qn, j)$ = 1; $\ldots;$ $p_1$ = $p-1$, $qpn-1$ = $(p-1)np^2+1$ and $\gcd(qn, j)$ = 1. Now, let us calculate $(qpn-1)R^{np^3,x+yp}_1$ for $qpn-1$ = $np^2+1,$ $2np^2+1,$ $\ldots,$ $(p-1)np^2+1$ under arithmetic modulo $np^3.$ 

When $qpn-1$ = $np^2+1,$ for given $n$, $x$ and $y$, $1 \leq x \leq p-1$, $y\in\mathbb{N}_0$, $0 \leq y \leq np-1$, $1 \leq x+yp \leq np^2-1$ and $n,x\in\mathbb{N}$, under arithmetic modulo $np^3,$ 

\vspace{.2cm}
\noindent
$(qpn-1)R^{np^3,x+yp}_1$ = $(np^2+1)R^{np^3,x+yp}_1$ 

\hspace{2.5cm} = $(np^2+1)\{p, x+yp, np^2-x-yp, np^2+x+yp, 2np^2-x-yp, 2np^2+x+yp, $  

 \hfill $3np^2-x-yp, 3np^2+x+yp, \ldots, $ $(p-1)np^2-x-yp, (p-1)np^2+x+yp, np^3-x-yp, np^3-p\}$ 

 = $\{p, xnp^2+x+yp, (p-x+1)np^2-x-yp, (x+1)np^2+x+yp, (p-x+2)np^2-x-yp$, 

\hfill   $(p-x+2)np^2-x-yp, (x+2)np^2+x+yp, (p-x+3)np^2-x-yp, (x+3)np^2+x+yp, \ldots,$ 

~\hfill  $(p-x+p-1)np^2-x-yp, (x-1)np^2+x+yp, (p-x)np^2-x-yp, np^3-p\} = R^{np^3,x+yp}_1.$ 

\vspace{.2cm}

Here, $\{xnp^2+x+yp$, $(x+1)np^2+x+yp,$ $(x+2)np^2+x+yp,$ $(x+3)np^2+x+yp,$ $\ldots,$ $(x-1)np^2+x+yp$ = $(x+(p-1))np^2+x+yp\}$ =  $\{x+yp,$ $np^2+x+yp,$ $2np^2+x+yp,$ $3np^2+x+yp,$ $\ldots,$ $(p-1)np^2+x+yp\}$ and $\{(p-x+1)np^2-x-yp,$ $(p-x+2)np^2-x-yp,$ $(p-x+3)np^2-x-yp,$ . . . , $(p-x+p-1)np^2-x-yp)$, $(p-x+p)np^2-x-yp)\}$ = $\{np^2-x-yp$, $2np^2-x-yp,$ $3np^2-x-yp,$ $\ldots,$ $(p-1)np^2-x-yp,$ $np^3-x-yp\}$ under arithmetic modulo $np^3$, $1 \leq x \leq p-1$, $y\in\mathbb{N}_0$, $1 \leq x+yp \leq np^2-1$, $0 \leq y \leq np-1$ and $n,x\in\mathbb{N}$.

Similarly, we can prove that $(qpn-1)R^{np^3,x+yp}_1$ = $R^{np^3,x+yp}_1$ when $qpn-1$ = $2np^2+1,$ $3np^2+1,$ $\ldots,$ $(p-1)np^2+1$ under arithmetic modulo $np^3.$ This implies, $C_{np^3}((qpn-1)R^{np^3,x+yp}_1)$ = $C_{np^3}(R^{np^3,x+yp}_1)$ and $\neq$ $C_{np^3}(R^{np^3,x+yp}_i)$ for $2 \leq i \leq p-1$, $qpn-1$ = $np^2+1,$ $2np^2+1,$ $\ldots,$ $(p-1)np^2+1$, $1 \leq x \leq p-1$, $y\in\mathbb{N}_0$, $1 \leq x+yp \leq np^2-1$, $0 \leq y \leq np-1$ and $n,x\in\mathbb{N}$. 

In a similar way, we can prove that for $2 \leq j \leq p-1$ and $\gcd(qpn-j, np^3)$ = 1, $(qpn-j)$ $R^{np^3,x+yp}_1$ = $R^{np^3,x+yp}_1$ and $\neq$ $R^{np^3,x+yp}_i$ for $2 \leq i \leq p-1$, $y\in\mathbb{N}_0$, $1 \leq x \leq p-1$, $0 \leq y \leq np-1$, $1 \leq x+yp \leq np^2-1$, $n,x\in\mathbb{N}$ and when $nqp-j$ = $np^2+1,$ $2np^2+1,$ $\ldots,$ $(p-1)np^2+1$. This implies, for $2 \leq i \leq p-1$ and given $n$, $x$ and $y$, $C_{np^3}(R^{np^3, x+yp}_1)$ and $C_{np^3}(R^{np^3, x+yp}_i)$ are not Adam's isomorphic, $1 \leq x \leq p-1$, $1 \leq x+yp \leq np^2-1$, $0 \leq y \leq np-1$, $y\in\mathbb{N}_0$ and $n,x\in\mathbb{N}$. Hence the sub-claim is true in this case. 

\vspace{.2cm}
\noindent
{\bf {\it Case 2.}} $(qpn-1)p = np^3-p+p_2np^3$, $p_2 \in \mathbb{N}_0$, $1 \leq qpn-1 \leq np^3-1$ and $n,q \in \mathbb{N}$. 

In this case, the possible values of $p_2$ are $0,~1,~2,~.~.~.~,~p-1$ since ~ $1 \leq qpn-1 \leq np^3-1$ and $n,q \in \mathbb{N}.$ When $p_2$ = $0$, $qpn-1$ = $np^2-1$; $p_2$ = $1$, $qpn-1$ = $2np^2-1$; $\ldots;$ $p_2$ = $p-1$, $qpn-1$ = $np^3-1.$ Now, let us calculate $(qpn-1)R^{np^3,x+yp}_1$ for $qpn-1$ = $np^2-1,$ $2np^2-1,$ $\ldots,$ $np^3-1$ under arithmetic modulo $np^3.$ 

When $nqp-1$ = $np^2-1$, for given $n$, $x$ and $y$, $1 \leq x \leq p-1$, $y\in\mathbb{N}_0$, $0 \leq y \leq np-1$, $1 \leq x+yp \leq np^2-1$ and $n,x\in\mathbb{N}$, under arithmetic modulo $np^3,$ 

\vspace{.2cm}
\noindent
$(qpn-1)R^{np^3,x+yp}_1 = (np^2-1)R^{np^3,x+yp}_1 $ 

~\hfill = $(np^2-1)\{p, x+yp, np^2-x-yp, np^2+x+yp, 2np^2-x-yp,2np^2+x+yp,  3np^2-x-yp, $  

 \hfill $3np^2+x+yp, \ldots, (p-1)np^2-x-yp, (p-1)np^2+x+yp, np^3-x-yp, np^3-p\}$ 

   = $\{np^3-p, xnp^2-x-yp, (p-(x+1))np^2+x+yp,$  $(x-1)np^2-x-yp, (p-(x+2))np^2+x+yp,$ 

\hspace{.5 cm} $(x-2)np^2-x-yp, (p-(x+3))np^2+x+yp,$  $(x-3)np^2-x-yp, \ldots, (p-(x-1))np^2+x+yp,$ 

~\hfill  $(p+x-(p-1))np^2-x-yp, (p-x)np^2+x+yp, p\} = R^{np^3,x+yp}_1.$ 

\vspace{.2cm}
Here, $\{xnp^2-x-yp,$ $(x-1)np^2-x-yp,$ $(x-2)np^2-x-yp,$ $(x-3)np^2-x-yp,$ ..., $(p+x-(p-1))np^2$ $-x-yp\}$ = $\{np^2-x-yp$, $2np^2-x-yp,$ $3np^2-x-yp,$ $\ldots,$ $(p-1)np^2-x-yp,$ $np^3-x-yp\}$ and $\{(p-(x+1))np^2+x+yp,$ $(p-(x+2))np^2+x+yp,$ $(p-(x+3))np^2+x+yp,$ $\ldots$, $(p-(x-1))np^2+x+yp),$  $(p-x)np^2+x+yp\}$ = $\{x+yp,$ $np^2+x+yp,$ $2np^2+x=yp,$ $3np^2+x+yp,$ $\ldots,$ $(p-1)np^2+x+yp\}$ under arithmetic modulo $np^3$, $1 \leq x \leq p-1$, $y\in\mathbb{N}_0$, $0 \leq y \leq np-1$, $1 \leq x+yp \leq np^2-1$ and $n,x\in\mathbb{N}$.

Similarly, we can prove that $(qpn-1)R^{np^3,x+yp}_1$ = $R^{np^3,x+yp}_1$ when $qpn-1$ = $2np^2-1,$ $3np^2-1,$ $\ldots,$ $np^3-1,$ under arithmetic modulo $np^3.$ This implies that $C_{np^3}((qpn-1)R^{np^3,x+yp}_1)$ = $C_{np^3}(R^{np^3,x+yp}_1)$ when $qpn-1$ = $np^2-1,$ $2np^2-1,$ $\ldots,$ $np^3-1.$ In a similar way, we can prove that $(qpn-j)R^{np^3,x+yp}_1$ = $R^{np^3,x+yp}_1,$ under arithmetic modulo $np^3$, when $qpn-j$ = $np^2-1,$ $2np^2-1,$ $\ldots,$ $np^3-1$ for $2 \leq j \leq p-1$ and $\gcd(qpn-j, np^3)$ = 1. This implies, for every $s = qpn-j$ which is relative prime to $np^3$ and for given $n$, $x$ and $y$, $C_{np^3}(sR^{np^3,x+yp}_1)$ = $C_{np^3}(R^{np^3,x+yp}_1)$ and $\neq$ $C_{np^3}(R^{np^3,x+yp}_i)$ for $2 \leq i \leq p-1$, $1 \leq j,x \leq p-1$,  $y\in\mathbb{N}_0$, $0 \leq y \leq np-1$, $1 \leq x+yp \leq np^2-1$ and for $nqp-j$ = $np^2-1,$ $2np^2-1,$ $\ldots,$ $(p-1)np^2-1$, $1 \leq qpn-j \leq np^3-1$ and $n,q,x\in\mathbb{N}$. Hence the sub-claim is also true in this case. 

This implies, for given $n$, $x$ and $y$ and for $i = 2,3,\ldots,p-1$, $C_{np^3}(R^{np^3,x+yp}_1)$ and $C_{np^3}(R^{np^3,x+yp}_i)$ can not be Type-1 isomorphic and so they are Type-2 isomorphic w.r.t.  $r$ = $p$ since $\Theta_{np^3,p,jn} (C_{np^3}(R^{np^3, x+yp}_i))$ = $C_{np^3}(R^{np^3, x+yp}_{i+j})$ for $1 \leq i,j \leq p$, $1 \leq x \leq p-1$,  $y\in\mathbb{N}_0$, $n,x\in\mathbb{N}$, $0 \leq y \leq np-1$, $1 \leq x+yp \leq np^2-1$ and $i+j$ in $R^{np^3, x+yp}_{i+j}$ is calculated under arithmetic modulo p. Hence the claim is true.

Similarly, we can show that for given $n$, $x$ and $y$, $1 \leq i,j \leq p$ and $i \neq$ $j$, $C_{np^3}(R^{np^3,x+yp}_i)$ and $C_{np^3}(R^{np^3,x+yp}_j)$ can not be Adam's isomorphic by proving $C_{np^3}((qpn-t)R^{np^3,x+yp}_i)$ = $C_{np^3}(R^{np^3,x+yp}_i)$ when $qpn-t$ = $np^2-1,$ $2np^2-1,$ $\ldots,$ $np^3-1$ as well as $qpn-t$ = $np^2+1,$ $2np^2+1,$ $\ldots,$ $(p-1)np^2+1$ where $\gcd(np^3, qpn-t)$ = 1, $1 \leq qpn-t \leq np^3-1$, $1 \leq t,x \leq p-1$, $0 \leq y \leq np-1$, $1 \leq x+yp \leq np^2-1$, $y\in\mathbb{N}_0$ and $n,q,t,x\in\mathbb{N}$.   

This implies, for given $n$, $x$ and $y$, all the $p$ isomorphic circulant graphs $C_{np^3}(R^{np^3,x+yp}_i)$ are Type-2 isomorphic w.r.t.  $r = p$ for $i$ = 1 to $p$, $y\in\mathbb{N}_0$, $1 \leq x \leq p-1$, $0 \leq y \leq np-1$, $1 \leq x+yp \leq np^2-1$ and $n,x\in\mathbb{N}$. Hence, we get the result.  
\end{proof}

\begin{theorem} \label{c2} {\rm Let $p$ be an odd prime number, $k \geq 3$, $1 \leq i,j \leq p$, $1 \leq x \leq p-1$, $y\in\mathbb{N}_0$, $0 \leq y \leq np - 1$, $1 \leq x+yp \leq np^2-1$, $d^{np^3, x+yp}_i$ = $(i-1)xpn+$ $x+yp$, $R^{np^3, x+yp}_i$ = $\{d^{np^3, x+yp}_i,$ $np^2-d^{np^3, x+yp}_i,$ $np^2+d^{np^3, x+yp}_i,$ $2np^2-$ $d^{np^3, x+yp}_i,$ $2np^2+d^{np^3, x+yp}_i,$ $3np^2-d^{np^3, x+yp}_i,$ $3np^2+d^{np^3, x+yp}_i,$ ..., $(p-1)np^2-d^{np^3, x+yp}_i,$ $(p-1)np^2+d^{np^3, x+yp}_i,$ $np^3-d^{np^3, x+yp}_i,$ $pp_1,$ $pp_2,$ $\ldots$, $pp_{k-2},$ $p(np^3-p_{k-2})$, $p(np^3-p_{k-3})$, . . . , $p(np^3-p_1)\}$, $\gcd(p_1,p_2,...,p_{k-2}) = 1$ and $i,j,k,n,x,p_1,p_2,...,p_{k-2} \in \mathbb{N}.$  Then, for a given set of values of $k, p, x, y, p_1, p_2, ... , p_{k-3},p_{k-2}$ and $n$ and for $i = 1,2,...,p$, $(i)$ $\Theta_{np^3,p,jn}(C_{np^3}(R^{np^3, x+yp}_i))$ = $C_{np^3}(R^{np^3, x+yp}_{i+j})$ and $(ii)$ the $p$ circulant graphs $C_{np^3}(R^{np^3, x+yp}_i)$ are either all Type-1 or all Type-2 (and without CI-property) w.r.t.  $r$ = $p$ where $i+j$ in $R^{np^3, x+yp}_{i+j}$ is calculated under addition modulo $p$.}
\end{theorem}
\begin{proof}\quad Using Theorem \ref{c1}, for given $n$, $x$ and $y$, the $p$ circulant graphs $C_{np^3}(R^{np^3, x+yp}_i)$ for $i$ = $1$ to $p$ are Type-2 isomorphic w.r.t.  $r = p$ where $p$ is an odd prime, $d^{np^3, x+yp}_i = (i-1)xpn+x+yp$ and $R^{np^3, x+yp}_i$ = $\{p$, $d^{np^3, x+yp}_i$, $np^2-d^{np^3, x+yp}_i$, $np^2+d^{np^3, x+yp}_i$, $2np^2-d^{np^3, x+yp}_i$, $2np^2+$ $d^{np^3, x+yp}_i,$ $3np^2-d^{np^3, x+yp}_i$, $3np^2+d^{np^3, x+yp}_i,$ . . . , $(p-1)np^2-d^{np^3, x+yp}_i$, $(p-1)np^2+d^{np^3, x+yp}_i,$ $np^3-d^{np^3, x+yp}_i,$ $np^3-p\}$. Then the result follows from Remark \ref{r12}, definition of Type-2 isomorphic circulant graphs and the property that Type-2 isomorphic circulant graphs have the property that they are graphs without $CI$-property. 
\end{proof}

Our next theorem is on $T2_{np^3,p}(C_{np^3}(R))$ and before that, let us consider a related problem. 

\begin{prm} \label{c3} {\rm Show that circulant graphs $C_{81}(R_i)$ are Type-2 isomorphic w.r.t. $r$ = 3 for $i$ = 1,2,3 where $R_1$ = $\{3,7,20,34,47,61,74,78\}$, $R_2$ = $\{3,11,16,38,43,65,70,78\}$ and $R_3$ = $\{2,3,25,29,52,56,78,79\}$. Also, find $T2_{81,3}(C_{81}(R_i))$, $1 \leq i \leq 3$.} 
\end{prm}
\noindent
{\bf Solution.} All the three sets $R_1$, $R_2$ and $R_3$ are symmetric subsets of $\mathbb{Z}_{81}$, each contains 3 and 78 as common elements and $\gcd(81, 3)$ = 3 = $\gcd(81, 78)$ and hence $C_{81}(R_i)$ may be of Type-2 isomorphic w.r.t.  $r$ = 3, $i$ = 1,2,3. Here, 81 = $3\times 3^3$ = $np^3$ which implies, $n = 3$, $p = 3$, $np^2 = 27$ and $m = \gcd(np^3, 3)$ = 3. 

The minimum jump size which is not a multiple of 3 in $R_1$ is 7 which implies that we can take $7 = x+yp$. In this case, $x = 1$, $y = 2$ and so $d^{81,7}_1$ = $x+yp$ = 7, $d^{81,7}_2$ = $npx+x+yp$ = 9+7 = 16 and $d^{81,7}_3$ = $2npx+x+yp$ = 18+7 = 25. Using Theorem \ref{c2}, we get,

\vspace{.1cm}
$R^{81,7}_1$ = $\{3,7,20,34,47,61,74,78\}$ = $R_1$ = $\Theta_{81,3,0}(R^{81,7}_1)$,  

\vspace{.1cm}
$R^{81,7}_2$ = $\{3,11,16,38,43,65,70,78\}$ = $R_2$ = $\Theta_{81,3,1}(R^{81,7}_1)$ and

\vspace{.1cm}  
$R^{81,7}_3$ = $\{2,3,25,29,52,56,78,79\}$ = $R_3$ = $\Theta_{81,3,2}(R^{81,7}_1)$. 

\vspace{.1cm}
$\Rightarrow$ $C_{81}(R_i)$ are isomorphic, $1 \leq i \leq 3$.

Moreover, 2 is the only relative prime to 3 between 1 and 3 and

$2R^{81,7}_1$ = $2 \times \{3,7,20,34,47,61,74,78\}$ = $\{6,14,40,68,13,41,67,75\}$ $\neq$ $R_2, R_3$;

$2R^{81,7}_2$ = $2 \times \{3,11,16,38,43,65,70,78\}$ = $\{6,22,32,76,45,49,59,75\}$ $\neq$ $R_1, R_3$ and 

$2R^{81,7}_3$ = $2 \times \{2,3,25,29,52,56,78,79\}$ = $\{4,6,50,58,23,31,75,77\}$ $\neq$ $R_1, R_2$. 

\vspace{.1cm}
$\Rightarrow$ $C_{81}(R_i)$ are not Adam's isomorphic, $1 \leq i \leq 3$.

\vspace{.1cm}
$\Rightarrow$ $C_{81}(R_i)$ are Type-2 isomorphic w.r.t. $r = 3$, $1 \leq i \leq 3$ and

\vspace{.1cm}
$T2_{81,3}(C_{81}(R_1))$ = $\{C_{81}(R_1), C_{81}(R_2), C_{81}(R_3)\}$ = $T2_{81,3}(C_{81}(R_2))$ = $T2_{81,3}(C_{81}(R_3))$. 

In Figures 27, 28, 29, isomorphic circulant graphs, $\Theta_{81,3,0}(C_{81}(R_1))$ = $C_{81}(R_1)$, $\Theta_{81,3,3}(C_{81}(R_1))$ = $C_{81}(R_2)$ and $\Theta_{81,3,6}(C_{81}(R_1))$ = $C_{81}(R_3)$ of Type-2 w.r.t. $r$ = 3 are given. \hfill $\Box$
\begin{figure}[ht]
	\centerline{\includegraphics[width=6.2in]{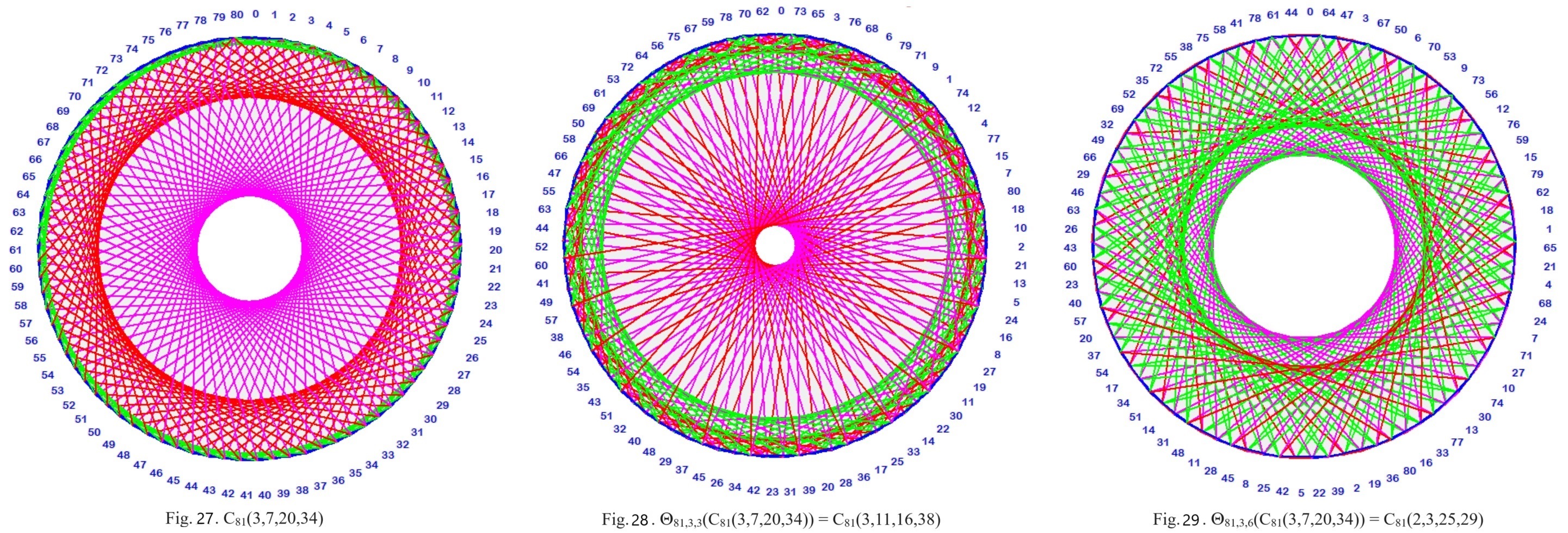}}
\end{figure}
\begin{theorem} \label{c4} {\rm Let $p$ be an odd prime number, $1 \leq i \leq p$, $1 \leq x \leq p-1$, $y\in\mathbb{N}_0$, $0 \leq y \leq np - 1$, $1 \leq x+yp \leq np^2-1$, $d^{np^3, x+yp}_i = (i-1)xpn+$ $x+yp$ and $R^{np^3, x+yp}_i$ = $\{p,$ $d^{np^3,x+yp}_i,$ $np^2-d^{np^3,x+yp}_i,$ $np^2+d^{np^3, x+yp}_i,$ $2np^2-$ $d^{np^3, x+yp}_i,$ $2np^2+d^{np^3, x+yp}_i,$ $3np^2-d^{np^3, x+yp}_i,$ $3np^2+d^{np^3, x+yp}_i,$ . . . , $(p-1)np^2$ - $d^{np^3, x+yp}_i,$ $(p-1)np^2+d^{np^3, x+yp}_i,$ $np^3-d^{np^3, x+yp}_i,$ $np^3-p\}$. Then, for $i$ = 1 to $p$, $T2_{np^3, p}(C_{np^3}(R^{np^3, x+yp}_i))$ = $\{\Theta_{np^3,p,jn}(C_{np^3}(R^{np^3,x+yp}_i)) = C_{np^3}(R^{np^3,x+yp}_{i+j}) : j = 0,1,...,p-1$ where $i+j$ in $C_{np^3}(R^{np^3,x+yp}_{i+j})$ is calculated under addition modulo $p \}$ and $(T2_{np^3, p}(C_{np^3}(R^{np^3, x+yp}_i)), \circ)$ is a Type-2 group of order $p$.}
\end{theorem}
\begin{proof}\quad  Using Theorem \ref{c1}, $\Theta_{np^3,p,jn}(C_{np^3}(R^{np^3,x+yp}_i))$ = $C_{np^3}(R^{np^3,x+yp}_{i+j})$ and for $i$ = 1 to $p$, the $p$ circulant graphs $C_{np^3}(R^{np^3,x+yp}_i)$ are isomorphic of Type-2 w.r.t. $p$ when $p$ is an odd prime number, $1 \leq i \leq p$, $1 \leq x \leq p-1$, $y\in\mathbb{N}_0$, $0 \leq y \leq np - 1$, $1 \leq x+yp \leq np^2-1$, $d^{np^3, x+yp}_i$ = $(i-1)xpn+$ $x+yp$, $R^{np^3, x+yp}_i$ = $\{p,$ $d^{np^3,x+yp}_i,$ $np^2-d^{np^3,x+yp}_i,$ $np^2+d^{np^3, x+yp}_i,$ $2np^2-$ $d^{np^3, x+yp}_i,$ $2np^2+d^{np^3, x+yp}_i,$ $3np^2-d^{np^3, x+yp}_i,$ $3np^2+d^{np^3, x+yp}_i,$ ..., $(p-1)np^2$ $-d^{np^3, x+yp}_i,$ $(p-1)np^2+d^{np^3, x+yp}_i,$ $np^3-d^{np^3, x+yp}_i,$ $np^3-p\}$ and $i+j$ in $R^{np^3, x+yp}_{i+j}$ is calculated under addition modulo $p$.

\begin {thebibliography}{10}

\bibitem {ad67}  
A. Adam, 
{\it Research problem 2-10},  
J. Combinatorial Theory, {\bf 3} (1967), 393.

\bibitem {amv} 
B. Alspach, J. Morris and V. Vilfred, 
{\it Self-complementary circulant graphs}, 
Ars Com., {\bf 53} (1999), 187-191.

\bibitem {bt} 
F. T. Boesch and R. Tindell, 
{\it Circulant and their connectivities},
J. Graph Theory, {\bf 8} (1984), 487-499.

\bibitem {da79}	
P. J. Davis, 
{\it Circulant Matrices,} 
Wiley, New York, 1979.

\bibitem {eltu} 
B. Elspas and J. Turner, 
{\it Graphs with circulant adjacency matrices}, 
J. Combinatorial Theory, {\bf 9} (1970), 297-307.

\bibitem {frs} 
D. Fronček, A. Rosa, and J. Širáň, 
{\it The existence of selfcomplementary circulant graphs}, European J. Combin., {\bf 17} (1996), 625--628.

\bibitem {ha69} 
F. Harary, 
{\it Graph Theory}, 
Addison-Wesley, 1969.

\bibitem {hz14}
Houqing Zhou,
{\it The Wiener Index of Circulant Graphs},
J. of Chemistry, Hindawi Pub. Cor. (2014), Article ID 742121. doi.org/10.1155/2014/742121.

\bibitem {h03}
F. K. Hwang,
{\it A survey on multi-loop networks},
Theoret. Comput. Sci., {\bf 299} (2003), 107-121.

\bibitem {krsi} 
I. Kra and S. R. Simanca, 
{\it On Circulant Matrices},  
AMS Notices, {\bf 59} (2012), 368--377.

\bibitem {li02} 
C. H. Li, 
{\it On isomorphisms of finite Cayley graphs - a survey}, 
Discrete Math., {\bf 256} (2002), 301--334.

\bibitem {mu04} 
M. Muzychuk, 
{\it A solution of the isomorphism problem for circulant graphs}, Proc. London Math. Soc., \textbf{88} (2004), 1--41.

\bibitem {mu97} 
M. Muzychuk, 
{\it On Adam's Conjecture for circulant graphs}, 
Discrete Math., \textbf{176} (1997), 285--298.

\bibitem {st02}
V. N. Sachkov and V. E. Tarakanov,
{\it Combinatorics of Nonnegative Matrices},
Translations of Math. Monographs, {\bf 213}, AM Society, Providence (2002).

\bibitem {sa62} H. Sachs, \emph{Uber selbstkomplementare Graphen}, Publ. Math. Debrecen, {\bf 9} (1962), 270--288.

\bibitem {to77} 
S. Toida, 
{\it A note on Adam's conjecture}, 
J. Combin. Theory Ser. B, \textbf{23} (1977), 239--246.

\bibitem {v96} 
V. Vilfred, 
{\it $\sum$-labelled Graphs and Circulant Graphs}, 
Ph.D. Thesis, University of Kerala, Thiruvananthapuram, Kerala, India (1996). 

\bibitem {v17} 
V. Vilfred, 
{\it A Study on Isomorphic Properties of Circulant Graphs:~ Self-complimentary, isomorphism, Cartesian product and factorization},  
Advances in Science, Technology and Engineering Systems (ASTES) Journal, \textbf{2 (6)} (2017), 236--241. DOI: 10.25046/ aj020628. ISSN: 2415-6698.

\bibitem {v13} 
V. Vilfred, 
{\it A Theory of Cartesian Product and Factorization of Circulant Graphs},  
Hindawi Pub. Corp. - J. Discrete Math.,  \textbf{Vol. 2013}, Article~ ID~ 163740, 10 pages.

\bibitem {vc13} 
V. Vilfred, 
{\it New Abelian Groups from Isomorphism of Circulant Graphs}, 
Proce. of Inter. Conf. on Applied Math. and Theoretical Computer Science, St. Xavier's Catholic Engineering College, Nagercoil, Tamil Nadu, India (2013), xiii--xvi.~ISBN~ 978 -93-82338 -30-7. 
 
\bibitem {v03} 
V. Vilfred, 
\emph{On Circulant Graphs, in Graph Theory and its Applications}, 
Narosa Publ., New Delhi, India, 2003, 34--36. ISBN 81-7319-569-2.

\bibitem {v20} 
V. Vilfred Kamalappan, 
\emph{ New Families of Circulant Graphs Without Cayley Isomorphism Property with $r_i = 2$},
Int. Journal of Applied and Computational Mathematics, (2020) 6:90, 34 pages. https://doi.org/10.1007/s40819-020-00835-0. Published online: 28.07.2020 by Springer.

\bibitem {vw1} 
V. Vilfred and P. Wilson, 
\emph{Families of Circulant Graphs without Cayley Isomorphism Property with $m_i = 3$}, 
IOSR Journal of Mathematics, \textbf{15 (2)} (2019), 24--31. DOI: 10.9790/5728-1502022431. ISSN: 2278-5728, 2319-765X. 
 
\bibitem {vw2} 
V. Vilfred and P. Wilson, 
\emph{New Family of Circulant Graphs without Cayley Isomorphism Property with $m_i = 5,$} 
Int. Journal of Scientific and Innovative Mathematical Research, \textbf{3 (6)} (2015), 39--47.

\bibitem {vw3} 
V. Vilfred and P. Wilson, 
\emph{New Family of Circulant Graphs without Cayley Isomorphism Property with $m_i = 7,$} 
IOSR Journal of Mathematics, \textbf{12} (2016), 32--37.

\bibitem {vw0A} 
V. Vilfred and P. Wilson, 
\emph{Type-2 Isomorphic Circulant Graphs of Order $np^3$ w.r.t. $r$ = $p$ and Their Abelian Groups},
arXiv: 2012.11372v10 [math.CO] (31 January 2024), 227 pages.
  
\end{thebibliography}
\end{document}